\documentclass[a4paper, reqno]{amsart}
\usepackage[utf8]{inputenc}
\usepackage{graphicx}
\usepackage{amsmath}
\usepackage{amssymb,amsthm}
\usepackage{mathtools}
\usepackage{etoolbox}
\usepackage[T1]{fontenc}
\usepackage{pgfplots}
\usepackage{hyperref,enumerate}
\usepackage{enumitem}
\usepackage{cite}

\setlist[enumerate]{label=\alph*),leftmargin=2\parindent}

\theoremstyle{plain}
\newtheorem{theorem}{Theorem}
\numberwithin{theorem}{section}
\newtheorem{prop}[theorem]{Proposition}
\newtheorem{lemma}[theorem]{Lemma}
\newtheorem{cor}[theorem]{Corollary}

\theoremstyle{definition}

\newtheorem{rem}[theorem]{Remark}

\allowdisplaybreaks

\newcommand{\N}{\mathbb{N}}
\newcommand{\Z}{\mathbb{Z}}

\newcommand{\R}{\mathbb{R}}

\newcommand{\RT}{\mathbb{R} \times \mathbb{T}}

\renewcommand{\phi}{\varphi}
\renewcommand{\epsilon}{\varepsilon}

\DeclarePairedDelimiterX\norm[1]\lVert\rVert{\ifblank{#1}{\cdot}{#1}}
\DeclarePairedDelimiterX\abs[1]\lvert\rvert{\ifblank{#1}{\cdot}{#1}}

\DeclarePairedDelimiter\set\{\}
\DeclarePairedDelimiter\br()
\DeclarePairedDelimiterX\intcc[1][]{\ifblank{#1}{0,1}{#1}}
\DeclarePairedDelimiterX\intco[1][){\ifblank{#1}{0,\infty}{#1}}
\DeclarePairedDelimiterX\intoc[1](]{\ifblank{#1}{-\infty,0}{#1}}
\DeclarePairedDelimiterX\intoo[1](){\ifblank{#1}{0,\infty}{#1}}

\newcommand{\absBig}[1]{\abs[\Big]{#1}}
\newcommand{\absbigg}[1]{\abs[\bigg]{#1}}
\newcommand{\absBigg}[1]{\abs[\Bigg]{#1}}

\newcommand{\setbig}[1]{\set[\big]{#1}}
\newcommand{\setBig}[1]{\set[\Big]{#1}}
\newcommand{\setbigg}[1]{\set[\bigg]{#1}}

\newcommand{\brbig}[1]{\br[\big]{#1}}
\newcommand{\brBig}[1]{\br[\Big]{#1}}
\newcommand{\brbigg}[1]{\br[\bigg]{#1}}
\newcommand{\brBigg}[1]{\br[\Bigg]{#1}}

\newcommand{\Fcal}{\mathcal{F}}

\newcommand{\Ucal}{\mathcal{U}}

\newcommand{\ew}{\newpage\noindent}

\begin{document}

\title[Strichartz estimates for gZK]
{Strichartz estimates for the generalized Zakharov-Kuznetsov equation on $ \RT $ and applications}
\author{Jakob Nowicki-Koth}

\address{J.~Nowicki-Koth: Mathematisches Institut der
Heinrich-Heine-Universit{\"a}t D{\"u}sseldorf, Universit{\"a}tsstr. 1,
40225 D{\"u}sseldorf, Germany}
\email{jakob.nowicki-koth@hhu.de}

\begin{abstract}
In this article, we address the Cauchy problem associated with the $k$-generalized Zakharov-Kuznetsov equation posed on $\R \times \mathbb{T}$. By establishing an almost optimal linear $L^4$-estimate, along with a family of bilinear refinements, we significantly lower the well-posedness threshold for all $k\geq2$. In particular, we show that the modified Zakharov-Kuznetsov equation is locally well-posed in $H^s(\mathbb{R} \times \mathbb{T})$ for all $s > \frac{11}{24}$.
\end{abstract}

\maketitle

\section{Introduction}

For $k \in \mathbb{N}$ and $s \in \R$, we consider the Cauchy problem
\[ \label{CP-$k$} \partial_tu + \partial_x \Delta_{xy} u = \pm \partial_x \brbig{u^{k+1}}, \qquad u(t=0) = u_0 \in H^{s}(\R \times \mathbb{T}) \tag{CP-$k$} \]
associated with the $k$-generalized Zakharov-Kuznetsov equation (gZK), where $u = u(t,x,y)$ is a real-valued function. For $k=1$, gZK reduces to the classical Zakharov-Kuznetsov equation (ZK), which was first derived on $\R^3$ by Zakharov and Kuznetsov \cite{Zakharov1974}, as a model for the propagation of nonlinear ion-acoustic waves in a magnetized plasma. When posed  on $\R^2$, ZK describes similar phenomena, and the derivation in this case is due to Laedke and Spatschek \cite{Laedke1982}. Regarding conserved quantities, it is well known that the mass 
\[ M(u(t)) \coloneqq \int_{\R \times \mathbb{T}} (u(t,x,y))^{2} \mathrm{d}(x,y) \] and the energy \[ E_{\pm}(u(t)) \coloneqq \frac{1}{2} \int_{\R \times \mathbb{T}} \abs{ \nabla u(t,x,y) }_2^{2} \pm \frac{2}{k+2} (u(t,x,y))^{k+2} \mathrm{d}(x,y) \]
are conserved over time, provided the solution is of sufficient regularity.
Both $M$ and $E_{\pm}$ play a central role in the well-posedness analysis of gZK, particularly in extending local to global solutions. The well-posedness theory of gZK has been extensively studied over the past decades, both on $\R^3$ and on $\R^2$. For the three-dimensional case, we refer to \cite{Pilod2015, Gruenrock2014, Gruenrock2015, Kinoshita2022, Ribaud2012, Linares2009, Herr, Kato2018}, whereas for the two-dimensional case, we refer to \cite{Faminskii, Linares2009a, Linares2011, Gruenrock2014a, Pilod2015, Kinoshita2021, Kinoshita2022, Bhattacharya2020, Biagioni2003, Farah2012, Ribaud}. The semi-periodic case, however, has not yet been studied in comparable depth and has only recently come into focus. We provide an overview of the results obtained to date: \\
In the case $k=1$, the foundational work was carried out by Linares, Pastor, and Saut \cite{Linares2010}, who established local well-posedness in $H^s(\R \times \mathbb{T})$ for all $s > \frac{3}{2}$. By developing a bilinear estimate that allows for a gain of up to $\frac{1}{2}-$ derivatives, Molinet and Pilod \cite{Pilod2015} were subsequently able to prove global well-posedness for all $s \geq 1$. Based on a refinement of said estimate, Osawa \cite{Osawa2022} succeeded in lowering the threshold for local well-posedness to $s > \frac{9}{10}$, and, in joint work with Takaoka \cite{Osawa2024}, further established global well-posedness for all $s > \frac{29}{31}$ by employing the $I$-method. More recently, Cao-Labora \cite{CaoLabora2025} developed a further refinement of said bilinear estimate, carefully adapted to the resonance function. By invoking geometric arguments, he was then able to improve the local well-posedness result to $s > \frac{3}{4}$ (or $s > \frac{1}{2}$ under an additional low-frequency condition), and in both cases established optimality up to the endpoint. For the case $k \geq 2$, Farah and Molinet \cite{Farah2024} proved the linear $L^4$-estimate

\begin{equation} \label{old L4} \norm{u}_{L_{txy}^4(\R \times \R \times \mathbb{T})} \lesssim \norm{u}_{X_{\frac{1}{6},\frac{3}{8}}(\phi)}, \end{equation}

where $X_{\frac{1}{6},\frac{3}{8}}(\phi)$ denotes the Bourgain space associated with the ZK equation, and by using this estimate in conjunction with the bilinear estimate due to Molinet and Pilod, they obtained local well-posedness for every $s > 1-$, explicitly establishing $s > \frac{5}{6}$ in the case $k=2$. Furthermore, by making use of the conservation laws, they also proved (small data) global well-posedness for every $s \geq 1$. \\
The development of increasingly sharp linear Strichartz estimates has become a central focus of current research concerning the well-posedness study of different problems posed on $\R \times \mathbb{T}$. Examples of such recent work include the contributions of Herr, Schippa, and Tzvetkov \cite{Schippa} on the semiperiodic dispersion-generalized KP-II equation; Corcho and Mallqui \cite{Corcho} on the semiperiodic NLS equation with fractional derivatives in the periodic component; and Ba\c{s}ako\u{g}lu, Sun, Tzvetkov, and Wang \cite{Basakoglu} on the semiperiodic hyperbolic NLS equations.
In this spirit, the proofs of several new linear and bilinear estimates of Strichartz type will also serve as the starting point in this article. In particular, we wish to highlight the following linear $L^4$-estimate, which we will furthermore show to be optimal up to the loss of an $\epsilon$-derivative.

\begin{theorem}
\label{L4}
Let $\epsilon > 0$ and $b > \frac{1}{2}$ be arbitrary. If $u \in X_{\epsilon, b}(\phi) $, then the following estimate holds true: \begin{equation} \label{ineq1} \norm{u}_{L_{txy}^4(\R \times \RT)} \lesssim_{\epsilon, b} \norm{u}_{X_{\epsilon, b}(\phi)}. \end{equation}

\end{theorem}

For the proof of this estimate, we will follow the approach of Takaoka and Tzvetkov \cite{Takaoka2001} for the semiperiodic Schrödinger equation and consider sets of the form \[ \set{(\xi,q) \in \R \times \mathbb{Z} \ | \ c_1 \leq p(\xi,q) \leq c_2} \]
for $c_1,c_2 \in \R$ and a specific polynomial $p$. The simple polynomial structure of
$p$ will then allow us to sharply estimate the measure of these sets, and a subsequent dyadic summation will yield the desired estimate. Moreover a modification of the proof of Theorem \ref{L4} will yield a family of bilinear refinements of \eqref{ineq1}, which in certain frequency ranges allow a gain of up to $\frac{1}{4}-$ derivatives. \\
As a first application of the newly obtained linear estimates, we will then verify the well-posedness results summarized in the following theorem.

\begin{theorem}
\label{general LWP}
Let $k \in \N^{\geq2}$, define \[ s_{0}(k) \coloneqq \begin{cases} \frac{1}{2} & \ \text{if} \ k=2 \\ \frac{8}{15} & \ \text{if} \ k=3 \\  1-\frac{16}{9k} & \ \text{if} \ k\geq 4 \end{cases}, \] and let $s > s_{0}(k)$ be arbitrary. Then for every $u_0 \in H^{s}(\RT)$, there exists a $\delta = \delta(\norm{u_0}_{H^{s_0(k)+}(\RT)}) > 0$ and a unique solution \[ u \in X_{s,\frac{1}{2}+}^\delta(\phi) \] to \eqref{CP-$k$}. Moreover, for every $\tilde{\delta} \in \intoo{0,\delta}$, there exists a neighborhood $\Ucal \subseteq H^{s}(\RT)$ of $u_0$ such that the flow map \[ S: H^{s}(\RT) \supseteq \Ucal \rightarrow X_{s,\frac{1}{2}+}^{\tilde{\delta}}(\phi), \ u_0 \mapsto S(u_0) \coloneqq u \] (data upon solution) is smooth.

\end{theorem}

Subsequently, in the case $k=2$ (which corresponds to the modified Zakharov-Kuznetsov equation), we will carry out a detailed case-by-case analysis using the bilinear refinements of \eqref{ineq1} and the resonance function, and push the local well-posedness threshold in this case down to $s > \frac{11}{24}$.

\begin{theorem}
\label{LWP mZK}
Let $s > \frac{11}{24}$ be arbitrary. Then for every $u_0 \in H^{s}(\RT)$, there exists a $\delta = \delta(\norm{u_0}_{H^{\frac{11}{24}+}(\RT)}) > 0$ and a unique solution \[ u \in X_{s,\frac{1}{2}+}^\delta(\phi) \] to \begin{equation*} \partial_t u + \partial_x \Delta_{xy} u = \pm \partial_x \brbig{u^3}, \quad u(t=0) = u_0. \end{equation*}
Moreover, for every $\tilde{\delta} \in \intoo{0,\delta}$, there exists a neighborhood $\Ucal \subseteq H^{s}(\RT)$ of $u_0$ such that the flow map \[ S: H^{s}(\RT) \supseteq \Ucal \rightarrow X_{s,\frac{1}{2}+}^{\tilde{\delta}}(\phi), \ u_0 \mapsto S(u_0) \coloneqq u \] (data upon solution) is smooth.

\end{theorem}

\subsection*{Acknowledgments}
This work is part of the author's PhD thesis, written under the supervision of Axel Grünrock. The author gratefully acknowledges his continued guidance, the suggestion of this line of problems, and numerous insightful discussions.

\section{Preliminaries}

We start by fixing notation and introducing the function spaces that will be used in this paper.

\subsection{Notation}

If $A$ is a measurable subset of $\R \times \Z$, $\R \times \R \times \Z$ or $\R^n$, then $\abs{A}$ will denote its Haar measure. For an arbitrary number $x \in \mathbb{R}$, $x +$ will denote a number slightly greater than $x$, while $x-$\footnote{In the context of interpolation, the notation $\infty -$ will be used to indicate a number that is very large, but still finite.} will denote a number slightly smaller than $x$. If we have two real numbers, $x$ and $y$, then their maximum will be denoted by $x \lor y$ and their minimum by $x \land y$. If $x$ and $y$ are both positive, then $x \lesssim y$ will mean that there exists some constant $c > 0$, such that $x \leq cy $ holds. In the case that the constant $c$ can be chosen particularly close to $0$, we will write $x \ll y$ and if both $x \lesssim y$ and $y \lesssim x$ hold, we will write $x \sim y$. Furthermore, we will denote the absolute value of a real number $x$ by $\abs{x}$, while for a vector $x = (x_1,x_2) \in \R^2$, the Euclidean norm will be denoted by $\abs{x}_2$.  Using the Euclidean norm, we define the Japanese brackets as $ \langle x \rangle \coloneqq (1+\abs{x}_2^2)^\frac{1}{2}$. For an admissible function $f = f(t,x,y)$, where $(t,x,y) \in \R^2 \times \mathbb{T}$, its space-time Fourier transform will be denoted by $\hat{f}$, and its inverse space-time Fourier transform by $\check{f}$. If we are interested in the partial Fourier transform of such a function, we indicate this by using subscripts, e.g., in $\Fcal_{x}f$ and $\Fcal_{xy}^{-1}f$. We use the Fourier transform to define the Bessel and Riesz potential operators of order $-s$, where $s$ is a real number:
\[ J^{s}f \coloneqq \mathcal{F}_{xy}^{-1} \langle (\xi,q) \rangle^{s} \mathcal{F}_{xy}f, \qquad I^{s}f \coloneqq \mathcal{F}_{xy}^{-1}  \abs{(\xi,q)}_2^{s} \mathcal{F}_{xy}f. \] If we wish to use $J$ and $I$ with respect to different combinations of variables, we specify them by using subscripts as well. We also make use of the Fourier transform to define the unitary group $(e^{-\partial_x \Delta_{xy} t})_{t \in \R}$ associated with the linear part of gZK:
\[ e^{-\partial_x \Delta_{xy} t}f \coloneqq \mathcal{F}_{xy}^{-1} e^{it\phi(\xi,q)} \mathcal{F}_{xy}f. \] Here, $\phi(\xi,q) \coloneqq \xi(\xi^{2}+q^{2})$ denotes the corresponding phase function. For later well-posedness considerations related to the modified Zakharov-Kuznetsov equation, it will be necessary to work with the associated resonance function $R$, defined by \[ R(\xi,q,\xi_1,q_1,\xi_2,q_2) \coloneqq \phi(\xi,q) - \phi(\xi_1,q_1)-\phi(\xi_2, q_2) - \phi(\xi-\xi_1-\xi_2,q-q_1-q_2). \] Following the approach of Molinet and Pilod in \cite{Pilod2015}, it will also be useful to introduce Littlewood-Paley projectors adapted to $R$: Let $(\xi,q) \in \R \times \Z$. We define the following dilated quantity $\abs{(\xi,q)} \coloneqq (3\xi^2+q^2)^{\frac{1}{2}}$. Then, for a fixed smooth cut-off function $\mu$ with \[ \mu \in C_{0}^{\infty}(\mathbb{R}),\quad 0 \leq \mu \leq 1, \quad \mu |_{[-\frac{5}{4},\frac{5}{4}]} = 1 \quad \text{and} \quad \text{supp}(\mu) \subseteq \intcc[\bigg]{-\frac{8}{5},\frac{8}{5}}, \] we set \[ \psi (\xi) \coloneqq \mu (\xi) - \mu (2 \xi) \] and define \[ \psi_{1}(\xi,q) \coloneqq \mu (\abs{(\xi,q)})  \quad \text{and} \quad \eta_1(\tau, \xi, q) \coloneqq \mu (\tau - \phi(\xi,q)), \] as well as \[ \psi_{2^{n}} (\xi,q) \coloneqq \psi(2^{-n} \abs{(\xi,q)}) \quad \text{and} \quad \eta_{2^n}(\tau, \xi, q) \coloneqq \psi(2^{-n} (\tau - \phi(\xi, q))), \] where $n \in \mathbb{N}$.
 If we now replace $2^{n}$ with $N$, we obtain \[ \text{supp}(\psi_{1}) \subseteq \setbigg{ \abs{(\xi,q)} \leq \frac{8}{5}} \quad \text{and} \quad \text{supp}(\psi_{N}) \subseteq \setbigg{ \frac{5}{8}N \leq \abs{(\xi,q)} \leq \frac{8}{5}N }, \quad N \geq 2 \] and this construction provides \[ \sum_{N \in \{2^{n} \ | \ n \in \mathbb{N}_{0} \}} \psi_{N}(\xi,q) = 1 \qquad \forall  (\xi,q) \in \mathbb{R} \times \mathbb{Z}. \] The Littlewood-Paley projectors are then given by \[  P_{N}f \coloneqq \mathcal{F}_{xy}^{-1}\psi_{N} \mathcal{F}_{xy}f, \quad \text{and} \quad Q_Lf \coloneqq \Fcal_{txy}^{-1} \eta_L \Fcal_{txy}f, \qquad N, L \in \{ 2^{n} \ | \  n \in \mathbb{N}_{0} \}. \]

\subsection{Function spaces}

For $s \in \mathbb{R}$ and $X \in \{ \mathbb{R}, \mathbb{T}, \mathbb{R} \times \mathbb{T} \}$, $H^{s}(X)$ will denote the usual Sobolev-space, endowed with the norm $\lVert \cdot \rVert_{H^{s}(X)} \coloneqq \lVert J^{s} \cdot \rVert_{L^{2}(X)}$, while its homogeneous version will be denoted by $\dot{H}^{s}(X)$ and equipped with the norm $\lVert \cdot \rVert_{\dot{H}^{s}(X)} \coloneqq \lVert I^{s} \cdot \rVert_{L^{2}(X)}$. We will often replace the underlying space $X$ in the notation with variable subscripts, as in  $H^{s}(\RT) = H^s_{xy}$. Now let $T \leq 1$ be a positive real number and $1\leq p,q \leq \infty$. Then, the mixed Lebesgue spaces such as $L^{p}_{tx}L^{q}_{y}$ and $L_{Ty}^{p}L_{x}^{q}$ will be endowed with the following norms: \[ \lVert f \rVert_{L_{tx}^{p}L_{y}^{q}} \coloneqq \brbigg{ \int_{\mathbb{R}^2} \lVert f(t,x,\cdot) \rVert_{L^{q}(\mathbb{T})}^{p}  \mathrm{d}(t,x) }^{\frac{1}{p}} \] and \[ \lVert f \rVert_{L_{Ty}^{p}L_{x}^{q}} \coloneqq \brbigg{ \int_{\intcc{-T,T} \times \mathbb{T}} \lVert f(t,\cdot,y) \rVert_{L^{q}(\mathbb{R})}^{p} \mathrm{d} (t,y) }^{\frac{1}{p}}, \]
with the obvious modifications in the case $p=\infty$. For later well-posedness considerations, we further introduce the semiperiodic Bourgain spaces $X_{s,b}(\phi)$, with regularity indices $s,b \in \mathbb{R}$. We equip these with the norm \[ \lVert f \rVert_{X_{s,b}(\phi)} \coloneqq \brbigg{ \int_{\mathbb{R}} \int_{\mathbb{R}} \sum_{q \in \Z} \langle (\xi,q) \rangle^{2s} \langle \tau - \phi(\xi,q) \rangle^{2b} \lvert \hat f(\tau,\xi,q) \rvert^{2}  \mathrm{d}(\tau,\xi) }^{\frac{1}{2}}, \] where $\phi(\xi,q) = \xi(\xi^2+q^2)$ refers to the phase function defined earlier. Lastly, we define the restriction norm spaces $X_{s,b}^{\delta}(\phi)$ as the sets of all restrictions of $X_{s,b}(\phi)$-functions to $ \intcc{-\delta, \delta} \times \RT$, where $\delta$ is a positive real number. These spaces will be equipped with the norm \[ \lVert f \rVert_{X_{s,b}^{\delta}(\phi)} \coloneqq \inf \setBig{ \lVert \widetilde{f} \rVert_{X_{s,b}(\phi)} \ \Big| \ \widetilde{f} |_{ \intcc{-\delta,\delta} \times \RT} = f \ \text{and} \ \widetilde{f} \in X_{s,b}(\phi) }. \] Since the phase function $\phi$ won't vary throughout this paper, we will just write $X_{s,b}$ and $X_{s,b}^{\delta}$, respectively.

\section{Linear and bilinear Strichartz estimates}

This section will be devoted to establishing the linear and bilinear Strichartz estimates needed for the proofs of Theorem \ref{general LWP} and Theorem \ref{LWP mZK}. We begin with the proof of a slightly sharpened version of a bilinear estimate that goes back to Molinet and Pilod (see Proposition 3.6 in \cite{Pilod2015}). This estimate allows for a gain of $\frac{1}{2}-$ derivatives for widely separated mixed frequencies $(\xi_1,q_1)$ and $(\xi_2,q_2)$.

\begin{prop}

\label{MP prop}

Let $\epsilon > 0$ and $b > \frac{1}{2}$ be arbitrary. Furthermore, we define the bilinear Fourier multiplier $MP(\cdot, \cdot)$ by its Fourier transform \begin{align*} &\widehat{MP(u,v)}(\tau,\xi,q) \coloneqq \\ & \int_{\R^2} \sum_{\substack{q_1 \in \Z \\ \ast}} \abs{\abs{(\xi_1,q_1)}^2 - \abs{(\xi_2,q_2)}^2}^\frac{1}{2} \hat{u}(\tau_1,\xi_1,q_1) \hat{v}(\tau_2,\xi_2,q_2) \mathrm{d}(\tau_1,\xi_1), \end{align*} where $\ast$ indicates the convolution constraint $(\tau,\xi,q) = (\tau_1+\tau_2,\xi_1+\xi_2,q_1+q_2)$. Then, the following estimate holds for all $J_{y}^{\frac{1}{2}+\epsilon}u, v \in X_{0,b}$: \begin{equation} \label{OG bilin} \norm{MP(u,v)}_{L_{txy}^2} \lesssim_{\epsilon, b} \norm{J_{y}^{\frac{1}{2}+\epsilon}u}_{X_{0,b}} \norm{v}_{X_{0,b}}. \end{equation}

\end{prop}

\begin{proof}

We fix $\epsilon > 0$ and $b \coloneqq \frac{1}{2}+ \epsilon' > \frac{1}{2}$. By employing Parseval's identity, the Cauchy-Schwarz inequality and Fubini's theorem, we obtain \begin{align*} \norm{MP(u,v)}_{L_{txy}^2} & \lesssim \brBigg{\sup_{(\tau,\xi,q) \in \R \times \R \times \Z} c_0^\frac{1}{2}} \norm{J_{y}^{\frac{1}{2}+\epsilon}u}_{X_{0,b}} \norm{v}_{X_{0,b}},  \end{align*} where \begin{align*} c_0 &\coloneqq \int_{\R^2} \sum_{\substack{q_1 \in \Z \\ \ast}} \abs{\abs{(\xi_1,q_1)}^2 - \abs{(\xi_2,q_2)}^2} \langle q_1 \rangle^{-1-2\epsilon} \langle \tau_1 - \phi(\xi_1,q_1) \rangle^{-1-2\epsilon'} \\ & \ \ \ \cdot \langle \tau_2 - \phi(\xi_2,q_2) \rangle^{-1-2\epsilon'} \mathrm{d} (\tau_1,\xi_1), \end{align*} and for the integration over $\tau_1$, Lemma 4.2 in \cite{Ginibre1997} yields that \begin{align*} &\int_{\R} \langle \tau_1 - \phi(\xi_1,q_1) \rangle^{-1-2\epsilon'} \langle \tau - \tau_1 - \phi(\xi-\xi_1,q-q_1) \rangle^{-1-2\epsilon'} \mathrm{d} \tau_1 \\ &\lesssim \langle \tau - \phi(\xi_1,q_1) - \phi(\xi - \xi_1, q - q_1) \rangle^{-1-2\epsilon'}. \end{align*}
From this, it follows that \begin{align*} c_0 &\lesssim \sum_{q_1 \in \Z} \langle q_1 \rangle^{-1-2\epsilon} \int_{\R} \abs{\abs{(\xi_1,q_1)}^2 - \abs{(\xi-\xi_1,q-q_1)}^2} \\ & \ \ \ \cdot \langle \tau - \phi(\xi_1,q_1) - \phi(\xi - \xi_1, q - q_1) \rangle^{-1-2\epsilon'} \mathrm{d} \xi_1, \end{align*}
and substituting \[ \tilde{\xi_1} = \tau - \phi(\xi_1,q_1) - \phi(\xi-\xi_1,q-q_1) \leadsto "\mathrm{d} \tilde{\xi_1} = (\abs{(\xi - \xi_1,q - q_1)}^2 - \abs{(\xi_1,q_1)}^2) \mathrm{d} \xi_1" \] leads to \[ c_0 \lesssim \sum_{q_1 \in \Z} \langle q_1 \rangle^{-1-2\epsilon} \int_{\R} \langle \tilde{\xi_1} \rangle^{-1-2\epsilon'} \mathrm{d} \tilde{\xi_1} < \infty. \]
This completes the proof of \eqref{OG bilin}.
\end{proof}

\begin{rem}

\label{remark mp}

\begin{enumerate}

\item[(i)] The proof shows that it does not matter whether $J_{y}^{\frac{1}{2}+}$ is applied to $u$ or $v$, so that the derivative loss can always be shifted to the low-frequency factor.

\item[(ii)] By applying Sobolev's embedding theorem and Lemma 4.2 from \cite{Ginibre1997} in the form \begin{align*} &\int_{\R} \langle \tau_1 - \phi(\xi_1,q_1) \rangle^{-\frac{1}{2}-2\epsilon'} \langle \tau - \tau_1 - \phi(\xi-\xi_1,q-q_1) \rangle^{-\frac{1}{2}-2\epsilon'} \mathrm{d} \tau_1 \lesssim 1, \end{align*} an argument analogous to that of Proposition \ref{MP prop} yields the estimate \begin{equation} \label{trivial MP} \norm{MP(u,v)}_{L_{txy}^2} \lesssim \norm{u}_{X_{1+,\frac{1}{4}+}} \norm{v}_{X_{1+,\frac{1}{4}+}}. \end{equation}
By bilinear interpolation between \eqref{OG bilin} and \eqref{trivial MP}, we then obtain that for every $\epsilon > 0$, there exists $\tilde{\epsilon} > 0$ such that \begin{equation} \label{MP dual} \norm{MP(u,v)}_{L_{txy}^2} \lesssim \norm{u}_{X_{\frac{1}{2}+\epsilon,\frac{1}{2}-\tilde{\epsilon}}} \norm{v}_{X_{\epsilon, \frac{1}{2}-\tilde{\epsilon}}}, \end{equation}
and this will turn out to be useful for later estimates by duality.

\end{enumerate}

\end{rem}

The linear Strichartz estimates that we want to establish next are based on the observation that the linear propagator $e^{-t\partial_x \Delta_{xy}}$ associated with gZK implicitly contains both the one-dimensional Schrödinger and Airy propagators. This becomes evident when applying partial Fourier transforms $\Fcal_{x}$ and $\Fcal_{y}$ to $e^{-t\partial_x \Delta_{xy}}u_{0}$. One obtains \begin{equation} \label{Schrödinger point of view} \brBig{\Fcal_{x}e^{-t\partial_x \Delta_{xy}}u_{0}}(\xi,y) = e^{it\xi^3} \cdot \brBig{e^{i(-\xi t) \partial_{y}^2} \Fcal_{x}u_{0}(\xi)}(y) \end{equation} and
\begin{equation} \label{Airy point of view} \brBig{\Fcal_{y} e^{-t\partial_x \Delta_{xy}}u_{0}}(x,q) = e^{(q^2 t) \partial_{x}} \brBig{e^{-t \partial_{x}^3}\Fcal_{y}u_{0}(q)}(x) = \brBig{e^{-t \partial_{x}^3}\Fcal_{y}u_{0}(q)}(x+q^2 t), \end{equation}
and for fixed spatial frequencies $\xi$ or $q$, one can then invoke well-known estimates for the Schrödinger and Airy equations. This argument has been employed multiple times in the context of the Schrödinger propagator (see, e.g. \cite{Herr, Panthee, Grünrock2009}), whereas it seems to be new in the case of the Airy propagator. We make the described approach precise in the propositions that follow.

\begin{prop} 
\label{Schrödinger Strichartz}
Let $T > 0$ and $\epsilon > 0$ be abitrary. Furthermore, let $u_0$ and $v_0$ be functions with $J_{x}^{\frac{1}{4}}u_{0} \in L_{xy}^2$ and $J_{x}^{\frac{1}{3}}J_{y}^{\epsilon}v_0 \in L_{xy}^2$. Then the two estimates \begin{equation} \label{Schrödinger L4} \norm{e^{-t \partial_{x} \Delta_{xy}}u_{0}}_{L_{Txy}^4} \lesssim_{T} \norm{J_{x}^\frac{1}{4} u_0}_{L_{xy}^2} \end{equation} and \begin{equation} \label{Schrödinger L6} \norm{e^{-t \partial_{x} \Delta_{xy}}v_{0}}_{L_{Txy}^6} \lesssim_{T,\epsilon} \norm{J_{x}^\frac{1}{3}J_{y}^{\epsilon} v_0}_{L_{xy}^2} \end{equation} hold. Moreover, if $p \in \intcc{2,6}$ and $b > \frac{1}{2}$ are given, then the estimate \begin{equation} \label{Schrödinger Lp} \norm{u}_{L_{Txy}^p} \lesssim_{T,\epsilon,b} \norm{J_{x}^{\frac{1}{2}-\frac{1}{p}} J_{y}^{(\frac{3}{2}-\frac{3}{p})\epsilon} u}_{X_{0,(\frac{3}{2}-\frac{3}{p})b}} \end{equation} holds for every time-localized $u$ with $J_{x}^{\frac{1}{2}-\frac{1}{p}} J_{y}^{(\frac{3}{2}-\frac{3}{p})\epsilon} u \in X_{0,(\frac{3}{2}-\frac{3}{p})b}$.

\end{prop}

\begin{proof}

We begin with the proof of \eqref{Schrödinger L4}. Taking \eqref{Schrödinger point of view} into account, the Sobolev embedding theorem for homogeneous spaces, followed by Parseval's identity in the $x$-variable and an application of Minkowski's integral inequality, yields \begin{align*} \norm{e^{-t \partial_{x} \Delta_{xy}}u_{0}}_{L_{Txy}^4} &\lesssim \norm{I_{x}^\frac{1}{4}e^{-t \partial_{x} \Delta_{xy}}u_{0}}_{L_{Ty}^4L_{x}^2} \\ & \sim \norm{\abs{\xi}^{\frac{1}{4}} e^{i(-\xi t) \partial_{y}^2} \Fcal_{x}u_{0}}_{L_{Ty}^4L_{\xi}^2} \\ & \leq \norm{\abs{\xi}^{\frac{1}{4}} e^{i(-\xi t) \partial_{y}^2} \Fcal_{x}u_{0}}_{L_{\xi}^2L_{Ty}^4} \\ & = \brBigg{\int_{\R} \abs{\xi}^\frac{1}{2} \brBigg{ \int_{\mathbb{T}} \int_{-T}^{T} \absBig{\brBig{e^{i(-\xi t)\partial_{y}^2}\Fcal_{x}u_{0}(\xi)}(y)}^4 \mathrm{d} t \mathrm{d} y}^{\frac{1}{2}} \mathrm{d} \xi}^\frac{1}{2}. \end{align*}
Now, substituting $\tilde{t} = -\xi t$ in the innermost integral, we obtain \[ ... = \brBigg{\int_{\R} \brBigg{ \int_{\mathbb{T}} \int_{-T\abs{\xi}}^{T\abs{\xi}} \absBig{\brBig{e^{i\tilde{t}\partial_{y}^2}\Fcal_{x}u_{0}(\xi)}(y)}^4 \mathrm{d} \tilde{t} \mathrm{d} y}^{\frac{1}{2}} \mathrm{d} \xi}^\frac{1}{2},  \]
and if $k \in \N_0$ is the uniquely determined integer such that $k \pi \leq T \abs{\xi} < (k+1) \pi$, then it follows from the $2\pi$-periodicity of the integrand in $\tilde{t}$ that \begin{align*} ... &\leq \brBigg{\int_{\R} \brBigg{ \int_{\mathbb{T}} \int_{-(k+1)\pi}^{(k+1)\pi} \absBig{\brBig{e^{i\tilde{t}\partial_{y}^2}\Fcal_{x}u_{0}(\xi)}(y)}^4 \mathrm{d} \tilde{t} \mathrm{d} y}^{\frac{1}{2}} \mathrm{d} \xi}^\frac{1}{2} \\ & \lesssim \brBigg{\int_{\R} (T\abs{\xi} + 1)^\frac{1}{2} \brBigg{ \int_{\mathbb{T}} \int_{\mathbb{T}} \absBig{\brBig{e^{i\tilde{t}\partial_{y}^2}\Fcal_{x}u_{0}(\xi)}(y)}^4 \mathrm{d} \tilde{t} \mathrm{d} y}^{\frac{1}{2}} \mathrm{d} \xi}^\frac{1}{2} \\ & \lesssim_{T} \brBigg{\int_{\R} \langle \xi \rangle^{\frac{1}{2}} \norm{e^{i\tilde{t}\partial_{y}^2}\Fcal_{x}u_{0}(\xi)}_{L_{\tilde{t} y}^4(\mathbb{T} \times \mathbb{T})}^2 \mathrm{d} \xi}^\frac{1}{2}. \end{align*}
For the inner $L_{\tilde{t}y}^4$-norm, we can now rely on the well-known Schrödinger $L^4$-estimate, originally due to Zygmund \cite{Zygmund}, to obtain \begin{align*} ... &\lesssim \brBigg{\int_{\R} \langle \xi \rangle^\frac{1}{2} \norm{\Fcal_{x}u_{0}(\xi)}_{L_{y}^2(\mathbb{T})}^2 \mathrm{d} \xi}^\frac{1}{2} \sim \norm{J_{x}^\frac{1}{4}u_{0}}_{L_{xy}^2}. \end{align*} This concludes the proof of \eqref{Schrödinger L4}.
To prove \eqref{Schrödinger L6}, we proceed analogously and arrive at \begin{align*} \norm{e^{-t \partial_{x} \Delta_{xy}}v_{0}}_{L_{Txy}^6} &\lesssim \norm{I_{x}^\frac{1}{3}e^{-t \partial_{x} \Delta_{xy}}v_{0}}_{L_{Ty}^6L_{x}^2} \\ & \sim \norm{\abs{\xi}^{\frac{1}{3}} e^{i(-\xi t) \partial_{y}^2} \Fcal_{x}v_{0}}_{L_{Ty}^6L_{\xi}^2} \\ & \leq \norm{\abs{\xi}^{\frac{1}{3}} e^{i(-\xi t) \partial_{y}^2} \Fcal_{x}v_{0}}_{L_{\xi}^2L_{Ty}^6} \\ & \lesssim \brBigg{ \int_{\R} \abs{\xi}^\frac{1}{3} (T\abs{\xi} + 1)^\frac{1}{3} \brBigg{ \int_{\mathbb{T}} \int_{\mathbb{T}} \absBig{ \brBig{e^{i\tilde{t} \partial_{y}^2}\Fcal_{x}v_{0}(\xi)}(y)}^6 \mathrm{d} \tilde{t} \mathrm{d} y}^{\frac{1}{3}} \mathrm{d} \xi}^\frac{1}{2} \\ & \lesssim_{T} \brBigg{\int_{\R} \langle \xi \rangle^{\frac{2}{3}} \norm{e^{i\tilde{t} \partial_{y}^2}\Fcal_{x}v_{0}(\xi)}_{L_{\tilde{t}y}^6(\mathbb{T} \times \mathbb{T})}^2 \mathrm{d} \xi}^\frac{1}{2}. \end{align*}
At this point, we can now apply the Schrödinger $L^6$-estimate with a loss of an $\epsilon$-derivative in $y$ (see Proposition 2.36 in \cite{Bourgain1993}), and conclude that \begin{align*} ... &\lesssim_{\epsilon} \brBigg{\int_{\R} \langle \xi \rangle^{\frac{2}{3}} \norm{J_{y}^\epsilon\Fcal_{x}v_{0}(\xi)}_{L_{y}^2(\mathbb{T})}^2 \mathrm{d} \xi}^\frac{1}{2} \sim \norm{J_{x}^{\frac{1}{3}}J_{y}^{\epsilon} v_0}_{L_{xy}^2}.  \end{align*}
It remains to prove \eqref{Schrödinger Lp}. The transfer principle (see Lemma 2.3 in \cite{Ginibre1997}) transforms \eqref{Schrödinger L4} and \eqref{Schrödinger L6} for time-localized $u$ and $b > \frac{1}{2}$ into \begin{equation} \label{Xsb Schrödinger L4} \norm{u}_{L_{Txy}^4} \lesssim_{T,b} \norm{J_{x}^\frac{1}{4}u}_{X_{0,b}} \end{equation} and \begin{equation} \label{Xsb Schrödinger L6} \norm{u}_{L_{Txy}^6} \lesssim_{T,\epsilon, b} \norm{J_{x}^\frac{1}{3}J_{y}^\epsilon u}_{X_{0,b}},  \end{equation} so that we obtain \eqref{Schrödinger Lp} by interpolating \eqref{Xsb Schrödinger L6} with the trivial estimate \begin{equation} \label{trivial L2} \norm{u}_{L_{Txy}^2} \lesssim \norm{u}_{L_{txy}^2} = \norm{u}_{X_{0,0}}. \end{equation}
\end{proof}

With the "Schrödinger point of view" now addressed, we turn to deriving linear estimates that arise from the Airy perspective.

\begin{prop}
\label{Airy Strichartz}
Let $u_0$ and $v_0$ be functions such that $J_{y}^\frac{1}{3}u_0 \in L_{xy}^2$ and $v_0 \in L_{xy}^2$. Then the estimates \begin{equation} \label{Airy L6} \norm{I_{x}^\frac{1}{6} e^{-t\partial_{x} \Delta_{xy}}u_{0}}_{L_{txy}^6} \lesssim \norm{J_{y}^\frac{1}{3}u_{0}}_{L_{xy}^2} \end{equation} and
\begin{equation} \label{Airy endpoint} \norm{I_{x}^\frac{1}{4} e^{-t\partial_{x} \Delta_{xy}} v_{0}}_{L_{t}^4L_{x}^\infty L_{y}^2} \lesssim \norm{v_{0}}_{L_{xy}^2} \end{equation} hold.\footnote{The "Airy endpoint estimate" will not be used in this paper, but it might be of independent interest for future applications.}
Furthermore, if $p \in \intcc{2,6}$, $b > \frac{1}{2}$, and $u$ is a function with $J_{y}^{\frac{1}{2}-\frac{1}{p}}u \in X_{0,(\frac{3}{2}-\frac{3}{p})b}$, then \begin{equation} \label{Airy Lp} \norm{I_{x}^{\frac{1}{4}-\frac{1}{2p}}u}_{L_{txy}^p} \lesssim_{b} \norm{J_{y}^{\frac{1}{2}-\frac{1}{p}}u}_{X_{0,(\frac{3}{2}-\frac{3}{p})b}} \end{equation}
holds true as well.

\end{prop}

\begin{proof}

We begin the proof of \eqref{Airy L6} in a manner similar to that of \eqref{Schrödinger L6}, except that we apply the Sobolev embedding theorem in the $y$-variable and, following Parseval's identity, make use of equation \eqref{Airy point of view}:
\begin{align*} \norm{I_{x}^\frac{1}{6} e^{-t\partial_{x} \Delta_{xy}}u_{0}}_{L_{txy}^6} &\lesssim \norm{J_{y}^\frac{1}{3} I_{x}^\frac{1}{6} e^{-t\partial_{x} \Delta_{xy}}u_{0}}_{L_{tx}^6L_{y}^2} \\ & \sim \norm{\langle q \rangle^\frac{1}{3} \brBig{I_{x}^\frac{1}{6} e^{-t\partial_{x}^3} \Fcal_{y}u_{0}}(x+q^2t)}_{L_{tx}^6 L_{q}^2} \\ & \leq \norm{\langle q \rangle^\frac{1}{3} \brBig{I_{x}^\frac{1}{6} e^{-t\partial_{x}^3} \Fcal_{y}u_{0}}(x+q^2t)}_{L_{q}^2 L_{tx}^6} \\ & = \brBigg{\sum_{q \in \Z} \langle q \rangle^\frac{2}{3} \brBigg{\int_{\R} \int_{\R} \absBig{\brBig{I_{x}^\frac{1}{6} e^{-t\partial_{x}^3} \Fcal_{y}u_{0}(q)}(x+q^2t)}^6 \mathrm{d} x \mathrm{d} t}^\frac{1}{3}}^\frac{1}{2}.   \end{align*}
By the translation invariance of the Lebesgue measure, it then follows that \begin{align*} ... &= \brBigg{\sum_{q \in \Z} \langle q \rangle^\frac{2}{3} \brBigg{\int_{\R} \int_{\R} \absBig{\brBig{I_{x}^\frac{1}{6} e^{-t\partial_{x}^3} \Fcal_{y}u_{0}(q)}(x)}^6 \mathrm{d} x \mathrm{d} t}^\frac{1}{3}}^\frac{1}{2} \\ & = \brBigg{\sum_{q \in \Z} \langle q \rangle^\frac{2}{3} \norm{I_{x}^\frac{1}{6} e^{-t\partial_{x}^3}\Fcal_{y}(q)}_{L_{tx}^6(\R \times \R)}^2}^\frac{1}{2}, \end{align*} and for the inner $L_{tx}^6$-norm, we can apply the Airy $L^6$-estimate proved by Kenig, Ponce, Vega (see Theorem 2.4 in \cite{Kenig1991}). This yields \begin{align*} ... &\lesssim \brBigg{\sum_{q \in \Z} \langle q \rangle^\frac{2}{3} \norm{\Fcal_{y}u_{0}(q)}_{L_{x}^2(\R)}^2}^\frac{1}{2} \sim \norm{J_{y}^\frac{1}{3}u_{0}}_{L_{xy}^2}, \end{align*} and the proof is complete.
The proof of \eqref{Airy endpoint} proceeds in an entirely analogous way, except that we can dispense with the use of the Sobolev embedding theorem and, instead of the Airy $L^6$-estimate, we apply the endpoint Airy $L_{t}^4 L_{x}^\infty$-estimate (see Theorem 2.4 in \cite{Kenig1991}). We obtain:
 \begin{align*} \norm{I_{x}^\frac{1}{4} e^{-t\partial_{x} \Delta_{xy}} v_{0}}_{L_{t}^4L_{x}^\infty L_{y}^2} & \sim \norm{\brBig{I_{x}^\frac{1}{4} e^{-t\partial_{x}^3} \Fcal_{y}v_{0}}(x+q^2t)}_{L_{t}^4L_{x}^\infty L_{q}^2} \\ &\leq \norm{\brBig{I_{x}^\frac{1}{4} e^{-t\partial_{x}^3} \Fcal_{y}v_{0}}(x+q^2t)}_{L_{t}^4 L_{q}^2 L_{x}^\infty} \\ & \leq \norm{\brBig{I_{x}^\frac{1}{4} e^{-t\partial_{x}^3} \Fcal_{y}v_{0}}(x+q^2t)}_{L_{q}^2 L_{t}^4 L_{x}^\infty} \\ &= \brBigg{ \sum_{q \in \Z} \norm{I_{x}^{\frac{1}{4}}e^{-t\partial_{x}^3}\Fcal_{y}v_{0}(q)}_{L_{t}^4(\R) L_{x}^\infty(\R)}^2}^\frac{1}{2} \\ & \lesssim \brBigg{ \sum_{q \in \Z} \norm{\Fcal_{y}v_{0}(q)}_{L_{x}^2(\R)}^2}^\frac{1}{2} \sim \norm{v_0}_{L_{xy}^2},  \end{align*} 
and thus the proof of this partial result is also complete. Finally, the transfer principle converts estimate \eqref{Airy L6} for $b > \frac{1}{2}$ into  
\begin{equation} \label{Airy L6 Xsb} \norm{I_{x}^\frac{1}{6}u}_{L_{txy}^6} \lesssim_{b} \norm{J_{y}^\frac{1}{3}u}_{X_{0,b}}, \end{equation}
 and by interpolating this with the trivial estimate \eqref{trivial L2}, we obtain
 \eqref{Airy Lp}.
\end{proof}

\begin{rem}

The application of the Sobolev embedding theorem in the proof of \eqref{Airy L6} can be avoided by starting directly in $L_y^2$. After applying the transfer principle, we thus obtain \begin{equation} \label{Airy L6 from L2} \norm{I_x^\frac{1}{6} u}_{L_{tx}^6L_y^2} \lesssim_b \norm{u}_{X_{0,b}}, \end{equation} for every $b > \frac{1}{2}$, and by interpolating this estimate with the trivial estimate \eqref{trivial L2}, we arrive at \begin{equation} \label{Airy L4 from L2} \norm{I_x^\frac{1}{8}u}_{L_{tx}^4L_y^2} \lesssim \norm{u}_{X_{0,\frac{3}{8}+}}. \end{equation}
This estimate will later prove useful in dealing with particularly small frequency domains.

\end{rem}

Similar to the approach taken by Molinet and Farah in the proof of their linear $L^4$-estimate \eqref{old L4}, we can combine the one-dimensional arguments from Proposition \ref{Schrödinger Strichartz} and Proposition \ref{Airy Strichartz} to minimize the resulting derivative loss.

\begin{cor}
\label{optimized Lp cor}

Let $T>0$, $0 < \epsilon \ll 1$, $b > \frac{1}{2}$ and $p \in \intcc{2,6}$. Then the estimate \begin{equation} \label{optimized Lp} \norm{u}_{L_{Txy}^p} \lesssim_{T,\epsilon,b} \norm{u}_{X_{\frac{1}{3}-\frac{2}{3p}+(\frac{1}{2}-\frac{1}{p})\epsilon,(\frac{3}{2}-\frac{3}{p})b}} \end{equation} holds for every time-localized $u \in X_{\frac{1}{3}-\frac{2}{3p}+(\frac{1}{2}-\frac{1}{p})\epsilon,(\frac{3}{2}-\frac{3}{p})b}$.

\end{cor}

\begin{proof}

We define the Fourier projectors \[ P_{\setbig{\abs{\xi} \leq \langle q \rangle^{\frac{2}{3}-2\epsilon}}}u \coloneqq \Fcal_{xy}^{-1}\chi_{\setbig{\abs{\xi} \leq \langle q \rangle^{\frac{2}{3}-2\epsilon}}}\Fcal_{xy}u \] and \[ P_{\setbig{\abs{\xi} > \langle q \rangle^{\frac{2}{3}-2\epsilon}}}u \coloneqq \Fcal_{xy}^{-1}\chi_{\setbig{\abs{\xi} > \langle q \rangle^{\frac{2}{3}-2\epsilon}}}\Fcal_{xy}u. \]
Then we have \[ \norm{u}_{L_{Txy}^p} \leq \norm{P_{\setbig{\abs{\xi} \leq \langle q \rangle^{\frac{2}{3}-2\epsilon}}}u}_{L_{Txy}^p} + \norm{P_{\setbig{\abs{\xi} > \langle q \rangle^{\frac{2}{3}-2\epsilon}}}u}_{L_{Txy}^p} ,\]
and for the first term we use \eqref{Schrödinger Lp}, which, taking into account \[ \chi_{\setbig{\abs{\xi} \leq \langle q \rangle^{\frac{2}{3}-2\epsilon}}} \langle \xi \rangle^{\frac{1}{2}-\frac{1}{p}} \langle q \rangle^{(\frac{3}{2}-\frac{3}{p})\epsilon} \lesssim \langle q \rangle^{\frac{1}{3}-\frac{2}{3p}+(\frac{1}{2}-\frac{1}{p})\epsilon}, \]
leads us to 
\begin{align*} \norm{P_{\setbig{\abs{\xi} \leq \langle q \rangle^{\frac{2}{3}-2\epsilon}}}u}_{L_{Txy}^p} &\lesssim_{T, \epsilon, b} \norm{J_{x}^{\frac{1}{2}-\frac{1}{p}}J_{y}^{(\frac{3}{2}-\frac{3}{p})\epsilon} P_{\setbig{\abs{\xi} \leq \langle q \rangle^{\frac{2}{3}-2\epsilon}}}u}_{X_{0,(\frac{3}{2}-\frac{3}{p})b}} \\ & \lesssim \norm{J_{y}^{\frac{1}{3}-\frac{2}{3p}+(\frac{1}{2}-\frac{1}{p})\epsilon}u}_{X_{0,(\frac{3}{2}-\frac{3}{p})b}} \\ & \lesssim \norm{u}_{X_{\frac{1}{3}-\frac{2}{3p}+(\frac{1}{2}-\frac{1}{p})\epsilon, (\frac{3}{2}-\frac{3}{p})b}}. \end{align*}
For the second term, we apply \eqref{Airy Lp} and note that \[ \chi_{\setbig{\abs{\xi} > \langle q \rangle^{\frac{2}{3}-2\epsilon}}} \abs{\xi}^{-(\frac{1}{4}-\frac{1}{2p})} \langle q \rangle^{\frac{1}{2}-\frac{1}{p}} \lesssim \langle q \rangle^{\frac{1}{3}-\frac{2}{3p}+(\frac{1}{2}-\frac{1}{p})\epsilon}, \]
which again yields \begin{align*} \norm{P_{\setbig{\abs{\xi} > \langle q \rangle^{\frac{2}{3}-2\epsilon}}}u}_{L_{Txy}^p} &\lesssim_{T, \epsilon, b} \norm{I_{x}^{-(\frac{1}{4}-\frac{1}{2p})}J_{y}^{\frac{1}{2}-\frac{1}{p}} P_{\setbig{\abs{\xi} > \langle q \rangle^{\frac{2}{3}-2\epsilon}}}u}_{X_{0,(\frac{3}{2}-\frac{3}{p})b}} \\ & \lesssim \norm{J_{y}^{\frac{1}{3}-\frac{2}{3p}+(\frac{1}{2}-\frac{1}{p})\epsilon}u}_{X_{0,(\frac{3}{2}-\frac{3}{p})b}} \\ & \lesssim \norm{u}_{X_{\frac{1}{3}-\frac{2}{3p}+(\frac{1}{2}-\frac{1}{p})\epsilon, (\frac{3}{2}-\frac{3}{p})b}}. \end{align*}
Altogether, we obtain \[ \norm{u}_{L_{Txy}^p} \lesssim_{T,\epsilon,b} \norm{u}_{X_{\frac{1}{3}-\frac{2}{3p}+(\frac{1}{2}-\frac{1}{p})\epsilon,(\frac{3}{2}-\frac{3}{p})b}}, \]
which is what we aimed to show.
\end{proof}

\begin{rem}

\begin{enumerate} \item[(i)] In the case $p=4$, $b = \frac{1}{2}+$, Corollary \ref{optimized Lp cor} provides the estimate \[ \norm{u}_{L_{Txy}^4} \lesssim \norm{u}_{X_{\frac{1}{6}+,\frac{3}{8}+}}, \] which constitutes a slightly weakened version of estimate \eqref{old L4}.
\item[(ii)] The corollary also yields the two estimates \begin{equation} \label{optimized L6} \norm{u}_{L_{Txy}^6} \lesssim \norm{u}_{X_{\frac{2}{9}+,\frac{1}{2}+}} \end{equation}
and \begin{equation} \label{optimized L6 dual} \norm{u}_{L_{Txy}^{6-}} \lesssim \norm{u}_{\frac{2}{9}+,\frac{1}{2}-}, \end{equation} which, in this form, will be sufficient for the remaining well-posedness considerations. However, the $L^5$-estimate provided by Corollary \ref{optimized Lp cor} is still not good enough - more on this will follow after the establishment of Theorem \ref{L4}.
\end{enumerate}

\end{rem}

While the estimates proved in Propositions \ref{Schrödinger Strichartz} and \ref{Airy Strichartz} are derived from already known one-dimensional estimates, we adopt a two-dimensional approach for the proof of Theorem \ref{L4}. The following auxiliary lemma clarifies what is meant by this, and is inspired by the idea of Takaoka and Tzvetkov in their proof of a linear $L^4$-estimate for the semiperiodic Schrödinger equation (see Lemma 2.1 in \cite{Takaoka2001}).

\begin{lemma} \label{mes} For $ \alpha \in \intcc{0, 1}$, $(\tau, \xi, q) \in \mathbb{R} \times \mathbb{R}^{>0} \times \mathbb{Z}$, $ K \geq 1 $, $ c = c(\tau, \xi, q) \in \R $, $h \in \set{0,\frac{1}{2}}$
and dyadic numbers $N_{1}, N_{2} \in 2^{\mathbb{N}_{0}}$, we define the following measurable subsets of $ \R \times \Z $:
\begin{align*} B_{\tau, \xi, q}^{\mathrm{lin}}  &\coloneqq  \setbigg{(\xi_{1},q_{1}) \in \R \times \Z \ \bigg| \  \abs{\xi_{1}} < \frac{\xi}{2}, \ \absbigg{\brbigg{\xi_{1} + \frac{\xi}{2}, q_{1} + \frac{q}{2} + h}} \lesssim N_{1}, \\ & \absbigg{\brbigg{\frac{\xi}{2} - \xi_{1}, \frac{q}{2} - q_{1} - h}} \lesssim N_{2}, \ \abs{3\xi^{2}-q^2} \gtrsim 1, \ p (\xi_{1}, q_{1} + h) \in \intcc{c,c+K}  }, \end{align*}
and
\begin{align*}  B_{\tau, \xi, q}^{\alpha}  &\coloneqq  \setbigg{(\xi_{1},q_{1}) \in \R \times \Z \ \bigg| \  \xi \gtrsim 1, \ \absbigg{\brbigg{\xi_{1} + \frac{\xi}{2}, q_{1} + \frac{q}{2} + h}} \lesssim N_{1}, \\ & \absbigg{\brbigg{\frac{\xi}{2} - \xi_{1}, \frac{q}{2} - q_{1} - h}} \lesssim N_{2}, \ \abs{3\xi^{2}-q^2} \gtrsim \xi^{\alpha}, \ p (\xi_{1}, q_{1} + h) \in \intcc{c,c+K}  },  \end{align*}
where $p$ denotes the polynomial function $p(x,y) \coloneqq \xi (3x^2+y^2) + 2qxy $.
Then, the following two estimates hold for every $ \epsilon > 0$:
\begin{equation} \label{lin} \sup_{(\tau,\xi,q) \in \R \times \R^{>0} \times \Z} \abs{B_{\tau, \xi, q}^{\mathrm{lin}}}^{\frac{1}{2}} \lesssim_{\epsilon} (N_{1} \lor N_{2})^{\epsilon} K^{\frac{1}{2}}, \end{equation}
and
\begin{equation} \label{bilina} \sup_{(\tau,\xi,q) \in \R \times \R^{>0} \times \Z} \xi^{\frac{\alpha}{4}} \cdot \abs{B_{\tau, \xi, q}^{\alpha}}^{\frac{1}{2}} \lesssim_{\epsilon} (N_{1} \lor N_{2})^{\epsilon} K^{\frac{1}{2}}. \end{equation}
\end{lemma}

\begin{proof}

We start by proving \eqref{lin}. Let $ \epsilon > 0$ be arbitrary. For a fixed $q_{1} \in \Z$, $p$ is an ordinary polynomial of degree two in $\xi_{1}$, so that if we define \[ D_{1} \coloneqq \frac{q^2-3 \xi^2}{9 \xi^2}(q_{1}+h)^{2}+ \frac{c}{3 \xi} \]
and \[ D_{2} \coloneqq \frac{q^2-3 \xi^2}{9 \xi^2}(q_{1}+h)^{2}+ \frac{c+K}{3 \xi}, \]
a simple calculation shows that
\begin{align*} & p(\xi_{1}, q_{1}+h) \in \intcc{c,c+K} \land \brbigg{\abs{\xi_{1}} < \frac{\xi}{2}} \Leftrightarrow \\ & \xi_{1} \in \begin{cases} \brBig{-\frac{(q_{1}+h)q}{3 \xi} + \brBig{ \intcc[\Big]{- D_{2}^{\frac{1}{2}}, - D_{1}^{\frac{1}{2}}} \cup \intcc[\Big]{ D_{1}^{\frac{1}{2}}, D_{2}^{\frac{1}{2}}}}} \cap \intoo{-\frac{\xi}{2},\frac{\xi}{2}} & \text{if} \ D_{1} \geq 0 \\
\brBig{-\frac{(q_{1}+h)q}{3 \xi} + \intcc[\Big]{- D_{2}^{\frac{1}{2}},  D_{2}^{\frac{1}{2}}}} \cap \intoo{-\frac{\xi}{2},\frac{\xi}{2}} & \text{if} \ D_{1} \leq 0, \ D_{2} \geq 0.
 \end{cases} \end{align*}
 Thus, summing the lengths of these intervals over all admissible $q_{1} \in \Z$, we obtain an upper bound for $\abs{B_{\tau, \xi, q}^{\mathrm{lin}}}$.
We divide this approach into four cases.

(i) \underline{$q^2 - 3 \xi^2 > 0$ and $c \geq 0$}: We first note that from the definition of $B_{\tau, \xi, q}^{\mathrm{lin}}$, it follows that  \[ \abs{q_{1}} \leq \frac{1}{2} \absBig{q_{1}+\frac{q}{2}+h} + \frac{1}{2} \absBig{\frac{q}{2}-q_{1}-h} + \abs{h} \lesssim (N_{1} \lor N_{2}). \] Furthermore, the assumptions $q^2 - 3 \xi^2 > 0$ and $c \geq 0$ imply $D_{1} \geq 0$ for all $q_{1} \in \Z$, so that we obtain 
\[ \abs{B_{\tau, \xi, q}^{\mathrm{lin}}} \lesssim \sum_{\abs{q_{1}} \lesssim (N_{1} \lor N_{2})} \underbrace{\absBig{\brBig{-\frac{(q_{1}+h)q}{3 \xi} + \brBig{ \intcc[\Big]{- D_{2}^{\frac{1}{2}}, - D_{1}^{\frac{1}{2}}} \cup \intcc[\Big]{ D_{1}^{\frac{1}{2}}, D_{2}^{\frac{1}{2}}}}} \cap \intoo{-\frac{\xi}{2},\frac{\xi}{2}}}}_{\eqcolon a(q_{1})}. \]
From this estimate, we can immediately see that in the case  $\xi \leq \frac{1}{(N_1 \lor N_2)}$, it follows from $ a(q_1) \leq \xi $ that
\[ \abs{B_{\tau, \xi, q}^{\mathrm{lin}}} \lesssim \sum_{\abs{q_{1}} \lesssim (N_{1} \lor N_{2})} a(q_1) \leq \frac{1}{(N_1 \lor N_2)} \sum_{\abs{q_{1}} \lesssim (N_{1} \lor N_{2})} 1 \lesssim 1 \leq (N_1 \lor N_2)^{2 \epsilon} K, \] so that we may restrict our further investigations to $\xi > \frac{1}{(N_1 \lor N_2)}$.
In addition to $a(q_1) \leq \xi$, we also have \[ a(q_1) \leq \absBig{\intcc[\Big]{D_{1}^{\frac{1}{2}},D_{2}^{\frac{1}{2}}}} = \frac{K}{3\xi} \frac{1}{D_{1}^{\frac{1}{2}} + D_{2}^{\frac{1}{2}}} \leq \frac{K}{3\xi} \frac{(3 \xi)^{\frac{1}{2}}}{K^{\frac{1}{2}}} = \frac{K^{\frac{1}{2}}}{(3\xi)^{\frac{1}{2}}}, \] which allows us to conclude that $ a(q_1) \lesssim K^{\frac{1}{2}}$. Thus, we obtain
\begin{align*} \sum_{\abs{q_{1}} \lesssim (N_{1} \lor N_{2})} a(q_1) &\lesssim K^{\frac{1}{2}}+\sum_{\substack{\abs{q_{1}} \lesssim (N_{1} \lor N_{2}) \\ q_1 \notin \set{0,-1}}} \absBig{\intcc[\Big]{D_{1}^{\frac{1}{2}},D_{2}^{\frac{1}{2}}}} \\ & \lesssim K^{\frac{1}{2}} + \frac{K^{\frac{1}{2}}}{\xi^{\frac{1}{2}}} \sum_{\substack{\abs{q_{1}} \lesssim (N_{1} \lor N_{2}) \\ q_1 \notin \set{0,-1}}} \brBig{ \frac{q^2-3\xi^2}{3\xi K} (q_1 + h)^2 + 1}^{-\frac{1}{2}} \end{align*} and a subsequent integral comparison yields
\begin{align*} ... & \lesssim  K^{\frac{1}{2}} + \frac{K^{\frac{1}{2}}}{\xi^{\frac{1}{2}}} \int_{0}^{\tilde{c} (N_1 \lor N_2)} \brBig{ \frac{q^2-3\xi^2}{3\xi K} y^2 + 1}^{-\frac{1}{2}} \mathrm{d}y \\ & \sim K^{\frac{1}{2}} + \frac{K}{(q^2-3\xi^2)^\frac{1}{2}} \int_{0}^{\tilde{c} \br{\frac{q^2-3\xi^2}{3 \xi K}}^{\frac{1}{2}} (N_1 \lor N_2)} (1+z^2)^{-\frac{1}{2}} \mathrm{d}z.   \end{align*}
Now, taking $ 1 \lesssim q^2-3\xi^2 \lesssim (N_1 \lor N_2)^2 $, $ \xi > \frac{1}{(N_1 \lor N_2)}$ and $K \geq 1$ into consideration, we get \[ ... \lesssim K \int_{0}^{\widetilde{\tilde{c}} (N_1 \lor N_2)^\frac{5}{2}} (1+z^2)^{-\frac{1}{2}} \mathrm{d}z \lesssim K \ln \brBig{(N_1 \lor N_2)^\frac{5}{2} + 1} \lesssim_{\epsilon} K (N_1 \lor N_2)^{2 \epsilon} \] and this concludes the discussion of case (i).

(ii) \underline{$q^2 - 3 \xi^2 > 0$ and $ c < 0$}: We divide this case into two subcases.
 
(ii.1) \underline{$K + c \geq 0$}: In this situation we easily verify \[ D_1 \geq 0 \Leftrightarrow \abs{q_1 + h} \geq \brBig{\frac{3\xi \abs{c}}{q^2-3\xi^2}}^{\frac{1}{2}} \ \text{and} \ D_2 \geq 0 \ \forall q_1 \in \Z, \]
allowing us to conclude \[ \abs{B_{\tau, \xi, q}^{\mathrm{lin}}} \lesssim \sum_{\brbig{\frac{3\xi \abs{c}}{q^2-3\xi^2}}^{\frac{1}{2}} \leq \abs{q_1 + h} \lesssim (N_1 \lor N_2)} a(q_1) \ + \sum_{\abs{q_1 + h} < \brbig{\frac{3\xi \abs{c}}{q^2-3\xi^2}}^{\frac{1}{2}}} b(q_1) \eqcolon (I) + (II), \] 
where $a(q_1)$ is defined as in case (i), and \[ b(q_1) \coloneq \absBig{\brBig{-\frac{(q_{1}+h)q}{3 \xi} + \intcc[\Big]{- D_{2}^{\frac{1}{2}},  D_{2}^{\frac{1}{2}}}} \cap \intoo{-\frac{\xi}{2},\frac{\xi}{2}}}. \] We first turn our attention to contribution $(II)$. If $\xi \leq 1$, then $b(q_1) \leq \xi \leq 1$, from which it follows that \[ (II) \leq \sum_{\abs{q_1 + h} < \brbig{\frac{3\xi \abs{c}}{q^2-3\xi^2}}^{\frac{1}{2}}} 1  \lesssim \brbigg{\brBig{\frac{3\xi \abs{c}}{q^2-3\xi^2}}^{\frac{1}{2}} + 1} \lesssim K^{\frac{1}{2}} \leq (N_1 \lor N_2)^{2\epsilon} K , \] where the final step relies on the inequalities $\xi \leq 1$, $q^2-3\xi^2 \gtrsim 1$ and $\abs{c} \leq K \geq 1$. In the case $\xi > 1$, we have
\[ b(q_1) \leq \absBig{\intcc[\Big]{- D_{2}^{\frac{1}{2}},  D_{2}^{\frac{1}{2}}}} = 2 \brBig{\frac{q^2-3\xi^2}{9\xi^2}(q_1 + h)^2 + \frac{K+c}{3\xi}}^{\frac{1}{2}} \leq 2 \frac{K^\frac{1}{2}}{(3\xi)^\frac{1}{2}} \lesssim K^{\frac{1}{2}}, \] with the second-to-last step justified by $ \abs{q_1 + h} < \brbig{\frac{3\xi \abs{c}}{q^2-3\xi^2}}^{\frac{1}{2}} $. This, together with $\xi >1$, $q^2-3\xi^2 \gtrsim 1$ and $\abs{c} \leq K \geq 1$,   leads to \[ (II) \lesssim \frac{K^\frac{1}{2}}{\xi^{\frac{1}{2}}} \sum_{\abs{q_1 + h} < \brbig{\frac{3\xi \abs{c}}{q^2-3\xi^2}}^{\frac{1}{2}}} 1 \lesssim \frac{K^\frac{1}{2}}{\xi^{\frac{1}{2}}} \brbigg{\brBig{\frac{3\xi \abs{c}}{q^2-3\xi^2}}^{\frac{1}{2}} + 1} \lesssim K  \leq (N_1 \lor N_2)^{2\epsilon} K, \] so that we have, in total, verified that \[ (II) \lesssim (N_1 \lor N_2)^{2 \epsilon} K. \] Let us now consider $(I)$. By the same reasoning as in case (i), we can restrict ourselves to the case $\xi > \frac{1}{(N_1 \lor N_2)}$. It is immediately clear that \[ a(q_1) \leq \begin{cases} \xi \leq 1 \leq K^\frac{1}{2} & \text{if} \ \xi \leq 1 \\ \absBig{\intcc[\Big]{D_{1}^{\frac{1}{2}},D_{2}^{\frac{1}{2}}}} \lesssim \frac{K^\frac{1}{2}}{\xi^\frac{1}{2}} \lesssim K^\frac{1}{2} & \text{if} \ \xi > 1  \end{cases} \] holds for all the values of $q_1$ occurring in the sum, so that we can write
\begin{align*} (I) & \lesssim K^\frac{1}{2} + \sum_{\brbig{\frac{3\xi \abs{c}}{q^2-3\xi^2}}^{\frac{1}{2}} + 1 \leq \abs{q_1 + h} \lesssim (N_1 \lor N_2)} \absBig{\intcc[\Big]{D_{1}^{\frac{1}{2}},D_{2}^{\frac{1}{2}}}} \\ & \lesssim K^\frac{1}{2} + \frac{K}{\xi^\frac{1}{2}} \sum_{\brbig{\frac{3\xi \abs{c}}{q^2-3\xi^2}}^{\frac{1}{2}} + 1 \leq \abs{q_1 + h} \lesssim (N_1 \lor N_2)} \brBig{\frac{q^2-3\xi^2}{3 \xi}(q_1 + h)^2 + K - \abs{c}}^{-\frac{1}{2}}.  \end{align*}
An integral comparison then leads to \begin{align*} ... & \lesssim K^\frac{1}{2} + \frac{K}{\xi^\frac{1}{2}} \int_{\brbig{\frac{3\xi \abs{c}}{q^2-3\xi^2}}^\frac{1}{2}}^{\tilde{c} (N_1 \lor N_2)} \brBig{\frac{q^2-3\xi^2}{3 \xi}y^2 + K - \abs{c}}^{-\frac{1}{2}} \mathrm{d}y \\ & \sim K^\frac{1}{2} + \frac{K}{(q^2-3\xi^2)^\frac{1}{2}} \int_{\abs{c}^\frac{1}{2}}^{\tilde{c}  \brbig{\frac{q^2-3\xi^2}{3 \xi}}^\frac{1}{2} (N_1 \lor N_2)} (z^2+ K- \abs{c})^{-\frac{1}{2}} \mathrm{d}z \\ & \lesssim  K^\frac{1}{2} + \frac{K}{(q^2-3\xi^2)^\frac{1}{2}} \int_{0}^{\tilde{c}  \brbig{\frac{q^2-3\xi^2}{3 \xi}}^\frac{1}{2} (N_1 \lor N_2) - \abs{c}^\frac{1}{2}} (u^2+ 1)^{-\frac{1}{2}} \mathrm{d}u,   \end{align*}
where we substituted $z = \brbig{\frac{q^2-3\xi^2}{3 \xi}}^\frac{1}{2}y$, followed by $u = z - \abs{c}^\frac{1}{2}$, and taking $ 1 \lesssim q^2-3\xi^2 \lesssim (N_1 \lor N_2)$, $ \xi > \frac{1}{(N_1 \lor N_2)}$ and $K \geq 1$ into account once again, we obtain - just as in (i) - that \[ (I) \lesssim_{\epsilon} (N_1 \lor N_2)^{2 \epsilon} K. \]

(ii.2) \underline{$K + c < 0$}: Compared to (ii.1), the only difference in this case is that \[ D_2 \geq 0 \Leftrightarrow \abs{q_1 + h} \geq \brBig{\frac{3 \xi (\abs{c} - K)}{q^2-3\xi^2}}^\frac{1}{2}. \] We thus obtain
\[ \abs{B_{\tau, \xi, q}^{\mathrm{lin}}} \lesssim \sum_{\brbig{\frac{3\xi \abs{c}}{q^2-3\xi^2}}^{\frac{1}{2}} \leq \abs{q_1 + h} \lesssim (N_1 \lor N_2)} a(q_1) \ + \sum_{\brbig{\frac{3 \xi (\abs{c} - K)}{q^2-3\xi^2}}^\frac{1}{2} \leq \abs{q_1 + h} < \brbig{\frac{3\xi \abs{c}}{q^2-3\xi^2}}^{\frac{1}{2}}} b(q_1) \] and the first sum can be estimated in exactly the same way as in (ii.1), since the estimates did not rely on $ K + c \geq 0$.
Moreover, as in (ii.1), we may again rely on $b(q_1) \leq \xi$ and $b(q_1) \lesssim \frac{K^\frac{1}{2}}{\xi^\frac{1}{2}}$, which yields \begin{align*} \sum_{\brbig{\frac{3 \xi (\abs{c} - K)}{q^2-3\xi^2}}^\frac{1}{2} \leq \abs{q_1 + h} < \brbig{\frac{3\xi \abs{c}}{q^2-3\xi^2}}^{\frac{1}{2}}} b(q_1) &\lesssim \xi \brbigg{\brBig{\frac{3\xi \abs{c}}{q^2-3\xi^2}}^{\frac{1}{2}} - \brBig{\frac{3 \xi (\abs{c} - K)}{q^2-3\xi^2}}^\frac{1}{2} + 1} \\ & \lesssim K^\frac{1}{2} \lesssim (N_1 \lor N_2)^{2 \epsilon} K  \end{align*} in the case $\xi \leq 1$ and 
\begin{align*} \sum_{\brbig{\frac{3 \xi (\abs{c} - K)}{q^2-3\xi^2}}^\frac{1}{2} \leq \abs{q_1 + h} < \brbig{\frac{3\xi \abs{c}}{q^2-3\xi^2}}^{\frac{1}{2}}} b(q_1) &\lesssim \frac{K^\frac{1}{2}}{\xi^\frac{1}{2}} \brbigg{\brBig{\frac{3\xi \abs{c}}{q^2-3\xi^2}}^{\frac{1}{2}} - \brBig{\frac{3 \xi (\abs{c} - K)}{q^2-3\xi^2}}^\frac{1}{2} + 1} \\ & \lesssim K \lesssim (N_1 \lor N_2)^{2 \epsilon} K \end{align*} in the case $\xi > 1$. Here, the inequalities $q^2-3\xi^2 \gtrsim 1$ and $K < \abs{c}$ were employed in both cases.

(iii) \underline{$q^2 - 3 \xi^2 < 0$ and $ c > 0$}: We observe that in this case, \[ D_1 \geq 0 \Leftrightarrow \abs{q_1+h} \leq \brBig{\frac{3\xi c}{3\xi^2-q^2}}^\frac{1}{2} \ \text{and} \ D_2 \geq 0 \Leftrightarrow \abs{q_1+h} \leq \brBig{\frac{3\xi (c+K)}{3\xi^2-q^2}}^\frac{1}{2}  \] hold. With $a(q_1)$ and $b(q_1)$ as before, we thus obtain
\[ \abs{B_{\tau, \xi, q}^{\mathrm{lin}}} \lesssim \sum_{\abs{q_1 + h} \leq \brbig{\frac{3 \xi c}{3\xi^2-q^2}}^\frac{1}{2}} a(q_1) \ + \sum_{\brbig{\frac{3 \xi c}{3\xi^2-q^2}}^\frac{1}{2} < \abs{q_1 + h} \leq \brbig{\frac{3\xi (c+K)}{3\xi^2-q^2}}^{\frac{1}{2}}} b(q_1) \] and again, the two sums are to be estimated one after the other.
First, we note that for the $q_1$, over which the sums extend, we have $a(q_1) \leq \xi$, $a(q_1) \lesssim \frac{K^\frac{1}{2}}{\xi^\frac{1}{2}}$, $b(q_1) \leq \xi$ and $b(q_1) \lesssim \frac{K^\frac{1}{2}}{\xi^\frac{1}{2}}$ at our disposal. If we denote the second sum by $(II)$, it follows that 
\[ (II) \lesssim \xi \brbigg{\brBig{\frac{3\xi (c+K)}{3\xi^2-q^2}}^\frac{1}{2} - \brBig{\frac{3\xi c}{3\xi^2-q^2}}^\frac{1}{2} + 1} \lesssim K^\frac{1}{2} \leq (N_1 \lor N_2)^{2 \epsilon} K \] in the case $\xi \leq 1$, and 
\[ (II) \lesssim \frac{K^\frac{1}{2}}{\xi^\frac{1}{2}} \brbigg{\brBig{\frac{3\xi (c+K)}{3\xi^2-q^2}}^\frac{1}{2} - \brBig{\frac{3\xi c}{3\xi^2-q^2}}^\frac{1}{2} + 1} \lesssim K \leq (N_1 \lor N_2)^{2 \epsilon} K \] in the case $\xi > 1$. In both estimates, we made use of $3\xi^2 - q^2 \gtrsim 1$ and $c > 0$.
For the first sum $(I)$, we aim once again to use an integral comparison. Due to our bounds for $a(q_1)$, we have \begin{align*} (I) &\lesssim K^\frac{1}{2} + \sum_{\abs{q_1 + h} \leq \brbig{\frac{3 \xi c}{3\xi^2-q^2}}^\frac{1}{2}-1} \absBig{\intcc[\Big]{D_{1}^{\frac{1}{2}},D_{2}^{\frac{1}{2}}}} \\ & \lesssim K^\frac{1}{2} + \frac{K}{\xi^\frac{1}{2} c^\frac{1}{2}} \sum_{\abs{q_1 + h} \leq \brbig{\frac{3 \xi c}{3\xi^2-q^2}}^\frac{1}{2}-1} \brBig{1 - \frac{3\xi^2-q^2}{3\xi c} (q_1 + h)^2}^{-\frac{1}{2}} \\ & \lesssim K^\frac{1}{2} + \frac{K}{\xi^\frac{1}{2} c^\frac{1}{2}} \int_{0}^{\brbig{\frac{3 \xi c}{3\xi^2-q^2}}^\frac{1}{2}} \brBig{1-\frac{3\xi^2-q^2}{3\xi c}y^2}^{-\frac{1}{2}} \mathrm{d} y \\ & \lesssim K^\frac{1}{2} + \frac{K}{(3\xi^2-q^2)^\frac{1}{2}} \underbrace{\int_{0}^{1} (1-z^2)^{-\frac{1}{2}} \mathrm{d} z}_{= \frac{\pi}{2}} \\ &\lesssim (N_1 \lor N_2)^{2 \epsilon} K, \end{align*}
making use of $3\xi^2-q^2 \gtrsim 1$ in the final step. With that, case (iii) is also settled.

(iv) \underline{$q^2 - 3 \xi^2 < 0$ and $ c \leq 0$}: In this final scenario, we have \[ D_1 \leq 0 \ \forall q_1 \in \Z \ \text{and} \ D_2 \geq 0 \Leftrightarrow \abs{q_1 + h} \leq \brBig{\frac{3\xi (K- \abs{c})}{3\xi^2-q^2}}^\frac{1}{2}, \] although the latter holds only under the additional assumption $ \abs{c} \leq K$. With $b(q_1)$ as previously defined, it follows that \[ \abs{B_{\tau, \xi, q}^{\mathrm{lin}}} \lesssim \sum_{\abs{q_1 + h} \leq \brbig{\frac{3 \xi (K-\abs{c})}{3\xi^2-q^2}}^\frac{1}{2}} b(q_1).  \]
For all $q_1$ appearing in the sum, we again have $b(q_1) \leq \xi$ and $b(q_1) \lesssim \frac{K^\frac{1}{2}}{\xi^{\frac{1}{2}}}$, so that we obtain \[ ... \lesssim \xi \brbigg{\brBig{\frac{3 \xi (K-\abs{c})}{3\xi^2-q^2}}^\frac{1}{2} + 1} \lesssim K^\frac{1}{2} \leq (N_1 \lor N_2)^{2 \epsilon} K \] in the case $\xi \leq 1$ and \[ ... \lesssim \frac{K^\frac{1}{2}}{\xi^\frac{1}{2}} \brbigg{\brBig{\frac{3 \xi (K-\abs{c})}{3\xi^2-q^2}}^\frac{1}{2} + 1} \lesssim K \leq (N_1 \lor N_2)^{2 \epsilon} K \] in the case $\xi > 1$.
In these last two estimates, we made use of $3\xi^2 - q^2 \gtrsim 1$ and thus the proof in the final case is also complete.

Now, collecting the results from cases (i) to (iv), we obtain \[ \abs{B_{\tau, \xi, q}^{\mathrm{lin}}}^\frac{1}{2} \lesssim_{\epsilon} (N_1 \lor N_2)^{\epsilon} K^\frac{1}{2} \] and taking the supremum on both sides yields assertion \eqref{lin}.

If we now wish to prove \eqref{bilina}, we need to take into account that we can no longer rely on $a(q_1), b(q_1) \leq \xi$. However, a careful examination of the proof of \eqref{lin} shows that this is not problematic, since \begin{enumerate} \item[(A)] we still have $a(q_1), b(q_1) \lesssim \frac{K^\frac{1}{2}}{\xi^\frac{1}{2}}$ at our disposal, and \item[(B)] the integral estimates in the cases (i) and (ii.1) work under the assumption $\xi \gtrsim 1$, which is already included in the definition of $B_{\tau, \xi, q}^{\alpha}$. \end{enumerate}
We thus obtain \[ \abs{B_{\tau,\xi,q}^\alpha} \lesssim_{\epsilon}  \frac{K}{\abs{3\xi^2-q^2}^\frac{1}{2}} (N_1 \lor N_2)^{2 \epsilon} + \frac{K^\frac{1}{2}}{\xi^\frac{1}{2}} \] and multiplication with $\xi^{\frac{\alpha}{2}}$ gives \[ \xi^{\frac{\alpha}{2}} \cdot \abs{B_{\tau,\xi,q}^\alpha} \lesssim_{\epsilon} \frac{\xi^{\frac{\alpha}{2}} K}{\abs{3\xi^2-q^2}^\frac{1}{2}} (N_1 \lor N_2)^{2 \epsilon} + \frac{K^\frac{1}{2}}{\xi^\frac{1-\alpha}{2}}.  \]
Since we now have $\abs{3\xi^2-q^2} \gtrsim \xi^\alpha$, $\xi \gtrsim 1$ and $\alpha \in \intcc{0,1}$, it follows that\[ \xi^{\frac{\alpha}{2}} \cdot \abs{B_{\tau,\xi,q}^\alpha} \lesssim_{\epsilon} (N_1 \lor N_2)^{2 \epsilon} K \] and taking square roots, followed by forming the supremum yields the claimed result. 
\end{proof}

With the proof of Lemma \ref{mes}, we have now completed the necessary groundwork to prove  the first main result of this paper.

\begin{proof}[Proof of Theorem \ref{L4}]

We fix arbitrary positive numbers $\epsilon > 0$ and $\epsilon' > 0$ and consider $u \in X_{\epsilon, \frac{1}{2}+\epsilon'}$.
Since we want to establish a linear estimate and the phase function $\phi(\xi,q) = \xi(\xi^2+q^2)$ is odd, we may assume without loss of generality that $\text{supp}(\hat{u}) \subseteq \R \times \mathbb{R}^{>0} \times \mathbb{Z}$ for the remainder of the proof.
We begin by noting that \eqref{ineq1} follows from \begin{equation} \begin{aligned} &\norm{(P_{N_1}Q_{L_1}u)(P_{N_2}Q_{L_2}u)}_{L_{txy}^2} \\ &\lesssim_{\epsilon} (N_1 \lor N_2)^{\frac{\epsilon}{2}} (L_1 \land L_2)^{\frac{1}{2}} (L_1 \lor L_2)^{\frac{1}{2}}  \norm{P_{N_1}Q_{L_1}u}_{L_{txy}^2} \norm{P_{N_2}Q_{L_2}u}_{L_{txy}^2} \end{aligned} \label{bilineq1} \end{equation} via dyadic summation, as this implies 
\begin{align*} &\norm{u}_{L_{txy}^4}^2 = \norm{u^2}_{L_{txy}^2} \\ & \lesssim \sum_{N_i, L_i} \norm{(P_{N_1}Q_{L_1}u)(P_{N_2}Q_{L_2}u)}_{L_{txy}^2} \\ & \lesssim_{\epsilon, \epsilon'} \sum_{N_i, L_i} N_1^{-\frac{\epsilon}{2}} N_2^{-\frac{\epsilon}{2}}L_1^{-\epsilon'}L_2^{-\epsilon'}N_1^{\epsilon}N_2^{\epsilon}L_1^{\frac{1}{2}+\epsilon'}L_2^{\frac{1}{2}+\epsilon'} \norm{P_{N_1}Q_{L_1}u}_{L_{txy}^2} \norm{P_{N_2}Q_{L_2}u}_{L_{txy}^2} \\ & \lesssim \brBigg{\sum_{N_i,L_i} N_1^{-\frac{\epsilon}{2}} N_2^{-\frac{\epsilon}{2}}L_1^{-\epsilon'}L_2^{-\epsilon'}} \norm{u}_{X_{\epsilon,\frac{1}{2}+\epsilon'}}^2 \lesssim \norm{u}_{X_{\epsilon,\frac{1}{2}+\epsilon'}}^2. \end{align*}
We therefore turn to the proof of \eqref{bilineq1}. By introducing the following two Fourier projectors \[ P_{\set{\abs{3\xi^2-q^2} > 1}}u \coloneqq \Fcal_{xy}^{-1} \chi_{\set{\abs{3\xi^2-q^2} > 1}} \Fcal_{xy}u \] and \[ P_{\set{\abs{3\xi^2-q^2} \leq 1}}u \coloneqq \Fcal_{xy}^{-1} \chi_{\set{\abs{3\xi^2-q^2} \leq  1}} \Fcal_{xy}u, \] we separate the left-hand side of \eqref{bilineq1} into two components, namely \[ (I) \coloneqq \norm{P_{\set{\abs{3\xi^2-q^2} > 1}}((P_{N_1}Q_{L_1}u)(P_{N_2}Q_{L_2}u))}_{L_{txy}^2} \] and
\[ (II) \coloneqq \norm{P_{\set{\abs{3\xi^2-q^2} \leq 1}}((P_{N_1}Q_{L_1}u)(P_{N_2}Q_{L_2}u))}_{L_{txy}^2}, \] both of which can now be suitably estimated. Using Parseval's identity, the Cauchy-Schwarz inequality, and Fubini's theorem, we obtain \[ (I) \lesssim \brBigg{\sup_{(\tau,\xi,q) \in \R \times \R^{>0} \times \mathbb{Z}} \abs{A_{\tau,\xi,q}}^{\frac{1}{2}}} \norm{P_{N_1}Q_{L_1}u}_{L_{txy}^2} \norm{P_{N_2}Q_{L_2}u}_{L_{txy}^2} \]
with \begin{align*} &A_{\tau,\xi,q} \coloneqq  \setbig{ (\tau_1,\xi_1,q_1) \in \R \times \R^{>0} \times \Z \ \big| \ \xi - \xi_1 > 0, \ \abs{3\xi^2-q^2} > 1, \ \abs{(\xi_1,q_1)} \lesssim N_1, \\ &  \abs{(\xi-\xi_1,q-q_1)} \lesssim N_2, \ \abs{\tau_1-\phi(\xi_1,q_1)} \lesssim L_1, \ \abs{\tau-\tau_1 - \phi(\xi-\xi_1,q-q_1)} \lesssim L_2}. \end{align*}
Furthermore, the triangle inequality gives \[ \abs{\tau - \phi(\xi_1,q_1) - \phi(\xi-\xi_1,q-q_1)} \leq \abs{\tau_1-\phi(\xi_1,q_1)} + \abs{\tau-\tau_1 - \phi(\xi-\xi_1,q-q_1)} \lesssim (L_1 \lor L_2), \]
allowing us to deduce \[ \abs{A_{\tau,\xi,q}}^{\frac{1}{2}} \lesssim (L_1 \land L_2)^{\frac{1}{2}} \abs{B_{\tau,\xi,q}}^\frac{1}{2} \] with 
\begin{align*} &B_{\tau,\xi,q} \coloneqq \setbig{ (\xi_1,q_1) \in \R^{>0} \times \Z \ \big| \ \xi - \xi_1 > 0, \ \abs{3\xi^2-q^2} > 1, \ \abs{(\xi_1,q_1)} \lesssim N_1, \\ &  \abs{(\xi-\xi_1,q-q_1)} \lesssim N_2, \ \abs{\tau-\phi(\xi_1,q_1)-\phi(\xi-\xi_1,q-q_1)} \lesssim (L_1 \lor L_2)}. \end{align*}
We now define the function \[ h: \Z \rightarrow \set{0,\frac{1}{2}}, \quad q \mapsto h(q) \coloneqq \begin{cases} 0 & \text{if} \ q \ \text{is even} \\ \frac{1}{2} & \text{if} \ q \ \text{is odd} \end{cases} \] and introduce new variables \[ \tilde{\xi_1} \coloneqq \xi_1-\frac{\xi}{2} \quad \text{and} \quad \tilde{q_1} \coloneqq q_1-\frac{q}{2} - h(q). \]
This leads to \[ \xi_1^3 + (\xi-\xi_1)^3 = \brbigg{\tilde{\xi_1} + \frac{\xi}{2}}^3 + \brbigg{\frac{\xi}{2}-\tilde{\xi_1}}^3 = \frac{\xi^3}{4} + 3\xi \tilde{\xi_1}^2, \] 
\begin{align*} \xi_1 q_1^2 = \brbigg{\tilde{\xi_1} + \frac{\xi}{2}}\brbigg{\frac{q}{2} + \tilde{q_1} + h(q)}^2 &= \tilde{\xi_1}(\tilde{q_1}+h(q))^2 +  \tilde{\xi_1} q (\tilde{q_1} + h(q)) + \tilde{\xi_1} \frac{q^2}{4} \\ &  \ \ \ + \frac{\xi}{2} (\tilde{q_1} + h(q))^2 + \frac{\xi}{2} q (\tilde{q_1} + h(q)) + \frac{\xi q^2}{8} \end{align*}
and \begin{align*} (\xi - \xi_1) (q-q_1)^2 &= \brbigg{\frac{\xi}{2} - \tilde{\xi_1}} \brbigg{\frac{q}{2} - (\tilde{q_1} + h(q))}^2 \\ &= \frac{\xi q^2}{8} - \frac{\xi}{2}q(\tilde{q_1} + h(q)) + \frac{\xi}{2}(\tilde{q_1} + h(q))^2 - \tilde{\xi_1} \frac{q^2}{4} + \tilde{\xi_1} q (\tilde{q_1} + h(q)) \\ & \ \ \ - \tilde{\xi_1}(\tilde{q_1} + h(q))^2, \end{align*}
and we thus obtain \[  \phi(\xi_1,q_1) + \phi(\xi - \xi_1,q-q_1) =  \underbrace{\xi(3\tilde{\xi_1}^2 + (\tilde{q_1} + h(q))^2) + 2q(\tilde{q_1} + h(q))\tilde{\xi_1}}_{= p(\tilde{\xi_1},\tilde{q_1} + h(q))} + \frac{\xi}{4}(\xi^2+q^2).	  \]
With this, we arrive at \begin{align*} &\abs{\tau - \phi(\xi_1,q_1) - \phi(\xi - \xi_1,q-q_1)} \leq \tilde{c} (L_1 \lor L_2) \Leftrightarrow \\ & -\tilde{c}(L_1 \lor L_2) + \tau - \frac{\xi}{4}(\xi^2+q^2) \leq p(\tilde{\xi_1},\tilde{q_1} + h(q)) \leq \tilde{c} (L_1 \lor L_2) + \tau - \frac{\xi}{4}(\xi^2+q^2), \end{align*} and upon setting \[ c(\tau,\xi,q) \coloneqq -\tilde{c}(L_1 \lor L_2) + \tau - \frac{\xi}{4}(\xi^2+q^2) \quad \text{and} \quad K \coloneqq 2\tilde{c}(L_1 \lor L_2), \] the substitution of the new variables, together with the translation invariance of the counting measure, yields
\[ \abs{B_{\tau,\xi,q}}^{\frac{1}{2}} = \abs{B_{\tau,\xi,q}^{\mathrm{lin}}(c,K)}^{\frac{1}{2}}, \] with $B_{\tau,\xi,q}^{\mathrm{lin}} = B_{\tau,\xi,q}^{\mathrm{lin}}(c,K)$ defined exactly as in Lemma \ref{mes}. Thus, we may now invoke the previously established estimate \eqref{lin}, from which we obtain \[ \abs{A_{\tau,\xi,q}}^{\frac{1}{2}} \lesssim_{\epsilon} (L_1 \land L_2)^\frac{1}{2} (L_1 \lor L_2)^\frac{1}{2} (N_1 \lor N_2)^\frac{\epsilon}{2} \] and this concludes the proof of \eqref{bilineq1} in the case $\abs{3\xi^2-q^2} >1$.
In order to control $(II)$, we proceed by duality and write \[ (II) \sim \sup_{\substack{f\in L_{txy}^2 \\ \norm{f}_{L_{txy}^2}\leq 1}} \abs[\bigg]{\int_{\R^2 \times \mathbb{T}} P_{\set{\abs{3\xi^2-q^2} \leq 1}}((P_{N_1}Q_{L_1}u)(P_{N_2}Q_{L_2}u)) \cdot \overline{f} \mathrm{d} (t,x,y)}. \]
Then, applying Parseval's identity, the Cauchy-Schwarz inequality and Fubini's theorem, we again obtain \[ ... \lesssim \brBigg{\sup_{(\tau_1,\xi_1,q_1) \in \R \times \R^{>0} \times \Z} \abs{\tilde{A}_{\tau_1,\xi_1,q_1}}^\frac{1}{2}} \norm{P_{N_1}Q_{L_1}u}_{L_{txy}^2} \norm{P_{N_2}Q_{L_2}u}_{L_{txy}^2} \underbrace{\norm{f}_{L_{txy}^2}}_{\leq 1} \]
with   \begin{align*} &\tilde{A}_{\tau_1,\xi_1,q_1} \coloneqq  \setbig{ (\tau,\xi,q) \in \R \times \R^{>0} \times \Z \ \big| \ \xi - \xi_1 > 0, \ \abs{3\xi^2-q^2} \leq  1, \ \abs{(\xi_1,q_1)} \lesssim N_1, \\ &  \abs{(\xi-\xi_1,q-q_1)} \lesssim N_2, \ \abs{\tau_1-\phi(\xi_1,q_1)} \lesssim L_1, \ \abs{\tau-\tau_1 - \phi(\xi-\xi_1,q-q_1)} \lesssim L_2}, \end{align*} and from the definition of $\tilde{A}_{\tau_1,\xi_1,q_1}$, we further deduce \begin{align*} \abs{\tilde{A}_{\tau_1,\xi_1,q_1}}^\frac{1}{2} &\lesssim (L_1 \lor L_2)^\frac{1}{2} \brBigg{ \sum_{\abs{q} \lesssim (N_1 \lor N_2)} \int_{\R} \chi_{\set{\abs{3\xi^2-q^2} \leq 1}} \mathrm{d} \xi }^\frac{1}{2} \\ & \lesssim (L_1 \lor L_2)^\frac{1}{2} \brBigg{ 1 + \sum_{\substack{\abs{q} \lesssim (N_1 \lor N_2) \\ q \neq 0} }\frac{1}{\abs{q}}}^\frac{1}{2} \\ & \lesssim (L_1 \lor L_2)^\frac{1}{2} \ln(\tilde{c} (N_1 \lor N_2))^\frac{1}{2} \\ & \lesssim_{\epsilon} (L_1 \lor L_2)^\frac{1}{2} (N_1 \lor N_2)^\frac{\epsilon}{2} \underbrace{(L_1 \land L_2)^\frac{1}{2}}_{\geq 1}.  \end{align*}
Thus, \eqref{bilineq1} is also established in the case $\abs{3\xi^2 - q^2} \leq 1$, completing the proof of Theorem \ref{L4}.
\end{proof}

\begin{rem}

If we interpolate \eqref{ineq1} with the trivial estimate \eqref{trivial L2}, we obtain \begin{equation} \label{ineq1d} \norm{u}_{L_{txy}^{4-}} \lesssim \norm{u}_{X_{0+,\frac{1}{2}-}} \end{equation} and this shows that we can essentially rely on \eqref{ineq1}, even if $b$ happens to be slightly smaller than $\frac{1}{2}$.

\end{rem}

Given the lack of dispersion in the $\mathbb{T}$-component, it is not surprising that such a linear $L^4$-estimate cannot achieve much more than the loss of an $\epsilon$-derivative. This is captured in the following proposition, which demonstrates the near-optimality of \eqref{ineq1}.

\begin{prop}
\label{failprop}

The estimate \begin{equation} \label{fail} \norm{u}_{L_{txy}^4} \lesssim \norm{u}_{X_{s,b}} \end{equation} fails for all $s<0$ and $b \in \R$.

\end{prop}

\begin{proof}

Let $s<0$ and $b\in \R$ be arbitrary. For $N \in \N$, we define the following sequence of functions via its Fourier transform: \[ \hat{u}_N(\tau,\xi,q) \coloneqq (\delta_{q,N}+\delta_{q,-N})\chi_{\intcc{-1,1}}(\xi) \chi_{\intcc{-1,1}}(\tau - \phi(\xi,q)). \]
Here, $\delta_{q,N}$ denotes the Kronecker delta.
A straightforward computation yields \begin{equation} \label{right} \norm{u_N}_{X_{s,b}}^2 \lesssim_{b} N^{2s} \end{equation} and an application of Parseval's identity gives \begin{align*} &\norm{u_N}_{L_{txy}^4}^2 = \norm{u_N^2}_{L_{txy}^2} \sim \norm{\hat{u}_{N} \ast \hat{u}_N}_{L_{\tau \xi q}^2} \\ &\geq \norm{\underbrace{(\delta_{q,N} \chi_{\intcc{-1,1}}(\xi) \chi_{\intcc{-1,1}}(\tau-\phi(\xi,q))) \ast (\delta_{q,-N}\chi_{\intcc{-1,1}}(\xi) \chi_{\intcc{-1,1}}(\tau-\phi(\xi,q)))}_{\eqcolon I}}_{L_{\tau \xi q}^2}, \end{align*}
where $\ast$ denotes the convolution with respect to $\tau$, $\xi$ and $q$. For the convolution integral $I$, we now have \begin{align*} I &= \delta_{q,0} \int_{\R} \chi_{\intcc{-1,1}}(\xi_1) \chi_{\intcc{-1,1}}(\xi-\xi_1) \int_{\R} \chi_{\intcc{-1,1}}(\tau_1 - \phi(\xi_1,N)) \\ & \ \ \ \cdot \chi_{\intcc{-1,1}}(\tau-\tau_1 - \phi(\xi-\xi_1,-N)) \mathrm{d} \tau_1 \mathrm{d} \xi_1 \\ & = \delta_{q,0} \int_{\R} \chi_{\intcc{-1,1}}(\xi_1) \chi_{\intcc{-1,1}}(\xi-\xi_1) \int_{\R} \chi_{\intcc{-1,1}}(\tilde{\tau_1}) \\ & \ \ \ \cdot \chi_{\intcc{-1,1}}(\tilde{\tau_1} - (\tau - \phi(\xi_1,N) - \phi(\xi-\xi_1,-N))) \mathrm{d} \tilde{\tau_1} \mathrm{d} \xi_1 \\ & \geq \delta_{q,0} \int_{\R} \chi_{\intcc{-1,1}}(\xi_1) \chi_{\intcc{-1,1}}(\xi-\xi_1) \chi_{\intcc{-1,1}}(\tau - \phi(\xi_1,N) - \phi(\xi-\xi_1,-N)) \mathrm{d} \xi_1 \\ & \geq \delta_{q,0} \chi_{\intcc{-\frac{1}{2},\frac{1}{2}}}(\xi) \int_\R  \chi_{\intcc{-\frac{1}{2},\frac{1}{2}}}(\xi_1) \chi_{\intcc{-1,1}}(\tau - \phi(\xi_1,N) - \phi(\xi-\xi_1,-N)) \mathrm{d} \xi_1, \end{align*} and \begin{align*} \tau - \phi(\xi_1,N)- \phi(\xi-\xi_1,-N) &= \tau - \xi_1^3-\xi_1N^2-(\xi-\xi_1)^3-(\xi-\xi_1)N^2 \\ & = \tau - \xi N^2 \underbrace{- \xi^3 + 3\xi^2\xi_1 - 3 \xi \xi_1^2}_{\in \intcc{-\frac{7}{8},\frac{7}{8}}} \end{align*}
allows us to conclude that \[ \chi_{\intcc{-1,1}}(\tau - \phi(\xi_1,N)- \phi(\xi - \xi_1,-N)) \geq \chi_{\intcc{-\frac{1}{8},\frac{1}{8}}}(\tau - \xi N^2) \] holds true. We thus have \[ I \geq \delta_{q,0} \chi_{\intcc{-\frac{1}{2},\frac{1}{2}}}(\xi) \chi_{\intcc{-\frac{1}{8},\frac{1}{8}}}(\tau - \xi N^2) \] and forming the $L_{\tau \xi q}^2$-norm yields \begin{equation} \label{left} \norm{u_N}_{L_{txy}^4}^2 \gtrsim \brbigg{ \int_{-\frac{1}{2}}^{\frac{1}{2}} \int_{-\frac{1}{8}+ \xi N^2}^{\frac{1}{8} + \xi N^2} 1 \mathrm{d} \tau \mathrm{d} \xi }^\frac{1}{2} = \frac{1}{2} \gtrsim 1. \end{equation}
Now, if \eqref{fail} were true, then combining \eqref{right} and \eqref{left} would lead to \[ 1 \lesssim \norm{u_N}_{L_{txy}^4} \lesssim \norm{u_N}_{X_{s,b}} \lesssim N^s \quad \forall N \in \N, \] which would contradict the assumption $s < 0$. This completes the proof.
\end{proof}

Although, according to Proposition \ref{failprop}, we cannot expect a gain in derivatives in a linear $L^4$-estimate, it is possible - using Lemma \ref{mes} - to establish a parameter-dependent family of bilinear estimates that will allow us to gain up to $\frac{1}{4}-$ of a derivative in certain frequency ranges. These estimates can be understood as bilinear refinements of the linear $L^4$-estimate \eqref{ineq1}.

\begin{prop}

Let $\epsilon > 0$ and $b > \frac{1}{2} $ be arbitrary. Furthermore, let the Fourier projector $P^\alpha$ for $\alpha \in \intcc{0,1}$ be defined by \[ P^\alpha u \coloneqq \Fcal_{xy}^{-1} \chi_{\set{\abs{3\xi^2-q^2} \gtrsim \abs{\xi}^\alpha, \abs{\xi} \gtrsim 1}} \Fcal_{xy}u. \]
Then, the estimate \begin{equation} \label{bilinrefinement} \norm{I_{x}^{\frac{\alpha}{4}}P^\alpha (u v)}_{L_{txy}^2} \lesssim_{\epsilon,b} \norm{u}_{X_{\epsilon,b}} \norm{v}_{X_{\epsilon,b}} \end{equation} holds for all $u,v \in X_{\epsilon,b}$.

\end{prop}

\begin{proof}
Since the phase function is odd, we may assume without loss of generality that $\xi > 0$.
Analogous to the argument in the proof of Theorem \ref{L4}, the estimate \eqref{bilinrefinement} to be established, follows from \begin{equation} \begin{aligned}  & \norm{I_{x}^\frac{\alpha}{4}P^\alpha((P_{N_1} Q_{L_1} u)(P_{N_2} Q_{L_2} v))}_{L_{txy}^2} \\ &\lesssim_{\epsilon}  (N_1 \lor N_2)^\frac{\epsilon}{2} (L_1 \lor L_2)^\frac{1}{2} (L_1 \land L_2)^\frac{1}{2} \norm{P_{N_1} Q_{L_1} u}_{L_{txy}^2} \norm{P_{N_2} Q_{L_2} v}_{L_{txy}^2}  \end{aligned} \label{helpbilinrefinement} \end{equation}
by means of dyadic summation. We therefore proceed to the proof of \eqref{helpbilinrefinement}. An application of Parseval's identity, the Cauchy-Schwarz inequality and Fubini's theorem yields \begin{align*} &\norm{I_{x}^\frac{\alpha}{4}P^\alpha((P_{N_1} Q_{L_1} u)(P_{N_2} Q_{L_2} v))}_{L_{txy}^2} \\ &\lesssim \brBigg{\sup_{(\tau,\xi,q) \in \R \times \R^{>0} \times \Z} \xi^\frac{\alpha}{4} \abs{A_{\tau, \xi, q}^\alpha}^\frac{1}{2}} \norm{P_{N_1} Q_{L_1} u}_{L_{txy}^2} \norm{P_{N_2} Q_{L_2} v}_{L_{txy}^2} \end{align*}
with \begin{align*} &A_{\tau,\xi,q}^\alpha \coloneqq  \setbig{ (\tau_1,\xi_1,q_1) \in \R \times \R \times \Z \ \big| \ \xi \gtrsim 1, \ \abs{3\xi^2-q^2} \gtrsim \xi^\alpha, \ \abs{(\xi_1,q_1)} \lesssim N_1, \\ &  \abs{(\xi-\xi_1,q-q_1)} \lesssim N_2, \ \abs{\tau_1-\phi(\xi_1,q_1)} \lesssim L_1, \ \abs{\tau-\tau_1 - \phi(\xi-\xi_1,q-q_1)} \lesssim L_2}. \end{align*}
Now, if we control the extension of $A_{\tau,\xi,q}^\alpha$ in the $\tau_1$-direction by $(L_1 \land L_2)$, and then substitute \[ \tilde{\xi_1} \coloneqq \xi_1-\frac{\xi}{2} \quad \text{and} \quad \tilde{q_1} \coloneqq q_1-\frac{q}{2} - h(q), \] just as in the proof of Theorem \ref{L4}, it follows - taking into account the translation invariance of the Haar measure on $\R \times \Z$ - that \[ \xi^\frac{\alpha}{4} \abs{A_{\tau,\xi,q}^\alpha}^\frac{1}{2} \lesssim (L_1 \land L_2)^\frac{1}{2} \xi^\frac{\alpha}{4} \abs{B_{\tau,\xi,q}^\alpha}^\frac{1}{2}, \] with $B_{\tau,\xi,q}^\alpha = B_{\tau,\xi,q}^\alpha(c,K) = B_{\tau,\xi,q}^\alpha(-\tilde{c}(L_1 \lor L_2) + \tau - \frac{\xi}{4}(\xi^2+q^2), 2 \tilde{c} (L_1 \lor L_2))$ defined exactly as in Lemma \ref{mes}. From an application of \eqref{bilina}, it finally follows that \[ \xi^\frac{\alpha}{4} \abs{A_{\tau, \xi,q}^\alpha}^\frac{1}{2} \lesssim_{\epsilon} (L_1 \land L_2)^\frac{1}{2} (L_1 \lor L_2)^\frac{1}{2} (N_1 \lor N_2)^\frac{\epsilon}{2}, \] and thus the proof of \eqref{helpbilinrefinement} is complete.
\end{proof}

\begin{rem}

Estimate \eqref{bilinrefinement} will likewise be available for later estimates involving duality: Taking Remark \ref{remark mp} (ii) into account, an argument analogous to the proof of Proposition \ref{MP prop} yields \begin{equation} \label{trivial refinement} \norm{I_{x}^\frac{\alpha}{4}P^\alpha(u  v)}_{L_{txy}^2} \lesssim \norm{J_{x}^{\frac{\alpha}{4}}u}_{X_{1+,\frac{1}{4}+}} \norm{J_{x}^{\frac{\alpha}{4}}v}_{X_{0,\frac{1}{4}+}}, \end{equation} which, after interpolation with \eqref{bilinrefinement}, leads to \begin{equation} \label{dual refinement} \norm{I_{x}^\frac{\alpha}{4}P^\alpha(u  v)}_{L_{txy}^2} \lesssim \norm{u}_{X_{0+,\frac{1}{2}-}} \norm{v}_{X_{0+,\frac{1}{2}-}}.  \end{equation}

\end{rem}

We now have all the necessary ingredients to show our well-posedness results in the cases $k=2$ (mZK) and $k\geq 4$. For the case $k=3$, we will conclude this section by proving $L^5$-estimates that are specifically tailored to the quartic nonlinearity.

\begin{cor}

Let $T>0$, $0< \epsilon \ll 1$ and $b>\frac{1}{2}$. Furthermore, if $u$, $v$, $w$ are time-localized\footnote{Time localization is not actually required for the function $v$.} functions with $J_{x}^\frac{1}{5}u \in X_{\epsilon,b}$, $ J_{y}^\frac{1}{5}v \in X_{\epsilon,b}$, and $w \in X_{\frac{2}{15}+\epsilon,b}$, then the following three estimates hold: \begin{equation} \label{Schrödinger L5} \norm{u}_{L_{Txy}^5} \lesssim_{T,\epsilon,b} \norm{J_{x}^\frac{1}{5}u}_{X_{\epsilon,b}}, \end{equation}
\begin{equation} \label{Airy L5} \norm{I_{x}^\frac{1}{10}v}_{L_{txy}^5} \lesssim_{b} \norm{J_{y}^\frac{1}{5}v}_{X_{\epsilon,b}}, \end{equation}
and \begin{equation} \label{optimized L5} \norm{w}_{L_{Txy}^5} \lesssim_{T,\epsilon,b} \norm{w}_{X_{\frac{2}{15}+\epsilon,b}}. \end{equation}

\end{cor}

\begin{proof}

By interpolating the newly obtained $L^4$-estimate \eqref{ineq1} with \eqref{Xsb Schrödinger L6}, we obtain \eqref{Schrödinger L5}, while \eqref{Airy L5} follows from interpolating \eqref{ineq1} and \eqref{Airy L6 Xsb}. By decomposing on the Fourier side into the regions $A \coloneqq \set{(\xi,q) \ | \ \abs{\xi} \leq \langle q \rangle^\frac{2}{3}}$ and $A^c \coloneqq \set{(\xi,q) \ | \ \abs{\xi} > \langle q \rangle^\frac{2}{3}}$, estimate \eqref{optimized L5} follows from using estimate \eqref{Schrödinger L5} on region $A$ and estimate \eqref{Airy L5} on region $A^c$.
\end{proof}

\section{Local well-posedness theory for GZK}

\subsection{General LWP for gZK}

It is commonly known that the well-posedness results stated in Theorem \ref{general LWP} follow from a standard fixed-point argument, provided that a certain multilinear estimate can be established. We state the desired estimate in the following

\begin{prop} \label{prop general lwp}
Let $k \in \N^{\geq 2}$, and let $s_{0}(k)$ be as defined in Theorem \ref{general LWP}. Then, for every $s > s_{0}(k)$, there exists a sufficiently small number $\epsilon = \epsilon(s) > 0$ such that the following estimate holds for all time-localized functions $u_i \in X_{s,\frac{1}{2}+\epsilon}$:
\begin{equation} \label{multlin general lwp} \norm{\partial_{x} \brBigg{\prod_{i=1}^{k+1} u_i}}_{X_{s,-\frac{1}{2}+2\epsilon}} \lesssim \prod_{i=1}^{k+1} \norm{u_i}_{X_{s,\frac{1}{2}+\epsilon}}. \end{equation}

\end{prop}

In this paper, we restrict our attention to the proof of \eqref{multlin general lwp} and refer the reader to, e.g., \cite{Bourgain1993, Gruenrock2004, Saut2001} for details concerning the fixed-point method.

\begin{proof}[Proof of Proposition \ref{prop general lwp}]
We divide the proof into three parts and start with some preparatory remarks:
By duality and an application of Parseval's identity, estimate \eqref{multlin general lwp} is equivalent to showing \begin{align*} I_{f} &\coloneqq \absBigg{\int_{\R^2} \sum_{q \in \Z} \int_{\R^{2k}} \sum_{\substack{q_1,...,q_k \in \Z \\ \ast}} \xi \langle (\xi,q) \rangle^{s} \brBigg{\prod_{i=1}^{k+1} \hat{u}_i(\tau_i,\xi_i,q_i)} \overline{\hat{f}}(\tau,\xi,q) \mathrm{d} \tilde{\tau} \mathrm{d} \tilde{\xi} \mathrm{d}(\tau,\xi)} \\ & \lesssim  \norm{f}_{X_{0,\frac{1}{2}-2\epsilon}} \prod_{i=1}^{k+1} \norm{u_i}_{X_{s,\frac{1}{2}+\epsilon}} \end{align*}
for all $f \in X_{0,\frac{1}{2}-2\epsilon}$ with $\norm{f}_{X_{0,\frac{1}{2}-2\epsilon}} \leq 1$. Here, the symbol $\ast$ denotes the convolution constraint $(\tau,\xi,q) = (\tau_1+...+\tau_{k+1},\xi_1+...+\xi_{k+1},q_{1}+...+q_{k+1})$, and we adopt the notation $\mathrm{d} \tilde{\tau} \mathrm{d} \tilde{\xi} \coloneqq \mathrm{d} \tau_1...\mathrm{d} \tau_{k} \mathrm{d} \xi_1...\mathrm{d} \xi_k$ for brevity. Furthermore, since the norms involved only depend on $\abs{\hat{u}_i}$ and $\abs{\hat{f}}$, we may, without loss of generality, assume that all $\hat{u}_i$ and $\hat{f} $ are nonnegative for the remainder of the proof. Finally, if $(\xi_i,q_i)$ denotes the frequency variable associated with the factor $u_i$, then by symmetry, it suffices to restrict the following analysis to the case where the frequency distribution satisfies $\abs{(\xi_1,q_1)} \geq \abs{(\xi_2,q_2)} \geq ... \geq \abs{(\xi_{k+1},q_{k+1})}$. \\ \\
(i) \underline{$k=2$}: \\
(i.1) \underline{$\abs{(\xi_1,q_1)} \lesssim 1$}: \\ In this case, we have $\abs{(\xi_i,q_i)} \lesssim 1$ and $\abs{(\xi,q)} \lesssim 1$, for all $i \in \set{1,2,3}$, so that Sobolev's embedding theorem applies without any derivative loss. Consequently, for any $0<\epsilon \ll 1$ and $s \geq 0$, we obtain \begin{align*} \norm{\partial_{x}(u_1u_2u_3)}_{X_{s,-\frac{1}{2}+2\epsilon}} &\lesssim \norm{u_1u_2u_3}_{L_{txy}^2} \\ & \leq \norm{u_1}_{L_{txy}^2} \norm{u_2}_{L_{t}^\infty L_{xy}^\infty} \norm{u_2}_{L_{t}^\infty L_{xy}^\infty} \\ & \lesssim \norm{u_1}_{L_{txy}^2} \norm{u_2}_{L_{t}^\infty L_{xy}^2} \norm{u_2}_{L_{t}^\infty L_{xy}^2} \\ & \lesssim \norm{u_1}_{X_{0,\frac{1}{2}+\epsilon}} \norm{u_2}_{X_{0,\frac{1}{2}+\epsilon}} \norm{u_3}_{X_{0, \frac{1}{2}+\epsilon}} \\ & \lesssim \prod_{i=1}^{3} \norm{u_i}_{X_{s,\frac{1}{2}+\epsilon}},  \end{align*} where we have used the fact that \begin{equation} \label{Linfty to L2} 
\norm{u}_{L_{t}^\infty H_{xy}^s} \lesssim_{b} \norm{u}_{X_{s,b}}, \qquad b>\frac{1}{2}. \end{equation} (i.2) \underline{$\abs{(\xi_1,q_1)} \gg 1$, $\abs{(\xi_1,q_1)} \gg \abs{(\xi_3,q_3)}$}: \\
(i.2.1) \underline{$ \abs{(\xi,q)} \gg \abs{(\xi_2,q_2)} $}: \\ In this situation, we have in particular $1 \ll \abs{(\xi_1,q_1)} \sim \abs{(\xi,q)} \sim \langle (\xi,q) \rangle$, and we obtain \begin{align*} \abs{\xi} \langle (\xi,q) \rangle^s &\lesssim \abs{\abs{(\xi_1,q_1)}^2-\abs{(\xi_3,q_3)}^2}^\frac{1}{2} \langle (\xi_1,q_1) \rangle^s \langle (\xi_3,q_3) \rangle^{-\frac{1}{2}+} \\ & \ \ \ \cdot \abs{\abs{(\xi,q)}^2-\abs{(\xi_2,q_2)}^2}^\frac{1}{2} \langle (\xi,q) \rangle^{0-} \langle (\xi_2,q_2) \rangle^{-\frac{1}{2}+}.  \end{align*}
Undoing Plancherel and subsequently applying Hölder's inequality leads to
\[ I_{f} \lesssim \norm{MP(J^su_1, J^{-\frac{1}{2}+} u_3)}_{L_{txy}^2} \norm{MP(J^{0-}f,J^{-\frac{1}{2}+}\tilde{u}_2)}_{L_{txy}^2}\footnote{We denote $\tilde{u}_2(t,x,y) \coloneqq u_2(-t,-x,-y).$}, \] and applying \eqref{OG bilin} and \eqref{MP dual} finally yields \[ ... \lesssim \norm{u_1}_{X_{s,\frac{1}{2}+\epsilon}} \norm{u_3}_{X_{0+,\frac{1}{2}+\epsilon}} \norm{f}_{X_{0+,\frac{1}{2}-2\epsilon}} \norm{u_2}_{X_{0+,\frac{1}{2}+\epsilon}}, \] which shows the desired estimate for any $s > 0$, provided $\epsilon > 0$ is sufficiently small. \\
(i.2.2) \underline{$ \abs{(\xi,q)} \lesssim \abs{(\xi_2,q_2)} $}: \\ For this frequency configuration, we have \[ \abs{\xi} \langle (\xi,q) \rangle^s \lesssim \abs{\abs{(\xi_1,q_1)}^2-\abs{(\xi_3,q_3)}^2}^\frac{1}{2} \langle (\xi_1,q_1) \rangle^s \langle (\xi_3,q_3) \rangle^{-\frac{1}{4}+} \langle (\xi,q) \rangle^{0-} \langle (\xi_2,q_2) \rangle^{\frac{1}{4}+}, \] and undoing Plancherel, followed by Hölder's inequality, yields \begin{align*} I_{f} &\lesssim \norm{MP(J^{s}u_1,J^{-\frac{1}{4}+}u_3)}_{L_{txy}^2} \norm{J^{0-}f J^{\frac{1}{4}+}\tilde{u}_2}_{L_{txy}^2} \\ & \leq \norm{MP(J^{s}u_1,J^{-\frac{1}{4}+}u_3)}_{L_{txy}^2} \norm{J^{0-}f}_{L_{txy}^{4-}} \norm{J^{\frac{1}{4}+}\tilde{u}_2}_{L_{Txy}^{4+}}. \end{align*}
For the first factor, we can once again rely on \eqref{OG bilin}, while for estimating the second factor, we make use of \eqref{ineq1d}. The third factor can be controlled using the estimate in \eqref{ineq1}, for example by interpolating it with \eqref{optimized L6}. We obtain \[ ... \lesssim \norm{u_1}_{X_{s,\frac{1}{2}+\epsilon}} \norm{u_3}_{X_{\frac{1}{4}+,\frac{1}{2}+\epsilon}} \norm{f}_{X_{0,\frac{1}{2}-2\epsilon}} \norm{u_2}_{X_{\frac{1}{4}+,\frac{1}{2}+\epsilon}}, \] and this yields the desired estimate for any $s > \frac{1}{4}$, with $\epsilon > 0$ taken small enough. \\
(i.3) \underline{$\abs{(\xi_1,q_1)} \gg 1, \abs{(\xi_1,q_1)} \sim \abs{(\xi_2,q_2)} \sim \abs{(\xi_3,q_3)}$}: \\ In this case, we can freely move derivatives around and obtain \[ \abs{\xi} \langle (\xi,q) \rangle^s \lesssim \langle (\xi,q) \rangle^{0-} \prod_{i=1}^{3} \langle (\xi_i,q_i) \rangle^{\frac{s}{3}+\frac{1}{3}+}. \]
Taking this into account, and then applying the dual version of \eqref{ineq1d}, followed by Hölder's inequality, leads to \begin{align*} \norm{\partial_{x} (u_1u_2u_3)}_{X_{s,-\frac{1}{2}+2\epsilon}} &\lesssim \norm{J^{0-}(J^{\frac{s}{3}+\frac{1}{3}+}u_1J^{\frac{s}{3}+\frac{1}{3}+}u_2J^{\frac{s}{3}+\frac{1}{3}+}u_3)}_{X_{0,-\frac{1}{2}+2\epsilon}} \\ & \lesssim \norm{J^{\frac{s}{3}+\frac{1}{3}+}u_1J^{\frac{s}{3}+\frac{1}{3}+}u_2J^{\frac{s}{3}+\frac{1}{3}+}u_3}_{L_{txy}^{\frac{4}{3}+}} \\ & \leq \norm{J^{\frac{s}{3}+\frac{1}{3}+}u_1}_{L_{Txy}^{4+}}  \norm{J^{\frac{s}{3}+\frac{1}{3}+}u_2}_{L_{Txy}^{4+}}  \norm{J^{\frac{s}{3}+\frac{1}{3}+}u_3}_{L_{Txy}^{4+}}. \end{align*}
A triple application of \eqref{ineq1}, suitably interpolated with \eqref{optimized L6}, ultimately yields \begin{align*} ... &\lesssim \norm{u_1}_{X_{\frac{s}{3}+\frac{1}{3}+,\frac{1}{2}+\epsilon}} \norm{u_2}_{X_{\frac{s}{3}+\frac{1}{3}+,\frac{1}{2}+\epsilon}} \norm{u_3}_{X_{\frac{s}{3}+\frac{1}{3}+,\frac{1}{2}+\epsilon}} \\ & \lesssim \norm{u_1}_{X_{s,\frac{1}{2}+\epsilon}} \norm{u_2}_{X_{s,\frac{1}{2}+\epsilon}} \norm{u_3}_{X_{s,\frac{1}{2}+\epsilon}}, \end{align*}
with the last inequality requiring \[ \frac{s}{3}+\frac{1}{3}+ \leq s \Leftrightarrow \frac{1}{2}+ \leq s. \]
For $\epsilon > 0$ chosen sufficiently small, the proof of \eqref{multlin general lwp} is thereby also completed in this case for all $s > \frac{1}{2} = s_0(2)$. Putting together all intermediate results finishes the proof in the case $k=2$. \\ \\
(ii) \underline{$k=3$}: \\
(ii.1) \underline{$\abs{(\xi_1,q_1)} \lesssim 1$}: \\ We proceed exactly as in the case (i.1) and use the Sobolev embedding theorem without loss of derivatives, followed by \eqref{Linfty to L2}. This yields \begin{align*} \norm{\partial_x (u_1u_2u_3u_4)}_{X_{s,-\frac{1}{2}+2\epsilon}} & \lesssim \norm{u_1u_2u_3u_4}_{L_{txy}^2} \\ & \leq \norm{u_1}_{L_{txy}^2} \norm{u_2}_{L_{t}^\infty L_{xy}^\infty} \norm{u_3}_{L_{t}^\infty L_{xy}^\infty} \norm{u_4}_{L_{t}^\infty L_{xy}^\infty} \\ & \lesssim \norm{u_1}_{L_{txy}^2} \norm{u_2}_{L_{t}^\infty L_{xy}^2} \norm{u_3}_{L_{t}^\infty L_{xy}^2} \norm{u_4}_{L_{t}^\infty L_{xy}^2} \\ & \lesssim \norm{u_1}_{X_{0,\frac{1}{2}+\epsilon}} \norm{u_2}_{X_{0,\frac{1}{2}+\epsilon}}\norm{u_3}_{X_{0,\frac{1}{2}+\epsilon}}\norm{u_4}_{X_{0,\frac{1}{2}+\epsilon}} \\ & \lesssim \prod_{i=1}^{4} \norm{u_i}_{X_{s,\frac{1}{2}+\epsilon}} \end{align*} for every $0 < \epsilon \ll 1$ and $s \geq 0$. \\
(ii.2) \underline{$\abs{(\xi_1,q_1)} \gg 1$, $\abs{(\xi_1,q_1)} \gg \abs{(\xi_4,q_4)}$}: \\
(ii.2.1) \underline{$\abs{(\xi,q)} \gg \abs{(\xi_2,q_2)}$}: \\ Taking $1 \ll \abs{(\xi_1,q_1)} \sim \abs{(\xi,q)} \sim \langle (\xi,q) \rangle$ into account, it follows under the given assumptions that \begin{align*} \abs{\xi} \langle (\xi,q) \rangle^s &\lesssim \abs{\abs{(\xi_1,q_1)}^2- \abs{(\xi_4,q_4)}^2}^{\frac{1}{2}} \langle (\xi_1,q_1) \rangle^s \langle (\xi_4,q_4) \rangle^{-\frac{1}{6}+} \\ & \ \ \ \cdot \abs{\abs{(\xi,q)}^2- \abs{(\xi_2,q_2)}^2}^{\frac{1}{2}} \langle (\xi,q) \rangle^{0-} \langle (\xi_2,q_2) \rangle^{-\frac{1}{6}+} \langle (\xi_3,q_3) \rangle^{-\frac{2}{3}+}.  \end{align*}
After undoing Plancherel and applying Hölder's inequality, we thus obtain \[ I_f \lesssim \norm{MP(J^su_1, J^{-\frac{1}{6}+}u_4)}_{L_{txy}^2} \norm{MP(J^{0-}f,J^{-\frac{1}{6}+}\tilde{u}_2)}_{L_{txy}^2} \norm{J^{-\frac{2}{3}+}\tilde{u}_3}_{L_{t}^\infty L_{xy}^\infty}, \] and for the first two factors we use \eqref{OG bilin} and \eqref{MP dual}, respectively, while the third factor can be treated using the Sobolev embedding theorem, followed by \eqref{Linfty to L2}. 
It further follows that \[ ... \lesssim \norm{u_1}_{X_{s,\frac{1}{2}+\epsilon}} \norm{u_4}_{X_{\frac{1}{3}+,\frac{1}{2}+\epsilon}} \norm{f}_{X_{0,\frac{1}{2}-2\epsilon}} \norm{u_2}_{X_{\frac{1}{3}+,\frac{1}{2}+\epsilon}} \norm{u_3}_{X_{\frac{1}{3}+,\frac{1}{2}+\epsilon}}, \] and if we choose $\epsilon > 0$ sufficiently small, we obtain the desired estimate for every $s> \frac{1}{3}$. \\
(ii.2.2) \underline{$\abs{(\xi,q)} \lesssim \abs{(\xi_2,q_2)}$}: \\ This frequency distribution allows us to deduce \begin{align*} \abs{\xi} \langle (\xi,q) \rangle^s &\lesssim \abs{\abs{(\xi_1,q_1)}^2- \abs{(\xi_4,q_4)}^2}^{\frac{1}{2}} \langle (\xi_1,q_1) \rangle^s \langle (\xi_4,q_4) \rangle^{-\frac{1}{18}} \\ & \ \ \ \cdot \langle (\xi,q) \rangle^{-\frac{2}{9}-} \langle (\xi_2,q_2) \rangle^{\frac{2}{9}+} \langle (\xi_3,q_3) \rangle^{\frac{1}{18}}, \end{align*} and after an application of Hölder's inequality in the physical space, we thus obtain \begin{align*} I_f &\lesssim \norm{MP(J^su_1,J^{-\frac{1}{18}}u_4)}_{L_{txy}^2} \norm{J^{-\frac{2}{9}-}f J^{\frac{2}{9}+}\tilde{u}_2 J^{\frac{1}{18}}\tilde{u}_3}_{L_{txy}^2} \\ & \leq \norm{MP(J^su_1,J^{-\frac{1}{18}}u_4)}_{L_{txy}^2} \norm{J^{-\frac{2}{9}-}f}_{L_{Txy}^{6-}} \norm{J^{\frac{2}{9}+}\tilde{u}_2}_{L_{Txy}^{6+}} \norm{J^{\frac{1}{18}}\tilde{u}_3}_{L_{Txy}^{6+}}. \end{align*}
For the first factor we can once again use the bilinear estimate \eqref{OG bilin}, and for the second, we apply the optimized $L^6$-estimate \eqref{optimized L6 dual}. The remaining two factors are treated using \eqref{optimized L6}, namely after interpolation with the trivial estimate \begin{equation} \label{trivial Linfty} \norm{u}_{L_{txy}^\infty} \lesssim \norm{u}_{X_{1+,\frac{1}{2}+}}. \end{equation}
These estimates lead us to \[ ... \lesssim \norm{u_1}_{X_{s,\frac{1}{2}+\epsilon}} \norm{u_4}_{X_{\frac{4}{9}+,\frac{1}{2}+\epsilon}} \norm{f}_{X_{0,\frac{1}{2}-2\epsilon}} \norm{u_2}_{X_{\frac{4}{9}+,\frac{1}{2}+\epsilon}} \norm{u_3}_{X_{\frac{5}{18}+,\frac{1}{2}+\epsilon}}, \] which shows that, in this case, \eqref{multlin general lwp} holds for every $s > \frac{4}{9}$, provided $\epsilon > 0$ is chosen small enough. \\
(ii.3) \underline{$\abs{(\xi_1,q_1)} \gg 1$, $\abs{(\xi_1,q_1)} \sim \abs{(\xi_2,q_2)} \sim \abs{(\xi_3,q_3)} \sim \abs{(\xi_4,q_4)}$}: \\ Let us assume without loss of generality that $\abs{\xi_{\text{max}}} \coloneqq \max{\set{\abs{\xi_1}, \abs{\xi_2}, \abs{\xi_3}, \abs{\xi_4}}} = \abs{\xi_1}$. From $\abs{\xi} \lesssim \abs{\xi_{\text{max}}} = \abs{\xi_1}$ and the fact that all $\abs{(\xi_i,q_i)}$ are of approximately the same size, it follows that \begin{align*} \abs{\xi} \langle (\xi,q) \rangle^{s} & \lesssim \abs{\xi}^{\frac{1}{10}} \langle (\xi,q) \rangle^{-\frac{1}{5}-} \abs{\xi_1}^\frac{1}{4} \langle (\xi_1,q_1) \rangle^{\frac{s}{4}+\frac{1}{20}+} \langle (\xi_2,q_2) \rangle^{\frac{s}{4}+\frac{1}{20}+\frac{13}{60}+} \\ & \ \ \ \cdot \langle (\xi_3,q_3) \rangle^{\frac{s}{4}+\frac{1}{20}+\frac{13}{60}+} \langle (\xi_4,q_4) \rangle^{\frac{s}{4}+\frac{1}{20}+\frac{13}{60}+}.\end{align*}
Interpolating the Airy $L^5$-estimate \eqref{Airy L5} with the trivial $L^2$-estimate \eqref{trivial L2}, we obtain \begin{equation} \norm{I_{x}^{\frac{1}{10}}u}_{L_{txy}^{5-}} \lesssim \norm{u}_{\frac{1}{5}+,\frac{1}{2}-}, \end{equation} and using the dual version of this estimate, followed by Hölder's inequality, yields \begin{align*} \norm{\partial_x(u_1u_2u_3u_4)}_{X_{s,-\frac{1}{2}+2\epsilon}} & \lesssim \norm{I_{x}^{\frac{1}{10}} J^{-\frac{1}{5}-} \brBigg{I_{x}^{\frac{1}{4}}J^{\frac{s}{4}+\frac{1}{20}+}u_1 \prod_{i=2}^{4}J^{\frac{s}{4}+\frac{1}{20}+\frac{13}{60}+}u_i }}_{X_{0,-\frac{1}{2}+2\epsilon}} \\ & \lesssim \norm{I_{x}^{\frac{1}{4}}J^{\frac{s}{4}+\frac{1}{20}+}u_1 \prod_{i=2}^{4}J^{\frac{s}{4}+\frac{1}{20}+\frac{13}{60}+}u_i}_{L_{txy}^{\frac{5}{4}+}} \\ & \leq \norm{I_{x}^{\frac{1}{4}}J^{\frac{s}{4}+\frac{1}{20}+}u_1}_{L_{txy}^5} \prod_{i=2}^{4} \norm{J^{\frac{s}{4}+\frac{1}{20}+\frac{13}{60}+}u_i}_{L_{Txy}^{5+}}. \end{align*}
For the first factor, we can once again invoke the Airy $L^5$-estimate \eqref{Airy L5}, and for the remaining factors we use the optimized $L^5$-estimate \eqref{optimized L5}, interpolated with \eqref{optimized L6}. We obtain \begin{align*} ... & \lesssim \norm{u_1}_{X_{\frac{s}{4}+\frac{2}{5}+,\frac{1}{2}+\epsilon}} \norm{u_2}_{X_{\frac{s}{4}+\frac{2}{5}+,\frac{1}{2}+\epsilon}}\norm{u_3}_{X_{\frac{s}{4}+\frac{2}{5}+,\frac{1}{2}+\epsilon}}\norm{u_4}_{X_{\frac{s}{4}+\frac{2}{5}+,\frac{1}{2}+\epsilon}} \\ & \lesssim \norm{u_1}_{X_{s,\frac{1}{2}+\epsilon}}\norm{u_2}_{X_{s,\frac{1}{2}+\epsilon}}\norm{u_3}_{X_{s,\frac{1}{2}+\epsilon}}\norm{u_4}_{X_{s,\frac{1}{2}+\epsilon}}, \end{align*}
and the final step requires \[ \frac{s}{4} + \frac{2}{5} + \leq s \Leftrightarrow \frac{8}{15} + \leq s. \] By choosing $\epsilon > 0$ sufficiently small, we thus obtain \eqref{multlin general lwp} for every $s> \frac{8}{15} = s_0(3)$, completing the proof in the case $k=3$.
 \\ \\
(iii) \underline{$k \geq 4$}: \\
(iii.1) \underline{$\abs{(\xi_1,q_1)} \lesssim 1$}: \\ This case can again be treated exactly as in (i.1) and (ii.1). For an arbitrary $s \geq 0$ and $0 < \epsilon \ll 1$, we obtain \begin{align*} \norm{ \partial_x \brBigg{ \prod_{i=1}^{k+1} u_i}}_{X_{s,-\frac{1}{2}+2\epsilon}} & \lesssim \norm{ \prod_{i=1}^{k+1} u_i}_{L_{txy}^2} \\ & \leq \norm{u_1}_{L_{txy}^2} \prod_{i=2}^{k+1} \norm{u_{i}}_{L_{t}^{\infty} L_{xy}^{\infty}} \\ & \lesssim \norm{u_1}_{L_{txy}^2} \prod_{i=2}^{k+1} \norm{u_i}_{L_{t}^\infty L_{xy}^2} \\ & \lesssim \norm{u_1}_{X_{0,\frac{1}{2}+\epsilon}} \prod_{i=2}^{k+1} \norm{u_i}_{X_{0,\frac{1}{2}+\epsilon}} \\ & \lesssim \prod_{i=1}^{k+1} \norm{u_i}_{X_{s,\frac{1}{2}+\epsilon}}, \end{align*}
which completes the proof in this subcase. \\
(iii.2) \underline{$\abs{(\xi_1,q_1)} \gg 1$, $\abs{(\xi_1,q_1)} \gg \abs{(\xi_5,q_5)}$}: \\
(iii.2.1) \underline{$\abs{(\xi,q)} \gg \abs{(\xi_2,q_2)}$}: \\ In this situation, we have $\abs{(\xi_1,q_1)} \sim \abs{(\xi,q)} \gg \abs{(\xi_2,q_2)}  \geq \abs{(\xi_3,q_3)}$, from which it follows that \begin{align*} \abs{\xi} \langle (\xi,q) \rangle^s & \lesssim \abs{\abs{(\xi_1,q_1)}^2 - \abs{(\xi_2,q_2)}^2}^\frac{1}{2} \langle (\xi_1,q_1) \rangle^s \langle (\xi_2,q_2) \rangle^{\frac{1}{2}-\frac{2}{k}} \\ & \ \ \ \cdot \abs{\abs{(\xi,q)}^2 - \abs{(\xi_3,q_3)}^2}^\frac{1}{2} \langle (\xi,q) \rangle^{0-} \langle (\xi_3,q_3)\rangle^{\frac{1}{2}-\frac{2}{k}} \prod_{i=4}^{k+1} \langle (\xi_i,q_i) \rangle^{-\frac{2}{k}+}. \end{align*}
By undoing Plancherel and applying Hölder's inequality, we then obtain \[ I_f \lesssim \norm{MP(J^su_1,J^{\frac{1}{2}-\frac{2}{k}}u_2)}_{L_{txy}^2} \norm{MP(J^{0-}f,J^{\frac{1}{2}-\frac{2}{k}}\tilde{u}_3)}_{L_{txy}^2} \prod_{i=4}^{k+1} \norm{J^{-\frac{2}{k}+}\tilde{u}_i}_{L_{t}^\infty L_{xy}^\infty}, \]
and for the first two factors we use \eqref{OG bilin} and \eqref{MP dual}, respectively, while the remaining factors can be dealt with by applying Sobolev's embedding theorem, followed by \eqref{Linfty to L2}. These steps lead to \[ ... \lesssim \norm{u_1}_{X_{s,\frac{1}{2}+\epsilon}} \norm{u_2}_{X_{1-\frac{2}{k}+,\frac{1}{2}+\epsilon}} \norm{f}_{X_{0,\frac{1}{2}-2\epsilon}} \norm{u_3}_{X_{1-\frac{2}{k}+,\frac{1}{2}+\epsilon}} \prod_{i=4}^{k+1} \norm{u_i}_{X_{1-\frac{2}{k}+,\frac{1}{2}+\epsilon}}, \]
and by choosing $\epsilon > 0$ small enough, we obtain \eqref{multlin general lwp} for every $s > 1- \frac{2}{k}$ in this case. \\
(iii.2.2) \underline{$\abs{(\xi,q)} \lesssim \abs{(\xi_2,q_2)}$}: \\ In this case, the frequency configuration under consideration, leads us to \begin{align*} \abs{\xi} \langle (\xi,q) \rangle^s &\stackrel{k=4}{\lesssim} \abs{\abs{(\xi_1,q_1)}^2- \abs{(\xi_5,q_5)}^2}^\frac{1}{2} \langle (\xi_1,q_1) \rangle^s \\ & \ \ \ \ \cdot \langle (\xi,q) \rangle^{-\frac{2}{9}-} \langle (\xi_2,q_2) \rangle^{\frac{3}{9}+} \langle (\xi_3,q_3) \rangle^{\frac{3}{9}+} \langle (\xi_4,q_4) \rangle^{-\frac{4}{9}+}  \end{align*}
and
 \begin{align*} \abs{\xi} \langle (\xi,q) \rangle^s & \stackrel{k\geq 5}{\lesssim} \abs{\abs{(\xi_1,q_1)}^2- \abs{(\xi_5,q_5)}^2}^\frac{1}{2} \langle (\xi_1,q_1) \rangle^s \langle (\xi_5,q_5) \rangle^{\frac{1}{2}-\frac{11}{6k}+} \\ & \ \ \ \ \cdot \langle (\xi,q) \rangle^{-\frac{2}{9}-} \langle (\xi_2,q_2) \rangle^{\frac{7}{9}-\frac{11}{6k}+} \langle (\xi_3,q_3) \rangle^{\frac{7}{9}-\frac{11}{6k}+} \prod_{\substack{i=4 \\ i\neq 5}}^{k+1} \langle (\xi_i,q_i) \rangle^{-\frac{11}{6k}+}, \end{align*}
and after applying Plancherel's theorem and Hölder's inequality, we thus obtain
\begin{align*} I_f &\lesssim \norm{MP(J^su_1,u_5)}_{L_{txy}^2} \norm{J^{-\frac{2}{9}-}f J^{\frac{3}{9}+}\tilde{u}_2 J^{\frac{3}{9}+}\tilde{u}_3}_{L_{txy}^2} \norm{J^{-\frac{4}{9}+}\tilde{u}_4}_{L_{t}^\infty L_{xy}^\infty} \\ & \leq \norm{MP(J^su_1,u_5)}_{L_{txy}^2} \norm{J^{-\frac{2}{9}-}f}_{L_{Txy}^{6-}} \norm{J^{\frac{3}{9}+}\tilde{u}_2}_{L_{Txy}^{6+}} \norm{J^{\frac{3}{9}+}\tilde{u}_3}_{L_{Txy}^{6+}} \\ & \ \ \ \cdot \norm{J^{-\frac{4}{9}+}\tilde{u}_4}_{L_{t}^\infty L_{xy}^\infty} \end{align*}
in the case $k=4$, and
\begin{align*} I_f &\lesssim \norm{MP(J^su_1,J^{\frac{1}{2}-\frac{11}{6k}+}u_5)}_{L_{txy}^2} \norm{J^{-\frac{2}{9}-}f J^{\frac{7}{9}-\frac{11}{6k}+}\tilde{u}_2 J^{\frac{7}{9}-\frac{11}{6k}+}\tilde{u}_3}_{L_{txy}^2} \\ & \ \ \ \cdot \prod_{\substack{i=4 \\ i\neq 5}}^{k+1} \norm{J^{-\frac{11}{6k}+}\tilde{u}_i}_{L_{t}^\infty L_{xy}^\infty} \\ & \leq \norm{MP(J^su_1,J^{\frac{1}{2}-\frac{11}{6k}+}u_5)}_{L_{txy}^2} \norm{J^{-\frac{2}{9}-}f}_{L_{Txy}^{6-}} \norm{J^{\frac{7}{9}-\frac{11}{6k}+}\tilde{u}_2}_{L_{Txy}^{6+}} \norm{J^{\frac{7}{9}-\frac{11}{6k}+}\tilde{u}_3}_{L_{Txy}^{6+}} \\ & \ \ \ \cdot \prod_{\substack{i=4 \\ i\neq 5}}^{k+1} \norm{J^{-\frac{11}{6k}+}\tilde{u}_i}_{L_{t}^\infty L_{xy}^\infty} \end{align*}
in the case $k \geq 5$.
In both cases we can now proceed in the same manner: The first factor can be estimated using \eqref{OG bilin}, and for the second factor we apply \eqref{optimized L6 dual}, while for the third and fourth factors we interpolate \eqref{optimized L6} with the trivial $L^\infty$-estimate \eqref{trivial Linfty}. All remaining factors are treated using the Sobolev embedding theorem and \eqref{Linfty to L2}, from which it follows that
\[ ... \lesssim \norm{u_1}_{X_{s,\frac{1}{2}+\epsilon}} \norm{u_5}_{X_{\frac{1}{2}+,\frac{1}{2}+\epsilon}} \norm{f}_{X_{0,\frac{1}{2}-2\epsilon}} \norm{u_2}_{X_{\frac{5}{9}+,\frac{1}{2}+\epsilon}} \norm{u_3}_{X_{\frac{5}{9}+,\frac{1}{2}+\epsilon}} \norm{u_4}_{X_{\frac{5}{9}+,\frac{1}{2}+\epsilon}} \]
and
\[ ... \lesssim \norm{u_1}_{X_{s,\frac{1}{2}+\epsilon}} \norm{f}_{X_{0,\frac{1}{2}-2\epsilon}} \prod_{i=2}^{k+1} \norm{u_i}_{X_{1-\frac{11}{6k}+,\frac{1}{2}+\epsilon}}, \]
respectively. Now, choosing $\epsilon > 0$ sufficiently small, the proof of \eqref{multlin general lwp} in this subcase is thus established for every $s > \frac{5}{9} = s_0(4)$ ($k=4$) and $s > 1- \frac{11}{6k}$ ($k \geq 5$). \\
(iii.3) \underline{$\abs{(\xi_1,q_1)} \gg 1$, $\abs{(\xi_1,q_1)} \sim ... \sim \abs{(\xi_5,q_5)}$}: \\
(iii.3.1) \underline{$\abs{\xi_{\text{max}}} = \abs{\xi_i}$, for some $i \in \set{1,...,5}$}: \\ Without loss of generality, let us again assume $\abs{\xi_{\text{max}}} = \abs{\xi_1}$. We then have $\abs{\xi} \lesssim \abs{\xi_1}$, and since we can redistribute the derivatives arbitrarily among the first five factors, it follows that \begin{align*} \abs{\xi} \langle (\xi,q) \rangle^s & \lesssim \abs{\xi}^{\frac{1}{6}} \langle (\xi,q) \rangle^{-\frac{1}{3}-} \abs{\xi_1}^{\frac{19}{90}} \langle (\xi_1,q_1) \rangle^{\frac{s}{5}+\frac{1}{15}+\frac{16(k-4)}{45k}+} \\ & \ \ \ \cdot \brBigg{\prod_{i=2}^{5} \langle (\xi_i,q_i) \rangle^{\frac{s}{5}+\frac{1}{15}+\frac{7}{45}+\frac{16(k-4)}{45k}+}} \prod_{i=6}^{k+1} \langle (\xi_i,q_i) \rangle^{-\frac{16}{9k}+}. \end{align*}
Now, if we interpolate the Airy $L^6$-estimate \eqref{Airy L6 Xsb} with the trivial $L^2$-estimate \eqref{trivial L2}, we obtain \begin{equation} \label{dual Airy L6} \norm{I_{x}^{\frac{1}{6}}u}_{L_{txy}^{6-}} \lesssim \norm{u}_{X_{\frac{1}{3}+,\frac{1}{2}-}},  \end{equation} and an application of its dual version, followed by Hölder's inequality then leads us to 
\begin{align*} &\norm{\partial_x \brBigg{ \prod_{i=1}^{k+1} u_i}}_{X_{s,-\frac{1}{2}+2\epsilon}} \\ &\lesssim \norm{I_{x}^{\frac{19}{90}} J^{\frac{s}{5}+\frac{1}{15}+\frac{16(k-4)}{45k}+}u_1 \brBigg{\prod_{i=2}^{5} J^{\frac{s}{5}+\frac{1}{15}+\frac{7}{45}+\frac{16(k-4)}{45k}+}u_i} \brBigg{ \prod_{i=6}^{k+1} J^{-\frac{16}{9k}+} u_i}}_{L_{txy}^{\frac{6}{5}+}} \\ & \leq \norm{I_{x}^{\frac{19}{90}} J^{\frac{s}{5}+\frac{1}{15}+\frac{16(k-4)}{45k}+}u_1}_{L_{txy}^6} \prod_{i=2}^{5} \norm{J^{\frac{s}{5}+\frac{1}{15}+\frac{7}{45}+\frac{16(k-4)}{45k}+}u_i}_{L_{Txy}^{6+}} \\ & \ \ \ \cdot \prod_{i=6}^{k+1} \norm{J^{-\frac{16}{9k}+} u_i}_{L_{t}^\infty L_{xy}^\infty}. \end{align*}
For the first factor, we now apply the Airy $L^6$-estimate \eqref{Airy L6 Xsb}, and for the second through fifth factors, we employ the optimized $L^6$-estimate \eqref{optimized L6} interpolated with the trivial $L^\infty$-estimate \eqref{trivial Linfty}. All remaining factors are treated using the Sobolev embedding theorem and estimate \eqref{Linfty to L2}, and we ultimately obtain \begin{align*} ... &\lesssim \brBigg{\prod_{i=1}^{5} \norm{u_i}_{X_{\frac{s}{5}+\frac{4}{5}-\frac{64}{45k}+,\frac{1}{2}+\epsilon}}} \prod_{i=6}^{k+1} \norm{u_i}_{X_{1-\frac{16}{9k}+,\frac{1}{2}+\epsilon}} \\ & \lesssim \brBigg{\prod_{i=1}^{5} \norm{u_i}_{X_{s,\frac{1}{2}+\epsilon}}} \prod_{i=6}^{k+1} \norm{u_i}_{X_{1-\frac{16}{9k}+,\frac{1}{2}+\epsilon}}.   \end{align*}
Here, the final step requires \[ \frac{s}{5}+\frac{4}{5}-\frac{64}{45k} + \leq s \Leftrightarrow 1- \frac{16}{9k}+ \leq s \]
and by choosing $\epsilon > 0$ sufficiently small, the desired estimate is thereby established for all $s > 1 - \frac{16}{9k} = s_{0}(k)$. \\
(iii.3.2) \underline{$\abs{\xi_{\text{max}}} = \abs{\xi_i}$, for some $i \in \set{6,...,k+1}$}\footnote{It is evident that this case cannot occur for $k=4$.}: \\
(iii.3.2.1) \underline{$\abs{(\xi_1,q_1)} \gg \abs{(\xi_i,q_i)}$}: \\ In this case, we have $\abs{\xi} \lesssim \abs{\xi_i} \leq \abs{(\xi_i,q_i)}$, which allows us to shift derivatives \\ onto the low-frequency factor $u_i$. 
This yields \begin{align*} \abs{\xi} \langle (\xi,q) \rangle^s & \lesssim \abs{\abs{(\xi_1,q_1)}^2-\abs{(\xi_i,q_i)}^2}^\frac{1}{2} \langle (\xi_1,q_1) \rangle^s \langle (\xi_i,q_i) \rangle^{\frac{1}{2}-\frac{11}{6k}+} \\ & \ \ \ \cdot \langle (\xi,q) \rangle^{-\frac{2}{9}-} \langle (\xi_2,q_2) \rangle^{\frac{7}{9}-\frac{11}{6k}+} \langle (\xi_3,q_3) \rangle^{\frac{7}{9}-\frac{11}{6k}+} \prod_{\substack{j=4 \\ j \neq i}}^{k+1} \langle (\xi_j,q_j) \rangle^{-\frac{11}{6k}+}, \end{align*}
and undoing Plancherel, followed by Hölder's inequality leads to \begin{align*} I_f &\lesssim \norm{MP(J^su_1,J^{\frac{1}{2}-\frac{11}{6k}+}u_i)}_{L_{txy}^2} \norm{J^{-\frac{2}{9}-}f J^{\frac{7}{9}-\frac{11}{6k}+}\tilde{u}_2 J^{\frac{7}{9}-\frac{11}{6k}+} \tilde{u}_3}_{L_{txy}^2} \\ & \ \ \ \cdot \prod_{\substack{j=4 \\ j \neq i}}^{k+1} \norm{J^{-\frac{11}{6k}+}\tilde{u}_j}_{L_{t}^\infty L_{xy}^\infty} \\ & \leq \norm{MP(J^su_1,J^{\frac{1}{2}-\frac{11}{6k}+}u_i)}_{L_{txy}^2} \norm{J^{-\frac{2}{9}-}f}_{L_{Txy}^{6-}} \norm{J^{\frac{7}{9}-\frac{11}{6k}+}\tilde{u}_2}_{L_{Txy}^{6+}} \\ & \ \ \ \cdot \norm{J^{\frac{7}{9}-\frac{11}{6k}+}\tilde{u}_3}_{L_{Txy}^{6+}} \prod_{\substack{j=4 \\ j \neq i}}^{k+1} \norm{J^{-\frac{11}{6k}+}\tilde{u}_j}_{L_{t}^\infty L_{xy}^\infty}.  \end{align*}
At this point, we can proceed exactly as in case (iii.2.2), and for sufficiently small $\epsilon > 0$, we thus obtain the desired multilinear estimate for every $s > 1- \frac{11}{6k}$.  \\
(iii.3.2.2) \underline{$\abs{(\xi_1,q_1)} \sim \abs{(\xi_i,q_i)}$}: \\ In this final subcase, no further work is required, since we have $\abs{(\xi_1,q_1)} \sim \abs{(\xi_2,q_2)} \\ \sim \abs{(\xi_3,q_3)} \sim \abs{(\xi_4,q_4)} \sim \abs{(\xi_i,q_i)} \sim \abs{(\xi_5,q_5)}$, allowing us to argue as in case (iii.3.1). We apply the Airy $L^6$-estimate to the factor $u_i$, and use the optimized $L^6$-estimate for the factors $u_1,...,u_4$. For all remaining factors, we invoke the Sobolev embedding theorem, which altogether yields the desired estimate \eqref{multlin general lwp} for every $s > 1- \frac{16}{9k}$.  
\end{proof}

\begin{rem}
The proof shows that local well-posedness for every $s> \frac{1}{2}$ in the case of the modified ZK equation can be achieved through the following simple alternative: for widely separated frequencies $\abs{(\xi_i,q_i)} \gg \abs{(\xi_j,q_j)}$, one applies the bilinear estimate \eqref{OG bilin}, whereas for closely spaced frequencies $\abs{(\xi_i,q_i)} \sim \abs{(\xi_j,q_j)}$, one relies on the linear $L^4$-estimate \eqref{ineq1}. In the next section, we will incorporate the newly developed bilinear refinements \eqref{bilinrefinement} and the resonance function in order to further lower the well-posedness threshold to $s > \frac{11}{24}$. 

\end{rem}

\subsection{Improved LWP for mZK}

As previously noted, the proof of Theorem \ref{LWP mZK} reduces to verifiying a concrete trilinear estimate. We fix it in the following

\begin{prop}
\label{mZK prop} For every $s > \frac{11}{24}$, there exists a sufficiently small number $\epsilon = \epsilon(s) > 0$ such that the following estimate holds for all time-localized functions $u_i \in X_{s,\frac{1}{2}+\epsilon}$:
\begin{equation} \label{multilin modified ZK} \norm{\partial_{x} (u_1u_2u_3)}_{X_{s,-\frac{1}{2}+2\epsilon}} \lesssim  \norm{u_1}_{X_{s,\frac{1}{2}+\epsilon}} \norm{u_2}_{X_{s,\frac{1}{2}+\epsilon}} \norm{u_3}_{X_{s,\frac{1}{2}+\epsilon}}. \end{equation}

\end{prop}

\begin{proof}
We begin by noting once again that \eqref{multilin modified ZK} is equivalent to showing \begin{align*} I_f &\coloneqq \absBigg{\int_{\R^2} \sum_{q \in \Z} \int_{\R^4} \sum_{\substack{q_1,q_2 \in \Z \\ \ast}} \xi \langle (\xi,q) \rangle^{s} \brBigg{\prod_{i=1}^{3} \hat{u}_i(\tau_i,\xi_i,q_i)} \overline{\hat{f}}(\tau,\xi,q) \mathrm{d} \tau_1 \mathrm{d} \tau_2 \mathrm{d} \xi_1 \mathrm{d} \xi_2 \mathrm{d} (\tau,\xi)} \\ & \lesssim \norm{f}_{X_{0,\frac{1}{2}-2\epsilon}} \norm{u_1}_{X_{s,\frac{1}{2}+\epsilon}}  \norm{u_2}_{X_{s,\frac{1}{2}+\epsilon}} \norm{u_3}_{X_{s,\frac{1}{2}+\epsilon}} \end{align*}
for all $f\in X_{0,\frac{1}{2}-2\epsilon}$ with $\norm{f}_{X_{0,\frac{1}{2}-2\epsilon}} \leq 1$, and the convolution constraint being $\ast \coloneqq (\tau,\xi,q) = (\tau_1+\tau_2+\tau_3,\xi_1+\xi_2+\xi_3,q_1+q_2+q_3)$. As in the proof of Proposition \ref{prop general lwp}, we may assume without loss of generality that $\hat{u}_i,\hat{f} \geq 0$ and $\abs{(\xi_1,q_1)} \geq \abs{(\xi_2,q_2)} \geq \abs{(\xi_3,q_3)}$, where $(\xi_i,q_i)$ denotes the frequency variable associated with the factor $u_i$. Furthermore, we will arrange the $\xi_i$ for $i=1,...,3$ in ascending order according to the magnitude of their absolute values, and we will use the notation $\xi_{\text{min}}$, $\xi_{\text{med}}$, $\xi_\text{max}$ for this purpose. Lastly, we will also use $(\xi_0,q_0)$ as an alternative notation for the frequency variable $(\xi,q)$. With these preliminary considerations in place, we now turn to a detailed case-by-case analysis. \\ \\
(i) \underline{$\abs{(\xi_1,q_1)} \lesssim 1$ or $\abs{(\xi_1,q_1)} \gg \abs{(\xi_3,q_3)}$}: \\ These cases have already been discussed in the proof of Proposition \ref{prop general lwp}, where we obtained \eqref{multilin modified ZK} for every $s > \frac{1}{4}$. \\
(ii) \underline{$\abs{(\xi_1,q_1)} \gg 1$, $\abs{(\xi_1,q_1)} \sim \abs{(\xi_2,q_2)} \sim \abs{(\xi_3,q_3)}$}: \\
(ii.1) \underline{$\abs{\xi_0} \lesssim 1$}: \\ In this situation, we have \[ \abs{\xi_0} \langle (\xi_0,q_0) \rangle^s \lesssim \langle (\xi_0,q_0) \rangle^{0-} \langle (\xi_1,q_1) \rangle^{\frac{s}{3}+} \langle (\xi_2,q_2) \rangle^{\frac{s}{3}+} \langle (\xi_3,q_3) \rangle^{\frac{s}{3}+}, \]
and an application of the dual version of \eqref{ineq1d}, followed by Hölder's inequality yields 
\begin{align*} \norm{\partial_x (u_1u_2u_3)}_{X_{s,-\frac{1}{2}+2\epsilon}} &\lesssim \norm{J^{0-}(J^{\frac{s}{3}+}u_1 J^{\frac{s}{3}+}u_2 J^{\frac{s}{3}+}u_3)}_{X_{0,-\frac{1}{2}+2\epsilon}} \\ & \lesssim \norm{J^{\frac{s}{3}+}u_1 J^{\frac{s}{3}+}u_2 J^{\frac{s}{3}+}u_3}_{L_{txy}^{\frac{4}{3}+}} \\ & \lesssim \norm{J^{\frac{s}{3}+}u_1}_{L_{Txy}^{4+}} \norm{J^{\frac{s}{3}+}u_2}_{L_{Txy}^{4+}} \norm{J^{\frac{s}{3}+}u_3}_{L_{Txy}^{4+}}. \end{align*}
By interpolating \eqref{optimized L6} with the linear $L^4$-estimate \eqref{ineq1}, and applying the resulting estimate three times, we obtain \begin{align*} ... & \lesssim \norm{u_1}_{X_{\frac{s}{3}+,\frac{1}{2}+\epsilon}} \norm{u_2}_{X_{\frac{s}{3}+,\frac{1}{2}+\epsilon}} \norm{u_3}_{X_{\frac{s}{3}+,\frac{1}{2}+\epsilon}} \\ & \lesssim \norm{u_1}_{X_{s,\frac{1}{2}+\epsilon}} \norm{u_2}_{X_{s,\frac{1}{2}+\epsilon}} \norm{u_2}_{X_{s,\frac{1}{2}+\epsilon}}, \end{align*} with the last step requiring
\[ \frac{s}{3}+ \leq s \Leftrightarrow  0+ \leq s. \]
By choosing $\epsilon > 0$ sufficiently small, the proof of \eqref{multilin modified ZK} is thereby established in this case for every $s > 0$. \\ 
(ii.2) \underline{${\lvert }\xi _{0}{\rvert } \gg 1\text{, }{\lvert }{\lvert }(\xi _{i},q_{i}){\rvert }^{2}-{\lvert }(\xi _{j},q_{j}){\rvert }^{2}{\rvert } \gtrsim {\lvert }\xi _{0}{\rvert }^{\frac{7}{6}}\text{, for some }i,j \in {\{0,1,2,3\}}\text{ with }i\neq j$}: \\
Given that
${\lvert }(\xi _{1},q_{1}){\rvert } \sim {\lvert }(\xi _{2},q_{2}){\rvert } \sim
{\lvert }(\xi _{3},q_{3}){\rvert }$, we may, without loss of generality distinguish
between two cases. If we have
${\lvert }{\lvert }(\xi _{1},q_{1}){\rvert }^{2}-{\lvert }(\xi _{2},q_{2}){\rvert }^{2}{\rvert } \gtrsim
{\lvert }\xi _{0}{\rvert }^{\frac{7}{6}}$, it follows that
\begin{align*}
{\lvert }\xi _{0}{\rvert } \langle (\xi _{0},q_{0}) \rangle ^{s} & \lesssim
{\lvert }{\lvert }(\xi _{1},q_{1}){\rvert }^{2}-{\lvert }(\xi _{2},q_{2}){\rvert }^{2}{\rvert }^{
\frac{1}{2}} \langle (\xi _{1},q_{1}) \rangle ^{s} \langle (\xi _{2},q_{2})
\rangle ^{-\frac{1}{24}+}
\\
& \ \ \ \cdot \langle (\xi _{0},q_{0}) \rangle ^{0-} \langle (\xi _{3},q_{3})
\rangle ^{\frac{11}{24}+},
\end{align*}
and undoing Plancherel, followed by H\"{o}lder's inequality, leads us to
\begin{align*}
I_{f} &\lesssim {\lVert }MP(J^{s}u_{1},J^{-\frac{1}{24}+}u_{2}){\rVert }_{L_{txy}^{2}}
{\lVert }J^{0-}f J^{\frac{11}{24}+}\tilde{u}_{3}{\rVert }_{L_{txy}^{2}}
\\
& \leq {\lVert }MP(J^{s}u_{1},J^{-\frac{1}{24}+}u_{2}){\rVert }_{L_{txy}^{2}}
{\lVert }J^{0-}f {\rVert }_{L_{txy}^{4-}} {\lVert }J^{\frac{11}{24}+}\tilde{u}_{3}{\rVert }_{L_{Txy}^{4+}}.
\end{align*}
We now apply estimate \eqref{OG bilin} to the first factor; for the second
factor, we use \eqref{ineq1d}, and for the third, \eqref{ineq1} interpolated
with \eqref{optimized L6}. This yields
\begin{equation*}
... \lesssim {\lVert }u_{1}{\rVert }_{X_{s,\frac{1}{2}+\epsilon }}
{\lVert }u_{2}{\rVert }_{X_{\frac{11}{24}+,\frac{1}{2}+\epsilon }} {\lVert }f {\rVert }_{X_{0,
\frac{1}{2}-2\epsilon }} {\lVert }u_{3}{\rVert }_{X_{\frac{11}{24}+,\frac{1}{2}+
\epsilon }},
\end{equation*}
and by choosing $\epsilon > 0$ small enough, we obtain the desired estimate \eqref{multilin modified ZK} for every $s > \frac{11}{24}$. On the other
hand, if
${\lvert }{\lvert }(\xi _{0},q_{0}){\rvert }^{2}-{\lvert }(\xi _{1},q_{1}){\rvert }^{2}{\rvert } \gtrsim
{\lvert }\xi _{0}{\rvert }^{\frac{7}{6}}$, it holds that
\begin{align*}
{\lvert }\xi _{0}{\rvert } \langle (\xi _{0},q_{0}) \rangle ^{s} & \lesssim
{\lvert }{\lvert }(\xi _{0},q_{0}){\rvert }^{2}-{\lvert }(\xi _{1},q_{1}){\rvert }^{2}{\rvert }^{
\frac{1}{2}} \langle (\xi _{0},q_{0}) \rangle ^{0-} \langle (\xi _{1},q_{1})
\rangle ^{-\frac{1}{24}+}
\\
& \ \ \ \cdot \langle (\xi _{2},q_{2}) \rangle ^{\frac{s}{2}+
\frac{11}{48}+} \langle (\xi _{3},q_{3}) \rangle ^{\frac{s}{2}+
\frac{11}{48}+},
\end{align*}
and it follows, analogously to the first case, that
\begin{equation*}
I_{f} \lesssim {\lVert }MP(J^{0-}f,J^{-\frac{1}{24}+}\tilde{u}_{1}){\rVert }_{L_{txy}^{2}}
{\lVert }J^{\frac{s}{2}+\frac{11}{48}+}u_{2}{\rVert }_{L_{txy}^{4}}
{\lVert }J^{\frac{s}{2}+\frac{11}{48}+}u_{3}{\rVert }_{L_{txy}^{4}}.
\end{equation*}
For the first factor, we now use \eqref{MP dual}, and the remaining two
factors can be treated using the $L^{4}$-estimate \eqref{ineq1}. We thus
obtain
\begin{align*}
... &\lesssim {\lVert }f {\rVert }_{X_{0,\frac{1}{2}-2\epsilon }} {\lVert }u_{1}{\rVert }_{X_{
\frac{11}{24}+,\frac{1}{2}+\epsilon }} {\lVert }u_{2}{\rVert }_{X_{\frac{s}{2}+
\frac{11}{48}+,\frac{1}{2}+\epsilon }} {\lVert }u_{3}{\rVert }_{X_{\frac{s}{2}+
\frac{11}{48}+,\frac{1}{2}+\epsilon }}
\\
& \lesssim {\lVert }f {\rVert }_{X_{0,\frac{1}{2}-2\epsilon }} {\lVert }u_{1}{\rVert }_{X_{
\frac{11}{24}+,\frac{1}{2}+\epsilon }} {\lVert }u_{2}{\rVert }_{X_{s,
\frac{1}{2}+\epsilon }} {\lVert }u_{3}{\rVert }_{X_{s,\frac{1}{2}+
\epsilon }},
\end{align*}
and the final step requires precisely
\begin{equation*}
\frac{s}{2}+\frac{11}{48}+ \leq s \Leftrightarrow \frac{11}{24}+ \leq s,
\end{equation*}
so that we arrive at the desired result in this case as well, provided
$\epsilon > 0$ is chosen small enough. \\
(ii.3) \underline{${\lvert }\xi _{0}{\rvert } \gg 1\text{, }{\lvert }{\lvert }(\xi _{i},q_{i}){\rvert }^{2}-{\lvert }(\xi _{j},q_{j}){\rvert }^{2}{\rvert } \ll {\lvert }\xi _{0}{\rvert }^{\frac{7}{6}}\text{, for all }i,j \in {\{0,1,2,3\}}\text{ with }i\neq j$}: \\
For the remainder of this proof, we may assume without loss of generality
that
 $(\xi _{1},\xi _{2},\xi _{3}) = (\xi _{\text{min}},\xi _{\text{med}},
\xi _{\text{max}})$. \\
(ii.3.1) \underline{${\lvert }\xi _{\text{med}}{\rvert } \ll {\lvert }\xi _{\text{max}}{\rvert }$}: \\
In this situation, we have
${\lvert }\xi _{0}{\rvert } \lesssim {\lvert }\xi _{1}+\xi _{3}{\rvert }$ and
${\lvert }\xi _{0}{\rvert } \lesssim {\lvert }\xi _{2}+\xi _{3}{\rvert }$, so that we obtain
\begin{align*}
{\lvert }\xi _{0}{\rvert } \langle (\xi _{0},q_{0}) \rangle ^{s} &\lesssim
{\lvert }\xi _{1}+\xi _{3}{\rvert }^{\frac{1}{4}} \langle (\xi _{1},q_{1})
\rangle ^{s-} \langle (\xi _{3},q_{3}) \rangle ^{
\frac{1}{4}+} {\lvert} \xi_0+(-\xi_2) {\rvert}^{\frac{1}{4}} \langle (\xi _{0},q_{0}) \rangle ^{0-} \\ & \ \ \ \cdot \langle (-\xi _{2},-q_{2})
\rangle ^{\frac{1}{4}+}
\end{align*}
and
\begin{align*}
{\lvert }\xi _{0}{\rvert } \langle (\xi _{0},q_{0}) \rangle ^{s} &\lesssim
{\lvert }\xi _{2}+\xi _{3}{\rvert }^{\frac{1}{4}} \langle (\xi _{2},q_{2})
\rangle ^{s-} \langle (\xi _{3},q_{3}) \rangle ^{
\frac{1}{4}+} {\lvert} \xi_0+(-\xi_1) {\rvert}^{\frac{1}{4}} \langle (\xi _{0},q_{0}) \rangle ^{0-} \\ & \ \ \ \cdot \langle (-\xi _{1},-q_{1})
\rangle ^{\frac{1}{4}+},
\end{align*}
respectively.
If we now have
${\lvert } 3(\xi _{1}+\xi _{3})^{2} - (q_{1}+q_{3})^{2}{\rvert } \gtrsim
{\lvert }\xi _{1}+\xi _{3}{\rvert }$, then an application of Parseval's identity and
H\"{o}lder's inequality yields
\begin{equation*}
I_{f} \lesssim
{\lVert }I_{x}^{\frac{1}{4}} P^{1}(J^{s-}u_{1} J^{\frac{1}{4}+}u_{3}){\rVert }_{L_{txy}^{2}}
{\lVert }I_{x}^{\frac{1}{4}} P^{1}(J^{0-}f J^{\frac{1}{4}+}\tilde{u}_{2}){\rVert }_{L_{txy}^{2}},
\end{equation*}
and the first factor can be treated using the bilinear refinement \eqref{bilinrefinement} of \eqref{ineq1}, while the second factor can be handled using the variant \eqref{dual refinement} of \eqref{bilinrefinement}. This leads us to
\begin{align*}
... &\lesssim {\lVert }u_{1}{\rVert }_{X_{s,\frac{1}{2}+\epsilon }}
{\lVert }u_{3}{\rVert }_{X_{\frac{1}{4}+,\frac{1}{2}+\epsilon }} {\lVert }f {\rVert }_{X_{0,
\frac{1}{2}-2\epsilon }} {\lVert }u_{2}{\rVert }_{X_{\frac{1}{4}+,\frac{1}{2}+
\epsilon }}
\\ &\lesssim {\lVert }f {\rVert }_{X_{0,
\frac{1}{2}-2\epsilon }} \prod_{i=1}^{3} {\lVert }u_{i} {\rVert }_{X_{s,
\frac{1}{2}+\epsilon }},
\end{align*}
with the last step being valid for all $s > \frac{1}{4}$, provided $\epsilon > 0$ is chosen small
enough. In the case
${\lvert }3(\xi _{2}+\xi _{3})^{2} - (q_{2}+q_{3})^{2}{\rvert } \gtrsim
{\lvert }\xi _{2}+\xi _{3}{\rvert }$, one uses the second pointwise estimate given above and follows the same argument to again reach $s > \frac{1}{4}$. We may therefore assume
\begin{equation*}
{\lvert } 3(\xi _{1}+\xi _{3})^{2} - (q_{1}+q_{3})^{2}{\rvert } \ll
{\lvert }\xi _{1}+\xi _{3}{\rvert } \sim {\lvert }\xi _{0}{\rvert }
\end{equation*}
and
\begin{equation*}
{\lvert }3(\xi _{2}+\xi _{3})^{2} - (q_{2}+q_{3})^{2}{\rvert } \ll
{\lvert }\xi _{2}+\xi _{3}{\rvert } \sim {\lvert }\xi _{0}{\rvert },
\end{equation*}
and we will first consider why, under these assumptions, ${\lvert} \xi_3 {\rvert} \sim {\lvert} (\xi_1,q_1) {\rvert}$ must hold: Under the assumption $ {\lvert} \xi_3 {\rvert} \ll {\lvert} (\xi_1,q_1) {\rvert} $, it follows - taking into account the active assumptions (ii.3) and (ii.3.1) - that
\begin{equation*}
{\lvert} q_1 {\rvert} \sim {\lvert} q_2 {\rvert} \sim {\lvert} q_3 {\rvert} \sim {\lvert} q_0 {\rvert} \sim {\lvert} (\xi_1,q_1) {\rvert},
\end{equation*}
and the two constraints mentioned above, which arise from the bilinear refinement of the linear $L^{4}$-estimate, further impose the sign configurations $(q_1,q_2,q_3,q_0) = (\pm,\pm,\mp,\pm)$. However, we also have ${\lvert} \xi_1+\xi_2 {\rvert} \ll {\lvert} \xi_0 {\rvert}$ and ${\lvert} \xi_0-\xi_3 {\rvert} \ll {\lvert} \xi_0 {\rvert}$, which, in conjunction with (ii.3) and the signs of the $q_i$, leads us to $ {\lvert} q_1 - q_2 {\rvert} \ll {\lvert} \xi_0 {\rvert}$, $ {\lvert} q_0 + q_3 {\rvert} \ll {\lvert} \xi_0 {\rvert}$, and
\begin{equation*}
2 {\lvert} q_1+q_3 {\rvert} = {\lvert} q_1 - q_2 + q_0 + q_3 {\rvert} \leq {\lvert} q_1 - q_2 {\rvert} + {\lvert} q_0 + q_3 {\rvert} \ll {\lvert} \xi_0 {\rvert},
\end{equation*}
in contradiction to ${\lvert} 3(\xi_1+\xi_3)^{2} - (q_1+q_3)^{2} {\rvert} \ll {\lvert} \xi_1 + \xi_3 {\rvert} \sim {\lvert} \xi_0 {\rvert}$. Consequently, it follows that ${\lvert} \xi_3 {\rvert} \sim {\lvert} (\xi_1,q_1) {\rvert}$ must be true. In this situation, two cases must be distinguished:
If ${\lvert} q_3 {\rvert} \gtrsim {\lvert} (\xi_1,q_1) {\rvert}$, then from
\begin{equation*}
{\lvert} q_1^{2} - {\lvert} (\xi_3,q_3) {\rvert}^{2} {\rvert} \leq {\lvert} {\lvert} (\xi_1,q_1) {\rvert}^{2} - {\lvert} (\xi_3,q_3) {\rvert}^{2} {\rvert} + 3 \xi_1^{2} \ll \xi_0^{2},
\end{equation*}
\begin{equation*}
{\lvert} q_2^{2} - {\lvert} (\xi_3,q_3) {\rvert}^{2} {\rvert} \leq {\lvert} {\lvert} (\xi_2,q_2) {\rvert}^{2} - {\lvert} (\xi_3,q_3) {\rvert}^{2} {\rvert} + 3 \xi_2^{2}  \ll  \xi_0^{2},
\end{equation*}
and the two constraints coming from the bilinear refinement, we again obtain the following two admissible sign configurations: $(q_1,q_2,q_3,q_0) = (\pm,\pm,\mp,\pm)$. At this stage, however, we can precisely as previously discussed, deduce that ${\lvert} q_1 + q_3 {\rvert} \ll {\lvert} \xi_0 {\rvert}$, which is in contradiction with ${\lvert} 3(\xi_1+\xi_3)^{2} - (q_1+q_3)^{2} {\rvert} \ll {\lvert} \xi_1 + \xi_3 {\rvert} \sim {\lvert} \xi_0 {\rvert}$. Thus, the sole remaining case to be examined is ${\lvert} q_3 {\rvert} \ll {\lvert} (\xi_1,q_1) {\rvert}$. In this case, the active assumptions (ii.3) and (ii.3.1) also yield ${\lvert} q_0 {\lvert} \ll {\lvert} (\xi_1,q_1) {\rvert}$ and ${\lvert} q_1 {\rvert} \sim {\lvert} q_2 {\rvert} \sim {\lvert} (\xi_1,q_1) {\rvert}$, which altogether imply that $q_1$ and $q_2$ must have opposite signs. In order to take advantage of this fact, we consider the resonance function $R$. A straightforward computation shows that
%
\begin{equation}
\label{resonance1}
R = -6(\xi _{1}+\xi _{2})(\xi _{1}+\xi _{3})(\xi _{2}+\xi _{3}) +
\sum _{i=1}^{3} \xi _{i} ({\lvert }(\xi _{0},q_{0}){\rvert }^{2} -
{\lvert }(\xi _{i},q_{i}){\rvert }^{2}),
\end{equation}
and if we now assume ${\lvert} \xi_1+\xi_2 {\rvert} \gtrsim {\lvert} \xi_0 {\rvert}^{\frac{1}{6}} \sim {\lvert} \xi_{\text{max}} {\rvert}^{\frac{1}{6}} $,
then by applying the
reverse triangle inequality and taking into account (ii.3),
${\lvert }\xi _{1}+\xi _{3}{\rvert } \sim {\lvert }\xi _{\text{max}}{\rvert }$, and
${\lvert }\xi _{2}+\xi _{3}{\rvert } \sim {\lvert }\xi _{\text{max}}{\rvert }$, it follows that
\begin{equation*}
{\lvert }R {\rvert } \gtrsim {\lvert }\xi _{\text{max}}{\rvert }^{\frac{13}{6}}.
\end{equation*}
Moreover, since
\begin{equation*}
{\lvert }R {\rvert } \leq \max _{i=0}^{3} \langle \tau _{i} - \varphi (\xi _{i},q_{i})
\rangle \eqqcolon \max _{i=0}^{3} \langle \sigma _{i} \rangle ,
\end{equation*}
we obtain
\begin{equation*}
{\lvert }\xi _{\text{max}}{\rvert }^{\frac{13}{6}} \lesssim \max _{i=0}^{3}
\langle \sigma _{i} \rangle ,
\end{equation*}
and now we must distinguish between two cases. If we have
$ \max _{i=0}^{3} \langle \sigma _{i} \rangle = \langle \sigma _{0}
\rangle $, it follows that
\begin{align*}
{\lvert }\xi _{0}{\rvert } \langle (\xi _{0},q_{0}) \rangle ^{s} \langle \sigma _{0}
\rangle ^{-\frac{1}{2}+2\epsilon } & \lesssim \prod _{i=1}^{3}
\langle (\xi _{i},q_{i}) \rangle ^{\frac{s}{3}},
\end{align*}
and undoing Plancherel, followed by H\"{o}lder's inequality yields
\begin{align*}
I_{f} &\lesssim
{\lVert }J^{\frac{s}{3}}u_{1} J^{\frac{s}{3}}u_{2} J^{\frac{s}{3}}u_{3}{\rVert }_{L_{txy}^{2}}
{\lVert }f {\rVert }_{X_{0,\frac{1}{2}-2\epsilon }}
\\
& \leq {\lVert }J^{\frac{s}{3}}u_{1}{\rVert }_{L_{Txy}^{6}}
{\lVert }J^{\frac{s}{3}}u_{2}{\rVert }_{L_{Txy}^{6}} {\lVert }J^{\frac{s}{3}}u_{3}{\rVert }_{L_{Txy}^{6}}
{\lVert }f {\rVert }_{X_{0,\frac{1}{2}-2\epsilon }}.
\end{align*}
By applying the optimized $L^{6}$-estimate \eqref{optimized L6} three times,
we arrive at
\begin{align*}
... &\lesssim {\lVert }u_{1}{\rVert }_{X_{\frac{s}{3}+\frac{2}{9}+,\frac{1}{2}+
\epsilon }}{\lVert }u_{2}{\rVert }_{X_{\frac{s}{3}+\frac{2}{9}+,\frac{1}{2}+
\epsilon }}{\lVert }u_{3}{\rVert }_{X_{\frac{s}{3}+\frac{2}{9}+,\frac{1}{2}+
\epsilon }}{\lVert }f {\rVert }_{X_{0,\frac{1}{2}-2\epsilon }}
\\
& \lesssim {\lVert }u_{1}{\rVert }_{X_{s,\frac{1}{2}+\epsilon }} {\lVert }u_{2}{\rVert }_{X_{s,
\frac{1}{2}+\epsilon }}{\lVert }u_{3}{\rVert }_{X_{s,\frac{1}{2}+\epsilon }}
{\lVert }f {\rVert }_{X_{0,\frac{1}{2}-2\epsilon }},
\end{align*}
with the last step requiring
\begin{equation*}
\frac{s}{3}+\frac{2}{9}+ \leq s \Leftrightarrow \frac{1}{3} + \leq s.
\end{equation*}
By choosing $\epsilon > 0$ sufficiently small, we therefore obtain the
desired estimate \eqref{multilin modified ZK} for every
$s > \frac{1}{3}$. On the other hand, if
$\max _{i=0}^{3} \langle \sigma _{i} \rangle \neq \langle \sigma _{0}
\rangle $ (WLOG,
$\max _{i=0}^{3} \langle \sigma _{i} \rangle = \langle \sigma _{1}
\rangle $), it follows that
\begin{equation*}
{\lvert }\xi _{0}{\rvert } \langle (\xi _{0},q_{0}) \rangle ^{s} \langle \sigma _{1}
\rangle ^{-\frac{1}{2}-\epsilon } \lesssim \langle (\xi _{1},q_{1})
\rangle ^{s} \langle (\xi _{0},q_{0}) \rangle ^{-\frac{2}{9}-}
\langle (\xi _{2},q_{2}) \rangle ^{\frac{1}{9}+} \langle (\xi _{3},q_{3})
\rangle ^{\frac{1}{9}+},
\end{equation*}
and we analogously obtain
\begin{align*}
I_{f} & \lesssim
{\lVert }J^{-\frac{2}{9}-}f J^{\frac{1}{9}+}u_{2} J^{\frac{1}{9}+}u_{3}{\rVert }_{L_{txy}^{2}}
{\lVert }u_{1}{\rVert }_{X_{s,\frac{1}{2}+\epsilon }}
\\
& \leq {\lVert }J^{-\frac{2}{9}-}f {\rVert }_{L_{Txy}^{6-}}
{\lVert }J^{\frac{1}{9}+}u_{2}{\rVert }_{L_{Txy}^{6+}}
{\lVert }J^{\frac{1}{9}+}u_{3}{\rVert }_{L_{Txy}^{6+}} {\lVert }u_{1}{\rVert }_{X_{s,
\frac{1}{2}+\epsilon }}.
\end{align*}
Applying estimate \eqref{optimized L6 dual} to the first factor, and estimate \eqref{optimized L6}, interpolated with \eqref{trivial Linfty}, to the
second and third factors, we ultimately obtain
\begin{equation*}
... \lesssim {\lVert }f {\rVert }_{X_{0,\frac{1}{2}-2\epsilon }} {\lVert }u_{2}{\rVert }_{X_{
\frac{1}{3}+,\frac{1}{2}+\epsilon }}{\lVert }u_{3}{\rVert }_{X_{\frac{1}{3}+,
\frac{1}{2}+\epsilon }} {\lVert }u_{1}{\rVert }_{X_{s,\frac{1}{2}+
\epsilon }},
\end{equation*}
which is precisely the estimate \eqref{multilin modified ZK}, valid for every
$s > \frac{1}{3}$, provided $\epsilon > 0$ is chosen small enough. We
have thus seen that in the case
${\lvert }\xi _{1}+\xi _{2}{\rvert } \gtrsim {\lvert }\xi _{0}{\rvert }^{\frac{1}{6}}$, the resonance function
leads us to the desired result. Conversely, assume that ${\lvert }\xi _{1}+\xi _{2}{\rvert } \ll {\lvert }\xi _{0}{\rvert }^{\frac{1}{6}}$. From (ii.3) and the fact that $q_1$ and $q_2$ have opposite signs, it then follows that ${\lvert} q_1+q_2 {\rvert} \ll {\lvert} \xi_0 {\rvert}^{\frac{1}{6}}$, and we now aim to exploit this smallness condition. For dyadic $N \in 2^{\mathbb{N}_{0}}$ and
$\beta \in \mathbb{R}$, we define the Fourier projectors $Q_{N}^{(\beta )}$ by setting
\begin{equation*}
Q_{N}^{(\beta )}u \coloneqq \mathcal{F}_{y}^{-1} \chi _{
{\{{\lvert }q {\rvert } \lesssim N^{\beta}\}}} \mathcal{F}_{y}u.
\end{equation*}
If we now set
\begin{equation*}
u_{1} = P_{N_{1}} u_{1}, \ u_{2} = P_{N_{2}}u_{2}, \
\text{and} \ u_{3} = P_{N_{3}} u_{3}
\end{equation*}
for dyadic numbers $N_{1},N_{2},N_{3} \in 2^{\mathbb{N}_{0}}$, it follows
that ${\lvert }q_{1}+q_{2}{\rvert } \ll N_1^{\frac{1}{6}}$, and redistributing derivatives gives\footnote{We assume without loss of generality that ${\lvert} \xi_1 {\rvert}, {\lvert} \xi_2 {\rvert} \gtrsim 1$. Otherwise, one may invoke the Sobolev embedding in $x$ without any loss of derivatives, which allows for the conclusion ${\lVert} u_{i} {\rVert}_{L_{Tx}^{4}L_{y}^{2}} \lesssim_{T} {\lVert} u_{i} {\rVert}_{X_{0,\frac{1}{2}+}}$, $i \in \{1,2\}$.}
\begin{equation*}
{\lvert} \xi_0 {\rvert} \langle (\xi_0,q_0) \rangle^{s} \lesssim N_1^{\frac{3}{4}-2s} {\lvert} \xi_0 {\rvert}^{\frac{1}{8}} {\lvert} \xi_1 {\rvert}^{\frac{1}{8}} \langle (\xi_1,q_1) \rangle^{s} {\lvert} \xi_2 {\rvert}^{\frac{1}{8}} \langle (\xi_2,q_2) \rangle^{s} {\lvert} \xi_3 {\rvert}^{\frac{1}{8}} \langle (\xi_3,q_3) \rangle^{s}.
\end{equation*}
An application of the dual version of \eqref{Airy L4 from L2}, followed by H\"{o}lder's inequality
then yields
\begin{align*}
&{\lVert }\partial _{x} (Q_{N_1}^{(\frac{1}{6})}(u_{1} u_{2}) u_{3}){\rVert }_{X_{s,-
\frac{1}{2}+2\epsilon }}
\\
& \lesssim
N_1^{\frac{3}{4}-2s} {\lVert }I_{x}^{\frac{1}{8}}(Q_{N_1}^{(\frac{1}{6})}(I_{x}^{\frac{1}{8}} J^{s}u_{1} I_{x}^{\frac{1}{8}} J^{s}u_{2}) I_{x}^{\frac{1}{8}} J^{s}u_{3}){\rVert }_{X_{0,-
\frac{1}{2}+2\epsilon }}
\\
& \lesssim
N_1^{\frac{3}{4}-2s} {\lVert }Q_{N_1}^{(\frac{1}{6})}(I_{x}^{\frac{1}{8}} J^{s}u_{1} I_{x}^{\frac{1}{8}} J^{s}u_{2}) I_{x}^{\frac{1}{8}} J^{s}u_{3}{\rVert }_{L_{tx}^{
\frac{4}{3}} L_{y}^{2}}
\\
&\leq
N_1^{\frac{3}{4}-2s} {\lVert }Q_{N_1}^{(\frac{1}{6})}(I_{x}^{\frac{1}{8}} J^{s}u_{1} I_{x}^{\frac{1}{8}} J^{s}u_{2}){\rVert }_{L_{tx}^{2}
L_{y}^{\infty}}
{\lVert }I_{x}^{\frac{1}{8}} J^{s}u_{3}{\rVert }_{L_{tx}^{4}L_{y}^{2}},
\end{align*}
and the inner $L_{y}^{\infty}$-norm can be estimated in the following manner:
\begin{align*}
{\lVert }Q_{N_1}^{(\frac{1}{6})}(I_{x}^{\frac{1}{8}} J^{s}u_{1} I_{x}^{\frac{1}{8}} J^{s}u_{2}){\rVert }_{L_{y}^{
\infty}} & \lesssim
{\lVert } \chi _{{\{{\lvert }q {\rvert } \lesssim N_1^{\frac{1}{6}}\}}} \Fcal _{y}(I_{x}^{\frac{1}{8}} J^{s}u_{1} I_{x}^{\frac{1}{8}} J^{s}u_{2}){\rVert }_{L_{q}^{1}}
\\
& \leq {\lVert }\chi _{{\{{\lvert }q {\rvert } \lesssim N_1^{\frac{1}{6}}\}}}{\rVert }_{L_{q}^{1}}
{\lVert }\Fcal _{y}(I_{x}^{\frac{1}{8}} J^{s}u_{1} I_{x}^{\frac{1}{8}} J^{s}u_{2}){\rVert }_{L_{q}^{
\infty}}
\\
& \lesssim
N_1^{\frac{1}{6}} {\lVert }I_{x}^{\frac{1}{8}} J^{s}u_{1} I_{x}^{\frac{1}{8}} J^{s}u_{2}{\rVert }_{L_{y}^{1}}.
\end{align*}
With this, we further obtain
\begin{align*}
{\lVert }\partial _{x} (Q_{N_1}^{(\frac{1}{6})}(u_{1} u_{2}) u_{3}){\rVert }_{X_{s,-
\frac{1}{2}+2\epsilon }} & \lesssim
N_1^{\frac{3}{4}+\frac{1}{6}-2s}{\lVert }I_{x}^{\frac{1}{8}} J^{s}u_{1} I_{x}^{\frac{1}{8}} J^{s}u_{2}{\rVert }_{L_{tx}^{2}L_{y}^{1}}
{\lVert }I_{x}^{\frac{1}{8}} J^{s}u_{3}{\rVert }_{L_{tx}^{4}L_{y}^{2}}
\\
& \leq N_1^{\frac{3}{4}+\frac{1}{6}-2s} \prod _{i=1}^{3}
{\lVert }I_{x}^{\frac{1}{8}} J^{s}u_{i}{\rVert }_{L_{tx}^{4}L_{y}^{2}},
\end{align*}
and three applications of \eqref{Airy L4 from L2} ultimately give
\begin{align*}
... \lesssim N_1^{\frac{3}{4}+\frac{1}{6}-2s} \prod _{i=1}^{3} {\lVert }u_{i}{\rVert }_{X_{s,\frac{1}{2}+\epsilon }}.
\end{align*}
The desired estimate \eqref{multilin modified ZK} now follows upon performing a dyadic summation over all $N_1 \sim N_2 \sim N_3 \gg 1$. This argument works provided that 
\begin{equation*}
\frac{3}{4} + \frac{1}{6} - 2s < 0 \Leftrightarrow \frac{11}{24} < s
\end{equation*}
and $\epsilon > 0$ is chosen sufficiently small; this concludes the discussion of this subcase. \\
(ii.3.2) \underline{${\lvert }\xi _{\text{min}}{\rvert } \ll {\lvert }\xi _{\text{med}}{\rvert } \sim {\lvert }\xi _{\text{max}}{\rvert }$}: \\
Under this assumption, the pointwise estimates
 ${\lvert }\xi _{0}{\rvert } \lesssim {\lvert }\xi _{1}+\xi _{2}{\rvert }$ and
${\lvert }\xi _{0}{\rvert } \lesssim {\lvert }\xi _{1}+\xi _{3}{\rvert }$ hold. Thus, if
${\lvert }3(\xi _{i}+\xi _{j})^{2}-(q_{i}+q_{j})^{2}{\rvert } \gtrsim
{\lvert }\xi _{i}+\xi _{j}{\rvert }$ holds for at least one tuple
$(i,j) \in {\{(1,2),(1,3)\}}$, then we may argue exactly as
in case (ii.3.1), so that \eqref{bilinrefinement} and \eqref{dual refinement} lead to the desired multilinear estimate \eqref{multilin modified ZK}, valid for every $s > \frac{1}{4}$. Consequently,
we may assume
\begin{equation*}
{\lvert }3(\xi _{1}+\xi _{2})^{2}-(q_{1}+q_{2})^{2}{\rvert } \ll
{\lvert }\xi _{1}+\xi _{2}{\rvert } \sim {\lvert }\xi _{3}{\rvert }
\end{equation*}
and
\begin{equation*}
{\lvert }3(\xi _{1}+\xi _{3})^{2}-(q_{1}+q_{3})^{2}{\rvert } \ll
{\lvert }\xi _{1}+\xi _{3}{\rvert } \sim {\lvert }\xi _{3}{\rvert }.
\end{equation*}
Let us now first consider the case where ${\lvert} \xi_3 {\rvert} \ll {\lvert} (\xi_1,q_1) {\rvert}$.
Then, invoking the general assumption (ii.3), we obtain
\begin{equation*}
{\lvert} q_1 {\rvert} \sim {\lvert} q_2 {\rvert} \sim {\lvert} q_3 {\rvert} \sim {\lvert} q_0 {\rvert} \sim {\lvert} (\xi_1,q_1) {\rvert},
\end{equation*}
and in combination with the two constraints from the bilinear refinement, only the sign distributions $(q_1,q_2,q_3,q_0) = (\pm,\mp,\mp,\mp)$ remain possible. We once again distinguish two cases. In the case ${\lvert} \xi_2+\xi_3 {\rvert} \gtrsim {\lvert} \xi_0 {\rvert}^{\frac{1}{3}}$, the resonance function $R$ in the form \eqref{resonance1} provides a remedy: taking into account (ii.3), ${\lvert} \xi_1+\xi_2 {\rvert} \sim {\lvert} \xi_{\text{max}} {\rvert}$, and ${\lvert} \xi_1+\xi_3 {\rvert} \sim {\lvert} \xi_{\text{max}} {\rvert}$, we obtain
\begin{equation*}
{\lvert} R {\rvert} \gtrsim {\lvert} \xi_0 {\rvert}^{\frac{1}{3}} {\lvert} \xi_{\text{max}} {\rvert}^{2},
\end{equation*}
and this - as already established in case (ii.3.1) - implies, at the very least, that an arbitrary $s > \frac{1}{3}$ can be reached in the multilinear estimate \eqref{multilin modified ZK}. Conversely, if ${\lvert} \xi_2+\xi_3 {\rvert} \ll {\lvert} \xi_0 {\rvert}^{\frac{1}{3}}$, it follows that ${\lvert} \xi_0-\xi_1 {\rvert} \ll {\lvert} \xi_0 {\rvert}^{\frac{1}{3}}$, and, upon taking into account (ii.3), the magnitudes, and the signs of the $q_i$, one consequently obtains ${\lvert} q_2-q_3 {\rvert} \ll {\lvert} \xi_0 {\rvert}^{\frac{1}{3}}$, $ {\lvert} q_0+q_1 {\rvert} \ll {\lvert} \xi_0 {\rvert}^{\frac{1}{3}}$, and hence
${\lvert} q_1+q_2 {\rvert} \ll {\lvert} \xi_0 {\rvert}^{\frac{1}{3}}$.
Moreover, we have
\begin{equation*}
{\lvert} \xi_0 {\rvert} \leq {\lvert} \xi_1 {\rvert} + {\lvert} \xi_2+\xi_3 {\rvert} \leq {\lvert} \xi_1 {\rvert} + \epsilon_0 {\lvert} \xi_0 {\rvert}^{\frac{1}{3}},
\end{equation*}
and thus, ${\lvert} \xi_0 {\rvert} \lesssim {\lvert} \xi_1 {\rvert}$, so that the Airy $L_{tx}^{4}L_{y}^{2}$-estimate \eqref{Airy L4 from L2} in the previously carried out smallness argument can be applied a total of four times: Again, for dyadic numbers $N_1,N_2,N_3 \in 2^{\mathbb{N}}$, we set
\begin{equation*}
u_i = P_{N_i}u_i, \quad i \in \{1,2,3\},
\end{equation*}
and obtain ${\lvert} q_1+q_2 {\rvert} \ll N_1^{\frac{1}{3}}$ as well as
\begin{equation*}
{\lvert} \xi_0 {\rvert} \langle (\xi_0,q_0) \rangle^{s} \lesssim N_1^{\frac{1}{2}-2s} {\lvert} \xi_0 {\rvert}^{\frac{1}{8}} \prod_{i=1}^{3} {\lvert} \xi_i {\rvert}^{\frac{1}{8}} \langle (\xi_i,q_i) \rangle^{s},
\end{equation*}
which altogether - computing exactly as in case (ii.3.1) - yields
\begin{equation*}
{\lVert }\partial _{x} (Q_{N_1}^{(\frac{1}{3})}(u_{1} u_{2}) u_{3}){\rVert }_{X_{s,-
\frac{1}{2}+2\epsilon }}  \lesssim N_1^{\frac{1}{2}+\frac{1}{3}-2s} \prod_{i=1}^{3} {\lVert} u_i {\rVert}_{X_{s,\frac{1}{2}+\epsilon}}.
\end{equation*}
A dyadic summation over all $N_1 \sim N_2 \sim N_3 \gg 1$ now leads to the desired multilinear estimate \eqref{multilin modified ZK}, valid for all
\begin{equation*}
\frac{1}{2}+\frac{1}{3}-2s < 0 \Leftrightarrow \frac{5}{12} < s,
\end{equation*}
provided $\epsilon > 0$ is chosen sufficiently small. With this, the discussion of the case ${\lvert} \xi_3 {\rvert} \ll {\lvert} (\xi_1,q_1) {\rvert}$ is complete. We now consider the case ${\lvert} \xi_3 {\rvert} \sim {\lvert} (\xi_1,q_1) {\rvert}$.
If ${\lvert} \xi_2+\xi_3 {\rvert} \gtrsim {\lvert} \xi_0 {\rvert}^{\frac{1}{3}}$ to begin with, no further steps are required, as we then have
\begin{equation*}
{\lvert} R {\rvert} \gtrsim {\lvert} \xi_0 {\rvert}^{\frac{1}{3}} {\lvert} \xi_{\text{max}} {\rvert}^{2}
\end{equation*}
(observe \eqref{resonance1}, (ii.3), ${\lvert} \xi_1+\xi_2 {\rvert} \sim {\lvert} \xi_{\text{max}} {\rvert}$, and ${\lvert} \xi_1+\xi_3 {\rvert} \sim {\lvert} \xi_{\text{max}} {\rvert}$), and, as previously computed in case (ii.3.1), this ensures (at the very least) that all values $s > \frac{1}{3}$ can be realized within \eqref{multilin modified ZK}. We may thus assume ${\lvert} \xi_2+\xi_3 {\rvert} \ll {\lvert} \xi_0 {\rvert}^{\frac{1}{3}}$, which directly implies ${\lvert} \xi_0 {\rvert} \lesssim {\lvert} \xi_1 {\rvert}$, and this fact will play a crucial role in a subsequent smallness argument.
Now, if ${\lvert} q_3 {\rvert} \gtrsim {\lvert} (\xi_1,q_1) {\rvert}$ were true, then, in conjunction with the active assumptions (ii.3), (ii.3.1), ${\lvert} \xi_0 - \xi_1 {\rvert} = {\lvert} \xi_2 + \xi_3 {\rvert} \ll {\lvert} \xi_0 {\rvert}^{\frac{1}{3}}$, and the two constraints arising from the inapplicability of the bilinear refinement, we would obtain, on the one hand, the relations
\begin{equation*}
{\lvert} q_1 {\rvert} \sim {\lvert} (\xi_1,q_1) {\rvert}, \ {\lvert} {\lvert} q_1 {\rvert} - {\lvert} q_0 {\rvert} {\rvert} \ll {\lvert} \xi_0 {\rvert}^{\frac{1}{3}}, \ \text{and} \ {\lvert} {\lvert} q_2 {\rvert} - {\lvert} q_3 {\rvert} {\rvert} \ll {\lvert} \xi_0 {\rvert}^{\frac{1}{3}},
\end{equation*}
and on the other hand, the sign configurations $(q_1,q_2,q_3,q_0) = (\pm,\mp,\mp,\mp)$. These constraints, however, can only hold simultaneously in the case ${\lvert} q_1 + q_2 {\rvert} \ll {\lvert} \xi_0 {\rvert}^{\frac{1}{3}}$, which would contradict ${\lvert} 3(\xi_1+\xi_2)^2-(q_1+q_2)^2 {\rvert} \ll {\lvert} \xi_1+\xi_2 {\rvert}$. We must therefore have ${\lvert} q_3 {\rvert} \ll {\lvert} (\xi_1,q_1) {\rvert}$ and consequently, by ${\lvert} \xi_2 + \xi_3 {\rvert} \ll {\lvert} \xi_0 {\rvert}^{\frac{1}{3}}$ and (ii.3), also ${\lvert} q_2 {\rvert} \ll {\lvert} (\xi_1,q_1) {\rvert}$.
Now, taking into account ${\lvert} q_1 {\rvert} \sim {\lvert} (\xi_1,q_1) {\rvert}$, it follows that $q_1$ and $q_0$ must have the same sign, and together with (ii.3) and ${\lvert} \xi_0-\xi_1 {\rvert} \ll {\lvert} \xi_0 {\rvert}^{\frac{1}{3}}$, this further implies that ${\lvert} q_2 + q_3 {\rvert} = {\lvert} q_0 - q_1 {\rvert} \ll {\lvert} \xi_0 {\rvert}^{\frac{1}{3}}$ must hold.
After passing to dyadic pieces once again and taking into account ${\lvert} \xi_0 {\rvert} \lesssim {\lvert} \xi_1 {\rvert}$, we then obtain - as in the previous discussion - the estimate
\begin{equation*}
{\lVert }\partial _{x} (u_{1} Q_{N_1}^{(\frac{1}{3})}(u_{2} u_{3})){\rVert }_{X_{s,-
\frac{1}{2}+2\epsilon }}  \lesssim N_1^{\frac{1}{2}+\frac{1}{3}-2s} \prod_{i=1}^{3} {\lVert} u_i {\rVert}_{X_{s,\frac{1}{2}+\epsilon}},
\end{equation*}
which yields the desired trilinear estimate \eqref{multilin modified ZK} after dyadic summation over all $N_1 \sim N_2 \sim N_3 \gg 1$, provided that $\epsilon > 0$ is chosen sufficiently small and
\begin{equation*}
\frac{1}{2}+\frac{1}{3}-2s < 0 \Leftrightarrow \frac{5}{12} < s
\end{equation*}
is satisfied. Thus, this subcase is also resolved. \\
(ii.3.3) \underline{${\lvert }\xi _{\text{min}}{\rvert } \sim {\lvert }\xi _{\text{med}}{\rvert } \sim {\lvert }\xi _{\text{max}}{\rvert }$}: \\
(ii.3.3.1) \underline{${\lvert} \xi_{\text{max}} {\rvert} \ll {\lvert} (\xi_1,q_1) {\rvert}$}: \\
In this situation, it follows from the active constraint (ii.3) that
\begin{equation*}
{\lvert} q_1 {\rvert} \sim {\lvert} q_2 {\rvert} \sim {\lvert} q_3 {\rvert} \sim {\lvert} q_0 {\rvert} \sim {\lvert} (\xi_1,q_1) {\rvert}, \ \text{with} \ {\lvert} {\lvert} q_i {\rvert} - {\lvert} q_j {\rvert} {\rvert} \ll {\lvert} (\xi_1,q_1) {\rvert} \ \forall i,j \in \{0,1,2,3\},
\end{equation*}
and this in turn implies that the $q_i$ cannot all have the same sign. Without loss of generality, let us assume that we have the following sign configuration: $(q_1,q_2,q_3,q_0) = (+,-,+,+)$. Then the assumption ${\lvert} \xi_1+\xi_3 {\rvert} \gtrsim {\lvert} \xi_0 {\rvert}^{\frac{1}{3}}$ allows us to infer the pointwise bound
\begin{align*}
{\lvert} \xi_0 {\rvert} \langle (\xi_0,q_0) \rangle^{s} &\lesssim {\lvert} \xi_1+\xi_3 {\rvert}^{\frac{1}{4}} \langle (\xi_1,q_1) \rangle^{s-} \langle (\xi_3,q_3) \rangle^{\frac{5}{12}+} {\lvert} \xi_0 + (-\xi_2) {\rvert}^{\frac{1}{4}} \langle (\xi_0,q_0) \rangle^{0-} \\ & \ \ \ \cdot \langle (-\xi_2,-q_2) \rangle^{\frac{5}{12}+},
\end{align*}
and undoing Plancherel, followed by Hölder's inequality and a subsequent application of \eqref{bilinrefinement} and \eqref{dual refinement}, yields
\begin{align*}
I_f &\lesssim {\lVert} I_{x}^{\frac{1}{4}}P^{1}(J^{s-}u_1 J^{\frac{5}{12}+} u_3) {\rVert}_{L_{txy}^{2}} {\lVert} I_{x}^{\frac{1}{4}}P^{1}(J^{0-}f J^{\frac{5}{12}+} \tilde{u}_2) {\rVert}_{L_{txy}^{2}}
\\ & \lesssim {\lVert} f {\rVert}_{X_{0,\frac{1}{2}-2\epsilon}} {\lVert} u_1 {\rVert}_{X_{s,\frac{1}{2}+\epsilon}} {\lVert} u_2 {\rVert}_{X_{\frac{5}{12}+,\frac{1}{2}+\epsilon}} {\lVert} u_3 {\rVert}_{X_{\frac{5}{12}+,\frac{1}{2}+\epsilon}}
\\ & \lesssim {\lVert} f {\rVert}_{X_{0,\frac{1}{2}-2\epsilon}} \prod_{i=1}^{3} {\lVert} u_i {\rVert}_{X_{s,\frac{1}{2}+\epsilon}},
\end{align*}
with the last step being valid for every $s > \frac{5}{12}$, provided $\epsilon > 0$ is chosen sufficiently small.
Here, the applicability of \eqref{bilinrefinement} and \eqref{dual refinement} is justified by the fact that, since ${\lvert} \xi_1+\xi_3 {\rvert} \ll {\lvert} (\xi_1,q_1) {\rvert}$ and ${\lvert} q_1+q_3 {\rvert} \sim {\lvert} (\xi_1,q_1) {\rvert}$, it follows that
\begin{equation*}
{\lvert} 3(\xi_1+\xi_3)^2 - (q_1+q_3)^2 {\rvert} \sim {\lvert} (\xi_1,q_1) {\rvert}^{2} \gg {\lvert} \xi_1+\xi_3 {\rvert}.
\end{equation*}
It therefore remains to be examined what happens in the case ${\lvert} \xi_1 + \xi_3 {\rvert} \ll {\lvert} \xi_0 {\rvert}^{\frac{1}{3}}$: Then it follows from (ii.3), from ${\lvert} \xi_0 - \xi_2 {\rvert} = {\lvert} \xi_1+\xi_3 {\rvert} \ll {\lvert} \xi_0 {\rvert}^{\frac{1}{3}}$, and from the sign configuration of the $q_i$, that ${\lvert} q_1-q_3 {\rvert} \ll {\lvert} \xi_0 {\rvert}^{\frac{1}{3}}$, ${\lvert} q_0 + q_2 {\rvert} \ll {\lvert} \xi_0 {\rvert}^{\frac{1}{3}}$, and thus ${\lvert} q_1+q_2 {\rvert} \ll {\lvert} \xi_0 {\rvert}^{\frac{1}{3}}$, which allows us to proceed exactly as in case (ii.3.2) (the active assumption (ii.3.3) allows the Airy $L_{tx}^4L_y^2$-estimate \eqref{Airy L4 from L2} to be applied beneficially four times here as well) to obtain the desired trilinear estimate for all $s > \frac{5}{12}$, provided $\epsilon > 0$ is chosen sufficiently small. \\
(ii.3.3.2) \underline{${\lvert} \xi_{\text{max}} {\rvert} \sim {\lvert} (\xi_1,q_1) {\rvert}$}: \\
(ii.3.3.2.1) \underline{${\lvert }\xi _{i}+\xi _{j}{\rvert } \gtrsim {\lvert }\xi _{\text{max}}{\rvert }\text{, for all }(i,j) \in {\{(1,2),(1,3),(2,3)\}}$}: \\
In this case, we have, due to the general assumption (ii.3), 
\begin{equation*}
{\lvert }R {\rvert } \gtrsim {\lvert }\xi _{\text{max}}{\rvert }^{3},
\end{equation*}
which allows us to control up to $\frac{3}{2}-$ derivatives by means of
the quantities $\langle \sigma _{i} \rangle $. Proceeding exactly as in
case (ii.3.1), we thus arrive at \eqref{multilin modified ZK}, and this
holds for every $s > \frac{1}{3}$, provided we choose
$\epsilon > 0$ sufficiently small. \\
(ii.3.3.2.2) \underline{${\lvert }\xi _{i}+\xi _{j}{\rvert } \gtrsim {\lvert }\xi _{\text{max}}{\rvert }\text{, for exactly two }(i,j) \in {\{(1,2),(1,3),(2,3)\}}$}: \\
Without loss of generality, we may assume that the sign pattern is $(\xi_1,\xi_2,\xi_3,\xi_0) = (+,-,+,+)$ with ${\lvert} \xi_1 + \xi_2 {\rvert} \gtrsim {\lvert} \xi_{\text{max}} {\rvert}$ and ${\lvert} \xi_2 + \xi_3 {\rvert} \ll {\lvert} \xi_{\text{max}} {\rvert}$.
If now ${\lvert} 3(\xi_1+\xi_3)^2 - (q_1+q_3)^2 {\rvert} \gtrsim {\lvert} \xi_1+\xi_3 {\rvert}$ or ${\lvert} 3(\xi_1+\xi_2)^2 - (q_1+q_2)^2 {\rvert} \gtrsim {\lvert} \xi_1+\xi_2 {\rvert}$, then the bilinear refinement of the linear $L^4$-estimate applies exactly as in case (ii.3.1) and yields $s > \frac{1}{4}$. Thus we may restrict ourselves to the situation where
\begin{equation*}
{\lvert} 3(\xi_1+\xi_3)^2 - (q_1+q_3)^2 {\rvert} \ll {\lvert} \xi_1+\xi_3 {\rvert}
\end{equation*}
and
\begin{equation*}
{\lvert} 3(\xi_1+\xi_2)^2 - (q_1+q_2)^2 {\rvert} \ll {\lvert} \xi_1+\xi_2 {\rvert}
\end{equation*}
hold. Furthermore, we may assume that ${\lvert} \xi_2 + \xi_3 {\rvert} \ll {\lvert} \xi_0 {\rvert}^{\frac{1}{3}}$, since otherwise, by (ii.3), we would have
\begin{equation*}
{\lvert} R {\rvert} \gtrsim {\lvert} \xi_0 {\rvert}^{\frac{1}{3}} {\lvert} \xi_{\text{max}} {\rvert}^{2},
\end{equation*}
so that the resonance function - again as in case (ii.3.1) - already leads to at least $s > \frac{1}{3}$. In addition to that, we must have ${\lvert} q_1 {\rvert} \sim {\lvert} (\xi_1,q_1) {\rvert}$ and ${\lvert} q_3 {\rvert} \sim {\lvert} (\xi_1,q_1) {\rvert}$, and because of (ii.3) together with ${\lvert} \xi_0 - \xi_1 {\rvert} = {\lvert} \xi_2 + \xi_3 {\rvert} \ll {\lvert} \xi_0 {\rvert}^{\frac{1}{3}}$, we then also obtain ${\lvert} q_2 {\rvert} \sim {\lvert} (\xi_1,q_1) {\rvert}$ and ${\lvert} q_0 {\rvert} \sim {\lvert} (\xi_1,q_1) {\rvert}$ with ${\lvert} {\lvert} q_1 {\rvert} - {\lvert} q_0 {\rvert} {\rvert} \ll {\lvert} \xi_0 {\rvert}^{\frac{1}{3}}$ and ${\lvert} {\lvert} q_2 {\rvert} - {\lvert} q_3 {\rvert} {\rvert} \ll {\lvert} \xi_0 {\rvert}^{\frac{1}{3}}$. Otherwise (taking into account (ii.3) and the signs of the $\xi_i$) we would have
\begin{equation*}
{\lvert} 3(\xi_1+\xi_3)^2 - (q_1+q_3)^2 {\rvert} \sim {\lvert} (\xi_1,q_1) {\rvert}^2 \gg {\lvert} \xi_1+\xi_3 {\rvert},
\end{equation*}
contradicting ${\lvert} 3(\xi_1+\xi_3)^2 - (q_1+q_3)^2 {\rvert} \ll {\lvert} \xi_1+\xi_3 {\rvert}$.
We are left with two cases:
\begin{enumerate}

\item[(1)] If $q_2$ and $q_3$ have opposite signs, then ${\lvert} q_2 + q_3 {\rvert} \ll {\lvert} \xi_0 {\rvert}^{\frac{1}{3}}$, and the smallness argument presented in case (ii.3.2) applies, yielding $s > \frac{5}{12}$ (the Airy $L_{tx}^4L_y^2$-estimate can be used four times due to the active assumption (ii.3.3)).

\item[(2)] If $q_2$ and $q_3$ have the same sign, then, since $q_1+q_2+q_3=q_0$, the quantities $q_0$ and $q_1$ must have opposite signs. Combining ${\lvert} q_2 - q_3 {\rvert} \ll {\lvert} \xi_0 {\rvert}^{\frac{1}{3}}$ with ${\lvert} q_0 + q_1 {\rvert} \ll {\lvert} \xi_0 {\rvert}^{\frac{1}{3}}$ then implies ${\lvert} q_1 + q_2 {\rvert} \ll {\lvert} \xi_0 {\rvert}^{\frac{1}{3}}$, so that we may again invoke the smallness argument to obtain $s > \frac{5}{12}$.
\end{enumerate}
(ii.3.3.2.3) \underline{${\lvert }\xi _{i}+\xi _{j}{\rvert } \gtrsim {\lvert }\xi _{\text{max}}{\rvert }\text{, for exactly one }(i,j) \in {\{(1,2),(1,3),(2,3)\}}$}: \\
We again consider, without loss of generality, the sign configuration $(\xi_1,\xi_2,\xi_3,\xi_0) = (+,-,+,+)$.
Since ${\lvert} \xi_0 {\rvert} \lesssim {\lvert} \xi_1+\xi_3 {\rvert}$, we may infer the pointwise bound
\begin{align*}
{\lvert} \xi_0 {\rvert} \langle (\xi_0,q_0) \rangle^s &\lesssim {\lvert} \xi_1+\xi_3 {\rvert}^{\frac{1}{24}} \langle (\xi_1,q_1) \rangle^{s-} \langle (\xi_3,q_3) \rangle^{\frac{11}{24}+} {\lvert} \xi_0 + (-\xi_2) {\rvert}^{\frac{1}{24}} \langle (\xi_0,q_0) \rangle^{0-} \\ & \ \ \ \cdot \langle (-\xi_2,-q_2) \rangle^{\frac{11}{24}+},
\end{align*}
and under the assumption ${\lvert} 3(\xi_1+\xi_3)^2-(q_1+q_3)^2 {\rvert} \gtrsim {\lvert} \xi_1+\xi_3 {\rvert}^{\frac{1}{6}}$, undoing Plancherel, followed by Hölder's inequality, and a subsequent application of \eqref{bilinrefinement} and \eqref{dual refinement} yields
\begin{align*}
I_f &\lesssim {\lVert} I_x^{\frac{1}{24}}P^{\frac{1}{6}}(J^{s-}u_1 J^{\frac{11}{24}+} u_3) {\rVert}_{L_{txy}^2} {\lVert} I_x^{\frac{1}{24}}P^{\frac{1}{6}}(J^{0-}f J^{\frac{11}{24}+} \tilde{u}_2) {\rVert}_{L_{txy}^2}
\\ & \lesssim {\lVert} f {\rVert}_{X_{0,\frac{1}{2}-2\epsilon}} {\lVert} u_1 {\rVert}_{X_{s,\frac{1}{2}+\epsilon}} {\lVert} u_2 {\rVert}_{X_{\frac{11}{24}+,\frac{1}{2}+\epsilon}} {\lVert} u_3 {\rVert}_{X_{\frac{11}{24}+,\frac{1}{2}+\epsilon}}
\\ &\lesssim {\lVert} f {\rVert}_{X_{0,\frac{1}{2}-2\epsilon}} \prod_{i=1}^{3} {\lVert} u_i {\rVert}_{X_{s,\frac{1}{2}+\epsilon}},
\end{align*}
with the last step requiring $s>\frac{11}{24}$ to hold. We may thus assume
\begin{equation} \label{bilinconstraint}
{\lvert} 3(\xi_1+\xi_3)^2 - (q_1+q_3)^2 {\rvert} \ll {\lvert} \xi_1+\xi_3 {\rvert}^{\frac{1}{6}}.
\end{equation}
From the general assumptions (ii.3), (ii.3.3), the signs of the $\xi_i$, ${\lvert} \xi_1+\xi_2 {\rvert} \ll {\lvert} \xi_0 {\rvert}$, and ${\lvert} \xi_2+\xi_3 {\rvert} \ll {\lvert} \xi_0 {\rvert}$, it then follows that
\begin{equation*}
{\lvert} q_1 {\rvert} \sim {\lvert} q_2 {\rvert} \sim {\lvert} q_3 {\rvert} \sim {\lvert} q_0 {\rvert} \sim {\lvert} (\xi_1,q_1) {\rvert} \ \text{with} \ {\lvert} {\lvert} q_i {\rvert} - {\lvert} q_j {\rvert} {\rvert} \ll {\lvert} \xi_0 {\rvert} \ \forall i,j \in \{0,1,2,3\},
\end{equation*}
and the constraint \eqref{bilinconstraint} additionally forces the sign patterns $(q_1,q_2,q_3,q_0) = \\ (\pm,\mp,\pm,\pm)$.
If now
\begin{equation*}
{\lvert} \xi_1+\xi_2 {\rvert} \ll {\lvert} \xi_0 {\rvert}^{\frac{5}{12}} \ \text{or} \ {\lvert} \xi_2+\xi_3 {\rvert} \ll {\lvert} \xi_0 {\rvert}^{\frac{5}{12}},
\end{equation*}
then by (ii.3), taking into account the signs of the $q_i$, we also obtain
\begin{equation*}
{\lvert} q_1+q_2 {\rvert} \ll {\lvert} \xi_0 {\rvert}^{\frac{5}{12}} \ \text{or} \ {\lvert} q_2+q_3 {\rvert} \ll {\lvert} \xi_0 {\rvert}^{\frac{5}{12}}.
\end{equation*}
Passing to dyadic pieces (cf. the smallness argument in case (ii.3.1) and case (ii.3.2)) then yields
\begin{equation*}
{\lVert }\partial _{x} (Q_{N_1}^{(\frac{5}{12})}(u_{1} u_{2}) u_{3}){\rVert }_{X_{s,-
\frac{1}{2}+2\epsilon }}  \lesssim N_1^{\frac{1}{2}+\frac{5}{12}-2s} \prod_{i=1}^{3} {\lVert} u_i {\rVert}_{X_{s,\frac{1}{2}+\epsilon}}
\end{equation*}
or
\begin{equation*}
{\lVert }\partial _{x} (u_{1} Q_{N_1}^{(\frac{5}{12})}(u_{2} u_{3})){\rVert }_{X_{s,-
\frac{1}{2}+2\epsilon }}  \lesssim N_1^{\frac{1}{2}+\frac{5}{12}-2s} \prod_{i=1}^{3} {\lVert} u_i {\rVert}_{X_{s,\frac{1}{2}+\epsilon}},
\end{equation*}
and after dyadic summation this gives the desired trilinear estimate \eqref{multilin modified ZK}, provided
\begin{equation*}
\frac{1}{2}+\frac{5}{12}-2s < 0 \Leftrightarrow \frac{11}{24} < s
\end{equation*}
and $\epsilon > 0$ is sufficiently small. We may therefore assume
\begin{equation*}
{\lvert} \xi_1+\xi_2 {\rvert} \gtrsim {\lvert} \xi_0 {\rvert}^{\frac{5}{12}} \ \text{and} \ {\lvert} \xi_2+\xi_3 {\rvert} \gtrsim {\lvert} \xi_0 {\rvert}^{\frac{5}{12}}.
\end{equation*}
Moreover, the assumption ${\lvert} q_1 - q_3 {\rvert} \gtrsim {\lvert} \xi_0 {\rvert}^{\frac{3}{4}}$ implies, by (ii.3), that ${\lvert} \xi_1 - \xi_3 {\rvert} \gtrsim {\lvert} \xi_0 {\rvert}^{\frac{3}{4}}$, and therefore also
\begin{equation*}
{\lvert} \xi_1+\xi_2 {\rvert} \gtrsim {\lvert} \xi_0 {\rvert}^{\frac{3}{4}} \ \text{or} \ {\lvert} \xi_2+\xi_3 {\rvert} \gtrsim {\lvert} \xi_0 {\rvert}^{\frac{3}{4}},
\end{equation*}
as otherwise we would have
\begin{equation*}
{\lvert} \xi_0 {\rvert}^{\frac{3}{4}} \lesssim {\lvert} \xi_1 - \xi_3 {\rvert} \leq {\lvert} \xi_1 + \xi_2 {\rvert} + {\lvert} \xi_2 + \xi_3 {\rvert} \ll {\lvert} \xi_0 {\rvert}^{\frac{3}{4}}.
\end{equation*}
In this case, however, since ${\lvert} \xi_1+\xi_2 {\rvert}, {\lvert} \xi_1+\xi_2 {\rvert} \gtrsim {\lvert} \xi_0 {\rvert}^{\frac{5}{12}}$, we obtain
\begin{equation*}
{\lvert} \xi_1+\xi_2 {\rvert} {\lvert} \xi_2+\xi_3 {\rvert} {\lvert} \xi_1+\xi_3 {\rvert} \gtrsim {\lvert} \xi_0 {\rvert}^{\frac{13}{6}},
\end{equation*}
and together with (ii.3) this yields
\begin{equation*}
{\lvert} R {\rvert} \gtrsim {\lvert} \xi_0 {\rvert}^{\frac{13}{6}}.
\end{equation*}
As previously discussed, this already suffices to obtain at least $s > \frac{1}{3}$, so that we can also assume ${\lvert} q_1 - q_3 {\rvert} \ll {\lvert} \xi_0 {\rvert}^{\frac{3}{4}}$. Moreover, these considerations work entirely analogously for the pairing $ q_0 + q_2$, which allows us to include ${\lvert} q_0 + q_2 {\rvert} \ll {\lvert} \xi_0 {\rvert}^{\frac{3}{4}}$ among our remaining active constraints as well.
As a final preparatory step for our last argument, we write
\begin{equation*}
f = P_{N} f, \ u_{1} = P_{N_{1}} u_{1}, \ u_{2} = P_{N_{2}}u_{2}, \
\text{and} \ u_{3} = P_{N_{3}} u_{3}
\end{equation*}
for dyadic numbers $N \sim N_1 \sim N_2 \sim N_3 \gg 1$, which allows us to summarize all the relevant remaining constraints in the form:
\begin{equation} \label{lastconstraints}
{\lvert} \sqrt{3}(\xi_1+\xi_3) \mp (q_1+q_3) {\rvert} \ll N^{-\frac{5}{6}}, \ {\lvert} q_1 - q_3 {\rvert} \ll N^{\frac{3}{4}}, \ \text{and} \ {\lvert} q_0 + q_2 {\rvert} \ll N^{\frac{3}{4}}.
\end{equation}
Let us now define, for $\alpha, \beta \in \mathbb{R}$ and dyadic $N \gg 1$, the bilinear restriction operator $S_N^{(\alpha,\beta)}$ by means of its Fourier transform:
\begin{align*}
&\widehat{S_N^{(\alpha,\beta)}(u,v)}(\tau,\xi,q) \coloneqq \\ & \int_{\mathbb{R}^2} \sum_{q_1 \in \mathbb{Z}} \chi_{\{{\lvert} \sqrt{3}\xi \mp q {\rvert} \lesssim N^\alpha\}} \chi_{\{{\lvert} 2q_1 - q {\rvert} \lesssim N^\beta\}} \hat{u}(\tau_1,\xi_1,q_1) \hat{v}(\tau-\tau_1,\xi-\xi_1,q-q_1) \mathrm{d}(\tau_1,\xi_1).
\end{align*}
Then, taking into account
\begin{equation*}
{\lvert} \xi_0 {\rvert} \langle (\xi_0,q_0) \rangle^{s} \lesssim N^{1-2s+} \langle (\xi_1,q_1) \rangle^{s} \langle (\xi_3,q_3) \rangle^{s} \langle (\xi_0,q_0) \rangle^{0-} \langle (\xi_2,q_2) \rangle^{s}
\end{equation*}
and the constraints in \eqref{lastconstraints}, it follows - by undoing Plancherel and an application of Hölder's inequality - that
\begin{equation*}
I_f \lesssim N^{1-2s+} {\lVert} S_N^{(-\frac{5}{6},\frac{3}{4})}(J^su_1,J^su_3) {\rVert}_{L_{txy}^2} {\lVert} S_N^{(-\frac{5}{6},\frac{3}{4})}(J^{0-}f,J^{s}\tilde{u}_2) {\rVert}_{L_{txy}^2},
\end{equation*}
and it remains to appropriately estimate the two $L^2$-norms.
To this end, let $b > \frac{1}{2}$ be given arbitrarily. Then, by duality, and after applying Plancherel's theorem, followed by the Cauchy-Schwarz inequality and Fubini's theorem, we obtain
\begin{equation*}
{\lVert} S_N^{(-\frac{5}{6},\frac{3}{4})}(u,v) {\rVert}_{L_{txy}^2} \lesssim \brBigg{\sup_{(\tau_1,\xi_1,q_1) \in \mathbb{R} \times \mathbb{R} \times \mathbb{Z}} c_0^{\frac{1}{2}}} {\lVert} u {\rVert}_{X_{0,b}} {\lVert} v {\rVert}_{X_{0,b}}
\end{equation*}
with
\begin{align*}
c_0 &= \langle \tau_1 - \phi(\xi_1,q_1) \rangle^{-2b} \\ & \ \ \ \cdot \sum_{\substack{q \in \mathbb{Z} \\ {\lvert} 2q_1 - q {\rvert} \lesssim N^{\frac{3}{4}}}} \int_{\mathbb{R}} \chi_{\{{\lvert} \sqrt{3}\xi \mp q {\rvert} \lesssim N^{-\frac{5}{6}}\}} \int_{\mathbb{R}} \langle \tau - \tau_1 - \phi(\xi-\xi_1,q-q_1) \rangle^{-2b} \mathrm{d}\tau \mathrm{d}\xi,
\end{align*}
and since $b > \frac{1}{2}$, the integral over $\tau$ is finite. Consequently, we have
\begin{align*}
c_0 &\lesssim_b \sum_{\substack{q \in \mathbb{Z} \\ {\lvert} 2q_1 - q {\rvert} \lesssim N^{\frac{3}{4}}}} \int_{\mathbb{R}} \chi_{\{{\lvert} \sqrt{3}\xi \mp q {\rvert} \lesssim N^{-\frac{5}{6}}\}} \mathrm{d}\xi
\\ & \lesssim N^{-\frac{5}{6}} \sum_{\substack{q \in \mathbb{Z} \\ {\lvert} 2q_1 - q {\rvert} \lesssim N^{\frac{3}{4}}}} 1 \\ & \lesssim N^{\frac{3}{4}-\frac{5}{6}}
\\ & = N^{-\frac{1}{12}},
\end{align*}
and this ultimately establishes the estimate
\begin{equation} \label{S_Nestimate}
{\lVert} S_N^{(-\frac{5}{6},\frac{3}{4})}(u,v) {\rVert}_{L_{txy}^2} \lesssim_b N^{-\frac{1}{24}} {\lVert} u {\rVert}_{X_{0,b}} {\lVert} v {\rVert}_{X_{0,b}}, \quad b > \frac{1}{2}.
\end{equation}
Bilinear interpolation of \eqref{S_Nestimate} with the trivial bound
\begin{equation*}
{\lVert} S_N^{(-\frac{5}{6},\frac{3}{4})}(u,v) {\rVert}_{L_{txy}^2} \lesssim {\lVert} u {\rVert}_{X_{1+,\frac{1}{4}+}} {\lVert} v {\rVert}_{X_{0,\frac{1}{4}+}}
\end{equation*}
then gives
\begin{equation} \label{dual_S_Nestimate}
{\lVert} S_N^{(-\frac{5}{6},\frac{3}{4})}(u,v) {\rVert}_{L_{txy}^2} \lesssim N^{-\frac{1}{24}+} {\lVert} u {\rVert}_{X_{0+,\frac{1}{2}-}} {\lVert} v {\rVert}_{X_{0,\frac{1}{2}-}},
\end{equation}
thereby allowing us to resume the calculations above: An application of \eqref{S_Nestimate} to the first factor and of \eqref{dual_S_Nestimate} to the second factor leads us to
\begin{equation*}
I_f \lesssim N^{\frac{11}{12}-2s+} {\lVert} f {\rVert}_{X_{0,\frac{1}{2}-2\epsilon}} \prod_{i=1}^{3} {\lVert} u_i {\rVert}_{X_{s,\frac{1}{2}+\epsilon}},
\end{equation*}
and a dyadic summation over all $N \sim N_1 \sim N_2 \sim N_3 \gg 1$ finally gives the desired trilinear estimate \eqref{multilin modified ZK}, provided that
\begin{equation*}
\frac{11}{12} - 2s + < 0 \Leftrightarrow \frac{11}{24} + < s
\end{equation*}
is satisfied. By choosing $\epsilon > 0$ sufficiently small, any $s > \frac{11}{24}$ can be achieved and this concludes the discussion of this subcase. With this, all possible cases have been addressed, and the proof of Proposition \ref{mZK prop} is complete.
\end{proof}

\ew

\bibliographystyle{amsplain}

\bibliography{references}

\end{document}